\documentclass[thmsa,notitlepage,UKenglish,12pt,letterpaper]{article}%
\usepackage{amssymb}
\usepackage{amsfonts}
\usepackage{amsmath}
\usepackage{dsfont}
\usepackage{harvard}
\usepackage{geometry}
\usepackage[doublespacing]{setspace}
\usepackage{graphicx}%
\providecommand{\U}[1]{\protect\rule{.1in}{.1in}}
\newtheorem{theorem}{Theorem}

\newtheorem{Lemma}{Lemma}[section]

\newtheorem{Proposition}{Proposition}[section]

\newenvironment{proof}[1][Proof]{\noindent \textbf{#1.} }{\  \rule{0.5em}{0.5em}}
\begin{document}

\title{\textbf{{\Large The Cross-Quantilogram: Measuring Quantile Dependence and
Testing Directional Predictability between Time Series}}\thanks{We thank a
Co-Editor, Jianqing Fan, an Associate Editor and three anonymous referees for
constructive comments. Han's work was supported by the National Research
Foundation of Korea (NRF-2013S1A5A8021502). Linton's work was supported by
Cambridge INET and the ERC. Oka's work was supported by Singapore Academic
Research Fund (FY2013-FRC2-003). Whang's work was supported by the SNU
Creative Leading Researcher Grant. } }
\author{Heejoon Han\thanks{Department of Economics, Sungkyunkwan University, Seoul,
Republic of Korea.} \hspace{1cm} Oliver Linton\thanks{Faculty of Economics,
University of Cambridge, Cambridge, UK.} \hspace{1cm} Tatsushi
Oka\thanks{Department of Economics, National University of Singapore,
Singapore.} \hspace{1cm} Yoon-Jae Whang\thanks{Department of Economics, Seoul
National University, Seoul, Republic of Korea.} }
\date{March 14, 2016}
\maketitle

\begin{abstract}
This paper proposes the cross-quantilogram to measure the quantile dependence
between two time series. We apply it to test the hypothesis that one time
series has no directional predictability to another time series. We establish
the asymptotic distribution of the cross-quantilogram and the corresponding
test statistic. The limiting distributions depend on nuisance parameters. To
construct consistent confidence intervals we employ a stationary bootstrap
procedure; we establish consistency of this bootstrap. Also, we consider a
self-normalized approach, which yields an asymptotically pivotal statistic
under the null hypothesis of no predictability. We provide simulation studies
and two empirical applications. First, we use the cross-quantilogram to detect
predictability from stock variance to excess stock return. Compared to
existing tools used in the literature of stock return predictability, our
method provides a more complete relationship between a predictor and stock
return. Second, we investigate the systemic risk of individual financial
institutions, such as JP Morgan Chase, Morgan Stanley and AIG.\ 

\end{abstract}

\newpage
%EndExpansion
\setcounter{page}{1}\doublespacing \ 

\section{Introduction}

Linton and Whang (2007) introduced the quantilogram to measure predictability
in different parts of the distribution of a stationary time series based on
the correlogram of \textquotedblleft quantile hits\textquotedblright. They
applied it to test the hypothesis that a given time series has no directional
predictability. More specifically, their null hypothesis was that the past
information set of the stationary time series $\{y_{t}\}$ does not improve the
prediction about whether $y_{t}$ will be above or below the unconditional
quantile. The test is based on comparing the quantilogram to a pointwise
confidence band. This contribution fits into a long literature of testing
predictability using signs or rank statistics, including the papers of Cowles
and Jones (1937), Dufour et al. (1998), and Christoffersen and Diebold (2002).
The quantilogram has several advantages compared to other test statistics for
directional predictability. It is conceptually appealing and simple to
interpret. Since the method is based on quantile hits it does not require
moment conditions like the ordinary correlogram and statistics like the
variance ratio that are derived from it, Mikosch and Starica (2000), and so it
works well for heavy tailed series. Many financial time series have heavy
tails, see, e.g., Mandelbrot (1963), Fama (1965), Rachev and Mittnik (2000),
Embrechts et al. (1997), Ibragimov et al. (2009), and Ibragimov (2009), and so
this is an important consideration in practice. Additionally, this type of
method allows researchers to consider very long lags in comparison with
regression type methods, such as Engle and Manganelli (2004).

%% Literature
There have been a number of recent works either extending or applying this
methodology. Davis and Mikosch (2009) have introduced the extremogram, which
is essentially the quantilogram for extreme quantiles, and Davis et al. (2012)
has provided the inference methods based on bootstrap and permutation for the
extremogram. See also Davis et al. (2013). Li (2008, 2012) has introduced a
Fourier domain version of the quantilogram while Hong (2000) has used a
Fourier domain approach for test statistics based on distributions. Further
development in the Fourier domain approach has been made by Hagemann (2013)
and Dette et al. (2015). See also Li (2014) and Kley et al. (2016). The
quantilogram has recently been applied to stock returns and exchange rates,
Laurini et al. (2008) and Chang and Shie (2011).

Our paper addresses three outstanding issues with regard to the quantilogram.
First, the construction of confidence intervals that are valid under general
dependence structures. Linton and Whang (2007) derived the limiting
distribution of the sample quantilogram under the null hypothesis that the
quantilogram itself is zero, in fact under a special case of that where the
process has a type of conditional heteroskedasticity structure. Even in that
very special case, the limiting distribution depends on model specific
quantities. They derived a bound on the asymptotic variance that allows one to
test the null hypothesis of the absence of predictability (or rather the
special case of this that they work with). Even when this model structure is
appropriate, the bounds can be quite large especially when one looks into the
tails of the distribution. The quantilogram is also useful in cases where the
null hypothesis of no predictability is not thought to be true - one can be
interested in measuring the degree of predictability of a series across
different quantiles. We provide a more complete solution to the issue of
inference for the quantilogram. Specifically, we derive the asymptotic
distribution of the quantilogram under general weak dependence conditions,
specifically strong mixing. The limiting distribution is quite complicated and
depends on the long run variance of the quantile hits. To conduct inference we
propose the stationary bootstrap method of Politis and Romano (1994) and prove
that it provides asymptotically valid confidence intervals. We investigate the
finite sample performance of this procedure and show that it works well. We
also provide \texttt{R} code that carries out the computations
efficiently.\footnote{This can be found at
http://www.oliverlinton.me.uk/research/software.} We also define a
self-normalized version of the statistic for testing the null hypothesis that
the quantilogram is zero, following Lobato (2001). This statistic has an
asymptotically pivotal distribution, under the null hypothesis, whose critical
values have been tabulated so that there is no need for long run variance
estimation or even bootstrap.

Second, we develop our methodology inside a multivariate setting and
explicitly consider the cross-quantilogram. Linton and Whang (2007) briefly
mentioned such a multivariate version of the quantilogram but they provided
neither theoretical results nor empirical results. In fact, the
cross-correlogram is a vitally important measure of dependence between time
series: Campbell, Lo, and MacKinlay (1997), for example, use the cross
autocorrelation function to describe lead lag relations between large stocks
and small stocks. We apply the cross-quantilogram to the study of stock return
predictability; our method provides a more complete picture of the
predictability structure. We also apply the cross-quantilogram to the question
of systemic risk. Our theoretical results described in the previous paragraph
are all derived for the multivariate case.

Third, we explicitly allow the cross-quantilogram to be based on conditional
(or regression) quantiles (Koenker and Basset, 1978). Using conditional
quantiles rather than unconditional quantiles, we measure directional
dependence between two time-series after parsimoniously controlling for the
information at the time of prediction.\footnote{Our analysis includes the
cross-quantilogram based on unconditional quantiles as a special case. In this
case, the cross-quantilogram is shown to be a functional of the empirical
copula introduced by Ruschendorf (1976) and Deheuvels (1979) as some
nonparametric measures of dependence, such as Spearman's rho and Kendall's
tau. In this special case, the asymptotic results for the empirical copula,
which are found in Stute (1984), Fermanian et al. (2004) and Segers (2012)
among others, can apply for the cross-quantilogram. Generally, however, the
cross-quantilogram here differs from the empirical copula process and needs
different treatment for analyzing its properties.} Moreover, we derive the
asymptotic distribution of the cross-quantilogram that are valid uniformly
over a range of quantiles.

The remainder of the paper is as follows: Section 2 introduces the
cross-quantilogram and Section 3 discusses its asymptotic properties. For
consistent confidence intervals and hypothesis tests, we define the bootstrap
procedure and introduce the self normalized test statistic. Section 4
considers the partial cross-quantilogram and gives a full treatment of its
behavior in large samples. In Section 5 we report results of some Monte Carlo
simulations to evaluate the finite sample properties of our procedures. In
Section 6 we give two applications: we investigate stock return predictability
and system risk using our methodology. Appendix contains all the proofs.

%% Notations
We use the following notation: The norm $\Vert\cdot\Vert$ denotes the
Euclidean norm, i.e., $\Vert z\Vert=(\sum_{j=1}^{d}z_{j}^{2})^{1/2}$ for
$z=(z_{1},\dots,z_{d})^{\top}\in\mathds{R}^{d}$ and the norm $\Vert\cdot
\Vert_{p}$ indicates the $L^{p}$ norm of a $d\times1$ random vector $z$, given
by $\Vert z\Vert_{p}=(\sum_{j=1}^{d}E|z_{j}|^{p})^{1/p}$ for $p>0$. Let
$1[\cdot]$ be the indicator function taking the value one when its argument is
true, and zero otherwise. We use $\mathds{R}$, $\mathds{Z}$ and $\mathds{N}$
to denote the set of all real numbers, all integers and all positive integers,
respectively. Let $\mathds{Z}_{+}=\mathds{N}\cup\{0\}$.

\section{The Cross-Quantilogram}

Let $\{(\mathbf{y}_{t},\mathbf{x}_{t}):t\in\mathds{Z}\}$ be a strictly
stationary time series with $\mathbf{y}_{t}=(y_{1t},y_{2t})^{\top}%
\in\mathds{R}^{2}$ and $\mathbf{x}_{t}=(x_{1t},x_{2t})\in\mathds{R}^{d_{1}%
}\times\mathds{R}^{d_{2}}$, where $x_{it}=[x_{it}^{(1)},\dots,x_{it}^{(d_{i}%
)}]^{\top}\in\mathds{R}^{d_{i}}$ with $d_{i}\in\mathds{N}$ for $i=1,2$. We use
$F_{y_{i}|x_{i}}(\cdot|x_{it})$ to denote the conditional distribution
function of the series $y_{it}$ given $x_{it}$ with density function
$f_{y_{i}|x_{i}}(\cdot|x_{it})$, and the corresponding conditional quantile
function is defined as $q_{i,t}(\tau_{i})=\inf\{v:F_{y_{i}|x_{i}}%
(v|x_{it})\geq\tau_{i}\}$ for $\tau_{i}\in(0,1),$ for $i=1,2$. Let
$\mathcal{T}$ be the range of quantiles we are interested in evaluating the
directional predictability. For simplicity, we assume that $\mathcal{T}$ is a
Cartesian product of\ two closed intervals in $(0,1)$, that is $\mathcal{T}%
\equiv\mathcal{T}_{1}\times\mathcal{T}_{2},\ $where $\mathcal{T}%
_{i}=[\underline{\tau}_{i},\overline{\tau}_{i}]$ for some $0<\underline{\tau
}_{i}<\overline{\tau}_{i}<1$.\footnote{It is straightforward to extend the
results to a more general case, e.g. the case for which $\mathcal{T}$ is the
union of a finite number of disjoint closed subsets of $(0,1)^{2}$.}

We consider a measure of serial dependence between two events $\{y_{1t}\leq
q_{1,t}(\tau_{1})\}$ and $\{y_{2,t-k}\leq q_{2,t-k}(\tau_{2})\}$ for an
arbitrary pair of $\tau=(\tau_{1},\tau_{2})^{\top}\in\mathcal{T}$ and for an
integer $k$. In the literature, $\{1[y_{it}\leq q_{i,t}(\cdot)]\}$ is called
the quantile-hit or quantile-exceedance process for $i=1,2$. The
cross-quantilogram is defined as the cross-correlation of the quantile-hit
processes
\begin{equation}
\rho_{\tau}(k)=\frac{E\left[  \psi_{\tau_{1}}(y_{1t}-q_{1,t}(\tau_{1}%
))\psi_{\tau_{2}}(y_{2,t-k}-q_{2,t-k}(\tau_{2}))\right]  }{\sqrt{E\left[
\psi_{\tau_{1}}^{2}(y_{1t}-q_{1,t}(\tau_{1}))\right]  }\sqrt{E\left[
\psi_{\tau_{2}}^{2}(y_{2,t-k}-q_{2,t-k}(\tau_{2}))\right]  }}, \label{q1}%
\end{equation}
for $k=0,\pm1,\pm2,\dots,$ where $\psi_{a}(u)\equiv1[u<0]-a$. The
cross-quantilogram captures serial dependence between the two series at
different conditional quantile levels. In the special case of a single time
series, the cross-quantilogram becomes the quantilogram proposed by Linton and
Whang (2007). Note that it is well-defined even for processes $\{(y_{1t}%
,y_{2t})\}_{t\in\mathds{N}}$ with infinite moments. Like the quantilogram, the
cross-quantilogram is invariant to any strictly monotonic transformation
applied to both series, such as the logarithmic transformation.\footnote{When
one is interested in measuring serial dependence between two events
$\{q_{1,t}(\tau_{1}^{l})\leq y_{1t}\leq q_{1,t}(\tau_{1}^{h})\}$ and
$\{q_{2,t-k}(\tau_{2}^{l})\leq y_{2,t-k}\leq q_{2,t-k}(\tau_{2}^{h})\}$ for
arbitrary $\left[  \tau_{1}^{l},\tau_{1}^{h}\right]  $ and $\left[  \tau
_{2}^{l},\tau_{2}^{h}\right]  $, one can use an alternative version of the
cross-quantilogram that is defined by replacing $\psi_{\tau_{i}}%
(y_{it}-q_{i,t}(\tau_{i}))$ in (\ref{q1}) with%
\[
\psi_{\left[  \tau_{i}^{l},\tau_{i}^{h}\right]  }(y_{it}-q_{i,t}(\left[
\tau_{i}^{l},\tau_{i}^{h}\right]  ))=1[q_{i,t}(\tau_{i}^{l})<y_{it}%
<q_{i,t}\left(  \tau_{i}^{h}\right)  ]-\left(  \tau_{i}^{h}-\tau_{i}%
^{l}\right)  .
\]
For example, if $\tau_{1}=\left[  0.9,1.0\right]  $ and $\tau_{2}=\left[
0.4,0.6\right]  ,$ the alternative version measures dependence between an
event that $y_{1t}$ is in a high range and an event that $y_{2,t-k}$ is in a
mid-range. In some cases, such an alternative version could be easier to
interpret and therefore be useful. The inference procedure provided in this
paper is also valid for the alternative version of the cross-quantilogram. See
the working paper version of this paper for an empirical application using the
alternative version.}

To construct the sample analogue of the cross-quantilogram based on
observations $\{(\mathbf{y}_{t},\mathbf{x}_{t})\}_{t=1}^{T}$, we first
estimate conditional quantile functions. In this paper, we consider the linear
quantile regression model proposed by Koenker and Bassett (1978) for
simplicity and let $q_{i,t}(\tau_{i})=x_{it}^{\top}\beta_{i}(\tau_{i})$ with a
$d_{i}\times1$ vector of unknown parameters $\beta_{i}(\tau_{i})$ for $i=1,2$.
To estimate the parameters $\beta(\tau)\equiv\lbrack\beta_{1}(\tau_{1})^{\top
},\beta_{2}(\tau_{2})^{\top}]^{\top}$, we separately solve the following
minimization problems:
\[
\hat{\beta}_{i}(\tau_{i})=\arg\min_{\beta_{i}\in\mathds{R}^{d_{i}}}\sum
_{t=1}^{T}\varrho_{\tau_{i}}\left(  y_{it}-x_{it}^{\top}\beta_{i}\right)  ,
\]
where $\varrho_{a}(u)\equiv u(a-1[u<0])$. Let $\hat{\beta}(\tau)\equiv
\lbrack\hat{\beta}_{1}(\tau_{1})^{\top},\hat{\beta}_{2}(\tau_{2})^{\top
}]^{\top}$ and $\hat{q}_{i,t}(\tau_{i})=x_{it}^{\top}\hat{\beta}_{i}(\tau
_{i})$ for i = 1,2. The sample cross-quantilogram is defined by
\begin{equation}
\hat{\rho}_{\tau}(k)=\frac{\sum_{t=k+1}^{T}\psi_{\tau_{1}}(y_{1t}-\hat
{q}_{1,t}(\tau_{1}))\psi_{\tau_{2}}(y_{2,t-k}-\hat{q}_{2,t-k}(\tau_{2}%
))}{\sqrt{\sum_{t=k+1}^{T}\psi_{\tau_{1}}^{2}(y_{1t}-\hat{q}_{1,t}(\tau_{1}%
))}\sqrt{\sum_{t=k+1}^{T}\psi_{\tau_{2}}^{2}(y_{2,t-k}-\hat{q}_{2,t-k}%
(\tau_{2}))}}, \label{q2}%
\end{equation}
for $k=0,\pm1,\pm2,\dots$. Given a set of conditional quantiles, the
cross-quantilogram considers dependence in terms of the direction of deviation
from conditional quantiles and thus measures the directional predictability
from one series to another. This can be\ a useful descriptive device. By
construction, $\hat{\rho}_{\tau}(k)\in\lbrack-1,1]$ with $\hat{\rho}_{\tau
}(k)=0$ corresponding to the case of no directional predictability. The form
of the statistic generalizes to the $l$ dimensional multivariate case and the
$(i,j)$th entry of the corresponding cross-correlation matrices $\Gamma
_{\bar{\tau}}(k)$ is given by applying (\ref{q2}) for a pair of variables
$(y_{it},x_{it})$ and $(y_{jt-k},x_{jt-k})$ and a pair of conditional
quantiles $(\hat{q}_{i,t}(\tau_{i}),\hat{q}_{j,t-k}(\tau_{j})))$ for
$\bar{\tau}=(\tau_{1},\dots,\tau_{l})^{^{\top}}$. The cross-correlation
matrices possess the usual symmetry property $\Gamma_{\bar{\tau}}%
(k)=\Gamma_{\bar{\tau}}(-k)^{^{\top}}$ when $\tau_{1}=\cdots=\tau_{d}.$

Suppose that $\tau\in\mathcal{T\ }$\ and $p$\ are given. One may be interested
in testing the null hypothesis $H_{0}:\rho_{\tau}(1)=\dots=\rho_{\tau}(p)=0\ $
against the alternative hypothesis that $\rho_{\tau}(k)\not =0$ for some
$k\in\{1,\dots,p\}$. This is a test for the directional predictability of
events up to $p$ lags $\{y_{2,t-k}\leq q_{2,t-k}(\tau_{2}):k=1,\dots,p\}$ for
$\{y_{1t}\leq q_{1,t}(\tau_{1})\}.$ For this hypothesis, we can use the
Box-Pierce type statistic $\hat{Q}_{\tau}^{(p)}=T\sum_{k=1}^{p}\hat{\rho
}_{\tau}^{2}(k)$. In practice, we recommend to use the Box-Ljung version
$\check{Q}_{\tau}^{(p)}\equiv T(T+2)\sum_{k=1}^{p}\hat{\rho}_{\tau}%
^{2}(k)/(T-k)$ which had small sample improvements in our simulations.

On the other hand, one may be interested in testing a stronger null
hypothesis, i.e. the absence of directional predictability over a set of
quantiles: $H_{0}:\rho_{\tau}(1)=\dots=\rho_{\tau}(p)=0,$ $\forall\tau
\in\mathcal{T},$ against the alternative hypothesis that $\rho_{\tau}%
(k)\not =0$ for some $(k,\tau)\in\{1,\dots,p\}\times\mathcal{T}$ with $p$
fixed. In this case, we can use the sup-version test statistic
\[
\sup_{\tau\in\mathcal{T}}\hat{Q}_{\tau}^{(p)}=\sup_{\tau\in\mathcal{T}}%
T\sum_{k=1}^{p}\hat{\rho}_{\tau}^{2}(k).
\]
Note that the portmanteau test statistic $\hat{Q}_{\tau}^{(p)}$ for a specific
quantile is a special case of the sup-version test statistic.

\section{Asymptotic Properties}

We next present the asymptotic properties of the sample cross-quantilogram and
related test statistics. Since these quantities contain non-smooth functions,
we employ techniques widely used in the literature on quantile regression, see
Koenker and Bassett (1978) and Pollard (1991) among others.

Define $\mathbf{y}_{t,k}=(y_{1t},y_{2,t-k})^{\top}$, $\mathbf{x}_{t,k}%
=(x_{1t},x_{2,t-k})$, $\mathbf{q}_{t,k}(\tau)=[q_{1,t}(\tau_{1}),q_{2,t-k}%
(\tau_{2})]^{\top}$ and $\hat{\mathbf{q}}_{t,k}(\tau)=[\hat{q}_{1,t}(\tau
_{1}),\hat{q}_{2,t-k}(\tau_{2})]^{\top}$ and let $\{\mathbf{y}_{t,k}%
\leq\mathbf{q}_{t,k}(\tau)\}=\{y_{1t}\leq q_{1}(\tau_{1}|x_{1t}),y_{2,t-k}\leq
q_{2}(\tau_{2}|x_{2t-k})\}$ and $F_{\mathbf{y}|\mathbf{x}}^{(k)}%
(\cdot|\mathbf{x}_{t,k})=P(\mathbf{y}_{t,k}\leq\cdot|\mathbf{x}_{t,k})$ for
$t=k+1,\dots,T$ and for some finite integer $k>0$. We use $\nabla G^{(k)}%
(\tau)$ to denote $\partial/\partial\mathbf{v}E[F_{\mathbf{y}|\mathbf{x}%
}^{(k)}(\mathbf{v}_{t,k}|\mathbf{x}_{t,k})]$ evaluated at $\mathbf{v}%
_{t,k}=\mathbf{q}_{t,k}(\tau)$, where $\mathbf{v}_{t,k}=[x_{1t}^{\top}%
v_{1},x_{2,t-k}^{\top}v_{2}]^{\top}$ for $v_{i}\in\mathds{R}^{d_{i}}$
($i=1,2$). Let $d_{0}=1+d_{1}+d_{2}$.

\vspace{0.5cm} \noindent\textbf{Assumption}

\begin{description}
\item[A1.] $\{(\mathbf{y}_{t},\mathbf{x}_{t})\}_{t\in\mathds{Z}}$ is strictly
stationary and strong mixing with coefficients $\{\alpha_{j}\}_{j\in
\mathds{Z}_{+}}$ that satisfy $\sum_{j=0}^{\infty}(j+1)^{2s-2}\alpha_{j}%
^{\nu/(2s+\nu)}<\infty$ for some integer $s\geq3$ and $\nu\in(0,1)$. For each
$i=1,2$, $E|x_{it}^{(j)}|^{2s+\nu}<\infty$ for all $j=1,\dots,d_{i}$, given
$x_{it}=[x_{it}^{(1)},\dots,x_{it}^{(d_{i})}]^{\top}$.

\item[A2.] The conditional distribution function $F_{y_{i}|x_{i}}(\cdot
|x_{it})$ has continuous densities $f_{y_{i}|x_{i}}(\cdot|x_{it})$, which is
uniformly bounded away from 0 and $\infty$ at $q_{i,t}(\tau_{i})$ uniformly
over $\tau_{i}\in\mathcal{T}_{i}$, for $i=1,2$ and for all $t\in\mathds{Z}$.

\item[A3.] For any $\epsilon>0$ there exists a $\nu(\epsilon)$ such that
$\sup_{\tau_{i}\in\mathcal{T}_{i}}\sup_{s:|s|\leq\nu(\epsilon)}|f_{y_{i}%
|x_{i}}(q_{i,t}(\tau_{i})|x_{it})-f_{y_{i}|x_{i}}(q_{i,t}(\tau_{i}%
)+s|x_{it})|<\epsilon$ for $i=1,2$ and for all $t\in\mathds{Z}.$

\item[A4.] For every $k\in\{1,\dots,p\}$, the conditional joint distribution
$F_{\mathbf{y}|\mathbf{x}}^{(k)}(\cdot|\mathbf{x}_{t,k})$ has the conditional
density $f_{\mathbf{y}|\mathbf{x}}^{(k)}(\cdot|\mathbf{x_{t,k}})$, which is
bounded uniformly in the neighborhood of quantiles of interest, and also has a
bounded, continuous first derivative for each argument uniformly in the
neighborhood of quantiles of interest and thus $\nabla G^{(k)}(\tau)$ exists
over $\tau\in\mathcal{T}$.

\item[A5.] For each $i=1,2$, there exist positive definite matrices $M_{i}$
and $D_{i}(\tau_{i})$ such that (a) plim$_{T\rightarrow\infty}T^{-1}\sum
_{t=1}^{T}x_{it}x_{it}^{\top}=M_{i}$ and (b) plim$_{T\rightarrow\infty}%
T^{-1}\sum_{t=1}^{T}f_{y_{i}|x_{i}}(q_{i,t}(\tau_{i})|x_{it})x_{it}%
x_{it}^{\top}=D_{i}(\tau_{i})$ uniformly in $\tau_{i}\in\mathcal{T}_{i}$.
\end{description}

Assumption A1 imposes the mixing rate used in Andrews and Pollard (1994) and a
moment condition on regressors, while allowing for the dependent variables to
be processes with infinite moments. For a strong mixing process, $\rho_{\tau
}(k)\rightarrow0$ as $k\rightarrow\infty$ for all $\tau\in(0,1).$ Assumption
A2 ensures that the conditional quantile function given $x_{it}$ is uniquely
defined while allowing for dynamic misspecification, or $P(y_{it}\leq
q_{i,t}(\tau_{i})|\mathcal{F}_{it})\not =\tau_{i}$ given some information set
$\mathcal{F}_{it}$ containing all \textquotedblleft relevant\textquotedblright%
\ information available at $t$ for $i=1,2$. In the absence of dynamic
misspecification, which is assumed in Hong et al. (2009) under their null
hypothesis, the analysis becomes substantially simple because each hit-process
$\{\psi_{\tau_{i}}(y-q_{i,t}(\tau_{i}))\}$ is a sequence of iid Bernoulli
random variables. As Corradi and Swanson (2006) discuss, however, results
under correct dynamic specification crucially rely on an appropriate choice of
the information set; specification search for the information set based on
pre-testing may have a nontrivial impact on inference. Thus, Assumption A2 is
appropriate for the purpose of testing directional predictability given a
particular information set $x_{it}$. Assumption A3 implies that the densities
are smooth in some neighborhood of the quantiles of interest. Assumption A4
ensures that the joint distribution of $(x_{1t},x_{2t-k})$ is continuously
differentiable. Assumption A5 is standard in the quantile regression literature.

To describe the asymptotic behavior of the cross-quantilogram, we define a set
of $d_{0}$-dimensional mean-zero Gaussian process $\{\mathbb{B}_{k}(\tau
):\tau\in\lbrack0,1]^{2}\}_{k=1}^{p}$ with covariance-matrix function for
$k,k^{\prime}\in\{1,\dots,p\}$ and for $\tau,\tau^{\prime}\in\mathcal{T}$,
given by
\[
\Xi_{kk^{\prime}}(\tau,\tau^{\prime})\equiv E[\mathbb{B}_{k}(\tau
)\mathbb{B}_{k^{\prime}}^{^{\top}}(\tau^{\prime})]=\sum_{l=-\infty}^{\infty
}\mathrm{cov}\left(  \xi_{l,k}(\tau),\xi_{0,k^{\prime}}^{^{\top}}(\tau
^{\prime})\right)  ,
\]
where $\xi_{t,k}(\tau)=(1[\mathbf{y_{t,k}\leq q_{t,k}(\tau)}],x_{1t}^{\top
}1[y_{1t}\leq q_{1,t}(\tau_{1})],x_{2t}^{\top}1[y_{2t}\leq q_{2,t}(\tau
_{2})])^{\top}$ for $t\in\mathds{Z}$. Define $\mathbb{B}^{(p)}(\tau
)=[\mathbb{B}_{1}(\tau)^{^{\top}},\dots,\mathbb{B}_{p}(\tau)^{^{\top}%
}]^{^{\top}}$ as the $d_{0}p$-dimensional zero-mean Gaussian process with the
covariance-matrix function denoted by $\Xi^{(p)}(\tau,\tau^{\prime})$ for
$\tau,\tau^{\prime}\in\mathcal{T}$. We use $\ell^{\infty}(\mathcal{T})$ to
denote the space of all bounded functions on $\mathcal{T}$ equipped with the
uniform topology and $(\ell^{\infty}(\mathcal{T}))^{p}$ to denote the
$p$-product space of $\ell^{\infty}(\mathcal{T})$ equipped with the product
topology. Let the notation \textquotedblleft$\Rightarrow$\textquotedblright%
\ denote the weak convergence due to Hoffman-Jorgensen in order to handle the
measurability issues, although outer probabilities and expectations are not
used explicitly in this paper for notational simplicity. See Chapter 1 of van
der Vaart and Wellner (1996) for a comprehensive treatment of weak convergence
in non-separable metric spaces.

The next theorem establishes the asymptotic properties of the cross-quantilogram.

\begin{theorem}
\label{theorem:lim-p} Suppose that Assumptions A1-A5 hold for some finite
integer $p>0.$ Then, in the sense of weak convergence of the stochastic
process in $(\ell^{\infty}(\mathcal{T}) )^{p}$ we have:
\begin{equation}
\noindent\sqrt{T}\left(  \hat{\rho}_{\tau}^{(p)}-\rho_{\tau}^{(p)}\right)
\Rightarrow\Lambda_{\tau}^{(p)}\mathbb{B}^{(p)}(\tau), \label{th1}%
\end{equation}
where $\hat{\rho}_{\tau}^{(p)}\equiv\lbrack\hat{\rho}_{\tau}(1),\dots
,\hat{\rho}_{\tau}(p)]^{^{\top}}$ and $\Lambda_{\tau}^{(p)}=\mathrm{diag}%
(\lambda_{\tau1}^{^{\top}},\dots,\lambda_{\tau p}^{^{\top}})$ with
\begin{equation}
\lambda_{\tau, k}=\frac{1}{\sqrt{\tau_{1}(1-\tau_{1})\tau_{2}(1-\tau_{2})}%
}\left[
\begin{array}
[c]{c}%
1\\
-\nabla G^{(k)}(\tau)[D_{1}^{-1}(\tau_{1}),D_{2}^{-1}(\tau_{2})]^{\top}%
\end{array}
\right]  . \label{eq:lambda-def}%
\end{equation}

\end{theorem}

\vspace{0.5cm}

Under the null hypothesis that $\rho_{\tau}(1)=\cdots=\rho_{\tau}(p)=0$ for
every $\tau\in\mathcal{T}$ , it follows that
\begin{equation}
\sup_{\tau\in\mathcal{T}}\hat{Q}_{\tau}^{(p)}\Rightarrow\sup_{\tau
\in\mathcal{T}}\Vert\Lambda_{\tau}^{(p)}\mathbb{B}^{(p)}(\tau)\Vert^{2},
\label{Qd}%
\end{equation}
by the continuous mapping theorem.

\subsection{Inference Methods}

\subsubsection{The Stationary Bootstrap}

The asymptotic null distribution presented in Theorem \ref{theorem:lim-p}
depends on nuisance parameters. We suggest to estimate the critical values by
the stationary bootstrap of Politis and Romano (1994). The stationary
bootstrap is a block bootstrap method with blocks of random lengths. The
stationary bootstrap resample is strictly stationary conditional on the
original sample.

Let $\{L_{i}\}_{i\in\mathds{N}}$ denote a sequence of iid random block lengths
having the geometric distribution with a scalar parameter $\gamma\equiv
\gamma_{T}\in(0,1)$: $P^{\ast}(L_{i}=l)=\gamma(1-\gamma)^{l-1}$ for each
positive integer $l$, where $P^{\ast}$ denotes the conditional probability
given the original sample. We assume that the parameter $\gamma$ satisfies the
following growth condition:

\vspace{0.5cm} \noindent\textbf{Assumption A6.} $T^{\nu/2(2s+\nu)(s-1)}%
\gamma+(\sqrt{T}\gamma)^{-1}\rightarrow0 $ as $T\rightarrow\infty$, where $s$
and $\nu$ are defined in Assumption A1.\vspace{0.5cm}

We need the condition that $\gamma=o(T^{-\nu/2(2s+\nu)(s-1) })$ for the
purpose of establishing uniform convergence over the subset $\mathcal{T}$ of
$\left[  0,1\right]  ^{2},$ given the moment conditions on regressors under
Assumption A1. This condition can be relaxed when regressors are uniformly
bounded because $\gamma= o(1)$ when $s = \infty$.

Let $\{K_{i}\}_{i\in\mathds{N}}$ be a sequence of iid random variables, which
have the discrete uniform distribution on $\{k+1,\dots,T\}$ and are
independent of both the original data and $\{L_{i}\}_{i\in\mathds{N}}$. We set
$B_{K_{i},L_{i}}=\{(\mathbf{y}_{t,k},\mathbf{x}_{t,k})\}_{t=K_{i}}%
^{K_{i}+L_{i}-1}$ representing the blocks of length $L_{i}$ starting with the
$K_{i}$-th pair of observations. The stationary bootstrap procedure generates
the bootstrap samples $\{(\mathbf{y}_{t,k}^{\ast},\mathbf{x}_{t,k}^{\ast
})\}_{t=k+1}^{T}$ by taking the first $(T-k)$ observations from a sequence of
the resampled blocks $\{B_{K_{i},L_{i}}\}_{i\in\mathds{N}}$. In this notation,
when $t>T$, $(\mathbf{y}_{t,k},\mathbf{x}_{t,k})$ is set to be $(\mathbf{y}%
_{jk},\mathbf{x}_{jk})$, where $j=k+(t\ \mathrm{mod}\ (T-k))$ and
$(\mathbf{y}_{k,k},\mathbf{x}_{k,k})=(\mathbf{y}_{t,k},\mathbf{x}_{t,k})$,
where mod denotes the modulo operator.\footnote{For any positive integers $a$
and $b$, the modulo operation $a\ \mathrm{mod}\ b$ is equal to the remainder,
on division of $a$ by $b$.}

Using the stationary bootstrap resample, we estimate the parameter $\beta
(\tau)$ by solving the minimization problem:
\[
\hat{\beta}_{1}^{\ast}(\tau_{1})=\arg\min_{\beta_{1}\in\mathds{R}^{d_{1}}}%
\sum_{t=k+1}^{T}\varrho_{\tau_{1}}(y_{1t}^{\ast}-x_{1t}^{\ast\top}\beta
_{1})\ \ \mathrm{and}\ \ \hat{\beta}_{2}^{\ast}(\tau_{2})=\arg\min_{\beta
_{2}\in\mathds{R}^{d_{2}}}\sum_{t=1}^{T-k}\varrho_{\tau_{2}}(y_{2t}^{\ast
}-x_{2t}^{\ast\top}\beta_{2}).
\]
Then the conditional quantile function given the stationary bootstrap
resample, $q_{i,t}^{\ast}(\tau_{i})\equiv x_{it}^{\ast\top}\beta_{i}(\tau
_{i})$, is estimated by $\hat{q}_{i,t}^{\ast}(\tau_{i})\equiv x_{it}^{\ast
\top}\hat{\beta}_{i}^{\ast}(\tau_{i})$ for each $i=1,2$. Define $\hat{\beta
}^{\ast}(\tau)=[\hat{\beta}_{1}^{\ast\top}(\tau_{1}),\hat{\beta}_{2}^{\ast
\top}(\tau_{2})]^{\top}$ and let $\hat{\mathbf{q}}_{t,k}^{\ast}(\tau)=[\hat
{q}_{1,t}^{\ast}(\tau_{1}),\hat{q}_{2,t-k}^{\ast}(\tau_{2})]^{\top}$ and
$\mathbf{q}_{t,k}^{\ast}(\tau)=[q_{1,t}^{\ast}(\tau_{1}),q_{2,t-k}^{\ast}%
(\tau_{2})]^{\top}$. We construct $\hat{\beta}^{\ast}(\tau)$ by using $(T-k)$
bootstrap observations, while $\hat{\beta}(\tau)$ is based on $T$
observations, but the difference of sample sizes is asymptotically negligible
given the finite lag order $k$.

The cross-quantilogram based on the stationary bootstrap resample is defined
as follows:
\begin{align*}
\hat{\rho}_{\tau}^{\ast}(k) = \frac{ \sum_{t=k+1}^{T} \psi_{\tau_{1}}%
(y_{1t}^{\ast} - \hat{q}_{1,t}^{\ast}(\tau_{1}) ) \psi_{\tau_{2}}%
(y_{2,t-k}^{\ast} - \hat{q}_{2,t-k}^{\ast}(\tau_{2}) ) }{ \sqrt{ \sum
_{t=k+1}^{T} \psi_{\tau_{1}}^{2}(y_{1t}^{\ast} - \hat{q}_{1,t}^{\ast}(\tau
_{1})) } \sqrt{ \sum_{t=k+1}^{T} \psi_{\tau_{2}}^{2}(y_{2,t-k}^{\ast} -
\hat{q}_{2,t-k}^{\ast}(\tau_{2})) } }.
\end{align*}

We consider the stationary bootstrap to construct a confidence interval for
each statistic of $p$ cross-quantilograms $\{\hat{\rho}_{\tau}(1),\dots
,\hat{\rho}_{\tau}(p)\}$ for a finite positive integer $p$ and subsequently
construct a confidence interval for the omnibus test based on the $p$
statistics. To maintain the original dependence structure, we use $(T-p)$
pairs of observations $\{[(\mathbf{y}_{t,1},\mathbf{x}_{t,1}),\dots
,(\mathbf{y}_{t,p},\mathbf{x}_{t,p})]\}_{t=p+1}^{T}$ to resample the blocks of
random lengths.

Given a vector cross-quantilogram $\hat{\rho}_{\tau}^{(p)\ast}$, we define the
omnibus test based on the stationary bootstrap resample as $\hat{Q}_{\tau
}^{(p)\ast}=T(\hat{\rho}_{\tau}^{(p)\ast}-\hat{\rho}_{\tau}^{(p)})^{^{\top}%
}(\hat{\rho}_{\tau}^{(p)\ast}-\hat{\rho}_{\tau}^{(p)})$. The following theorem
shows the validity of the stationary bootstrap procedure for the
cross-quantilogram. We use the concept of weak convergence in probability
conditional on the original sample, which is denoted by \textquotedblleft%
$\Rightarrow^{\ast}$", see van der Vaart and Wellner (1996, p.~181).

\begin{theorem}
\label{theorem:boostrap validity} Suppose that Assumption A1-A6 hold. Then, in
the sense of weak convergence conditional on the sample we have:

\noindent(a) $\sqrt{T}\left(  \hat{\rho}_{\tau}^{(p)\ast}-\hat{\rho}_{\tau
}^{(p)}\right)  \Rightarrow^{\ast}\Lambda_{\tau}^{(p)}\mathbb{B}^{(p)}(\tau)$
\ \ in probability;

\noindent(b) Under the null hypothesis that $\rho_{\tau}(1)=\cdots=\rho_{\tau
}(p)=0$ for every $\tau\in\mathcal{T}$,
\[
\sup_{z\in\mathds{R}}\left\vert P^{\ast}\left(  \sup_{\tau\in\mathcal{T}}%
\hat{Q}_{\tau}^{(p)\ast}\leq z\right)  -P\left(  \sup_{\tau\in\mathcal{T}}%
\hat{Q}_{\tau}^{(p)}\leq z\right)  \right\vert \rightarrow^{p}0.
\]

\end{theorem}

In practice, repeating the stationary bootstrap procedure $B$ times, we obtain
$B$ sets of cross-quantilograms and $\{\hat{\rho}_{\tau,b}^{(p)\ast}%
=[\hat{\rho}_{\tau,b}^{\ast}(1),\dots,\hat{\rho}_{\tau,b}^{\ast}(p)]^{^{\top}%
}\}_{b=1}^{B}$ and $B$ sets of omnibus tests $\{\hat{Q}_{\tau,b}^{(p)\ast
}\}_{b=1}^{B}$ with $\hat{Q}_{\tau,b}^{(p)\ast}=T(\hat{\rho}_{\tau,b}%
^{(p)\ast}-\hat{\rho}_{\tau}^{(p)})^{^{\top}}(\hat{\rho}_{\tau,b}^{(p)\ast
}-\hat{\rho}_{\tau}^{(p)})$. For testing jointly the null of no directional
predictability, a critical value, $c_{Q,\alpha}^{\ast}$, corresponding to a
significance level $\alpha$ is given by the $(1-\alpha)100\%$ percentile of
$B$ test statistics $\{\sup_{\alpha\in\mathcal{T}}\hat{Q}_{\alpha,b}^{(p)\ast
}\}_{b=1}^{B}$, that is,
\[
c_{Q,\alpha}^{\ast}=\inf\left\{  c:P^{\ast}\left(  \sup_{\tau\in\mathcal{T}%
}\hat{Q}_{\tau,b}^{(p)\ast}\leq c\right)  \geq1-\alpha\right\}  .
\]
For the individual cross-quantilogram, we pick up percentiles $(c_{1k,\alpha
}^{\ast},c_{2k,\alpha}^{\ast})$ of the bootstrap distribution of $\{\sqrt
{T}(\hat{\rho}_{\tau,b}^{\ast}(k)-\hat{\rho}_{\tau}(k))\}_{b=1}^{B}$ such that
$P^{\ast}(c_{1k,\alpha}^{\ast}\leq\sqrt{T}(\hat{\rho}_{\tau,b}^{\ast}%
(k)-\hat{\rho}_{\tau}(k))\leq c_{2k,\alpha}^{\ast})=1-\alpha$, in order to
obtain a $100(1-\alpha)\%$ confidence interval for $\rho_{\tau}(k)$ given by
$[\hat{\rho}_{\tau}(k)+T^{-1/2}c_{1k,\alpha}^{\ast},\ \hat{\rho}_{\tau
}(k)+T^{-1/2}c_{2k,\alpha}^{\ast}].$

In the following theorem, we provide a power analysis of the omnibus test
statistic $\sup_{\tau\in\mathcal{T}}\hat{Q}_{\tau}^{(p)}$ when we use a
critical value $c_{Q,\alpha}^{\ast}$. We consider fixed and local
alternatives. The fixed alternative hypothesis against the null of no
directional predictability is
\begin{equation}
H_{1}:\rho_{\tau}(k)\mathrm{\ }\neq0\mathrm{\ for\ some\ }(\tau,k)\in
\mathcal{T}\ \times\{1,\dots,p\}, \label{eq:alt-fix}%
\end{equation}
and the local alternative hypothesis is given by%
\begin{equation}
H_{1T}:\rho_{\tau}(k)=\zeta/\sqrt{T}\ \ \mathrm{for\ some}\ (\tau
,k)\in\mathcal{T}\ \times\{1,\dots,p\}, \label{eq:alt-local}%
\end{equation}
where $\zeta$ is a finite non-zero constant. Thus, under the local
alternative, there exists a $p\times1$ vector $\zeta_{\tau}^{(p)}$ such that
$\rho_{\tau}^{(p)}=T^{-1/2}\zeta_{\tau}^{(p)}$ with $\zeta_{\tau}^{(p)}$
having at least one non-zero element for some $\tau\in\mathcal{T}$.

We consider the asymptotic power of a test for the directional predictability
over a range of quantiles with multiple lags in the following theorem;
however, the results can be applied to test for a specific quantile or a
specific lag order. The following theorem shows that the cross-quantilogram
process has non-trivial local power against the $\sqrt{T}$-local alternatives.

\begin{theorem}
\label{theorem:alternative} Suppose that Assumptions A1-A6 hold. Then: (a)
Under the fixed alternative in (\ref{eq:alt-fix}),%
\[
\lim_{T\rightarrow\infty}P\left(  \sup_{\tau\in\mathcal{T}}\hat{Q}_{\tau
}^{(p)}>c_{Q,\alpha}^{\ast}\right)  \rightarrow1.
\]
(b) Under the local alternative in (\ref{eq:alt-local})%
\[
\lim_{T\rightarrow\infty}P\left(  \sup_{\tau\in\mathcal{T}}\hat{Q}_{\tau
}^{(p)}>c_{Q,\alpha}^{\ast}\right)  =P\left(  \sup_{\tau\in\mathcal{T}}%
\Vert\Lambda_{\tau}^{(p)}\mathbb{B}^{(p)}(\tau)+\zeta_{\tau}^{(p)}\Vert
^{2}\geq c_{Q,\alpha}\right)  ,
\]
where $c_{Q,\alpha}=\inf\{c:P(\sup_{\tau\in\mathcal{T}}\Vert\Lambda_{\tau
}^{(p)}\mathbb{B}^{(p)}(\tau)\Vert^{2}\leq c))\geq1-\alpha\}$.
\end{theorem}

\subsubsection{The Self-Normalized Cross-Quantilogram}

We use recursive estimates to construct a self-normalized cross-quantilogram.
The self-normalized approach was proposed by Lobato (2001) and was recently
extended by Shao (2010) to a class of asymptotically linear test
statistics.\footnote{Kuan and Lee (2006) apply the approach to a class of
specification tests, the so-called $M$ tests, which are based on the moment
conditions involving unknown parameters. Chen and Qu (2015) propose a
procedure for improving the power of the $M$ test, by dividing the original
sample into subsamples before applying the self-normalization procedure.} The
self-normalized approach has a tight link with the fixed-$b$ asymptotic
framework proposed by Kiefer et al. (2000).\footnote{The fixed-$b$ asymptotic
has been further studied by Bunzel et al. (2001), Kiefer and Vogelsang (2002,
2005), Sun et al. (2008), Kim and Sun (2011) and Sun and Kim (2012) among
others.} The self-normalized statistic has an asymptotically pivotal
distribution whose critical values have been tabulated so that there is no
need for long run variance estimation or even bootstrap. As discussed in
section 2.1 of Shao (2010), the self-normalized and the fixed-$b$ approach
have better size properties, compared with the standard approach involving a
consistent asymptotic variance estimator, while it may be asymptotically less
powerful under local alternatives (see Lobato (2001) and Sun et al. (2008) for instance).

Given a subsample $\{(\mathbf{y}_{t},\mathbf{x}_{t})\}_{t=1}^{s}$, we can
estimate sample quantile functions by solving minimization problems
\[
\hat{\beta}_{i,s}(\tau_{i})=\arg\min_{\beta_{i}\in\mathds{R}^{d_{i}}}%
\sum_{t=1}^{s}\varrho_{\tau_{i}}\left(  y_{it}-x_{it}^{\top}\beta_{i}\right)
,
\]
for $i=1,2$. Let $\hat{q}_{i,t,s}(\tau_{i})=x_{it}^{\top}\hat{\beta}%
_{i,s}(\tau_{i})$. We consider the minimum subsample size $s$ larger than
$[T\omega]$, where $\omega\in(0,1)$ is an arbitrary small positive constant.
The trimming parameter, $\omega$, is necessary to guarantee that the quantiles
estimators based on subsamples have standard asymptotic properties and plays a
different role to that of smoothing parameters in long-run variance
estimators. Our simulation study suggests that the performance of the test is
not sensitive to the trimming parameter.

A key ingredient of the self-normalized statistic is an estimate of
cross-correlation based on subsamples:
\[
\hat{\rho}_{\tau,s}(k)=\frac{\sum_{t=k+1}^{s} \psi_{\tau_{1}}(y_{1t}- \hat
{q}_{1,t,s}(\tau_{1})) \psi_{\tau_{2}}(y_{2,t-k}-\hat{q}_{2,t-k,s}(\tau_{2}%
))}{\sqrt{\sum_{t=k+1}^{s}\psi_{\tau_{1}}^{2}(y_{1t}-\hat{q}_{1,t,s}(\tau
_{1}))}\sqrt{\sum_{t=k+1}^{s}\psi_{\tau_{2}}^{2}(y_{2,t-k}-\hat{q}%
_{2,t-k,s}(\tau_{2}))}},
\]
for $[T\omega]\leq s\leq T$. For a finite integer $p>0$, let $\hat{\rho}%
_{\tau,s}^{(p)}=[\hat{\rho}_{\tau,s}(1),\dots,\hat{\rho}_{\tau,s}(p)]^{^{\top
}}$. We construct an outer product of the cross-quantilogram using the
subsample
\[
\hat{V}_{\tau, p}=T^{-2}\sum_{s=[T\omega]}^{T}s^{2}\left(  \hat{\rho}_{\tau
,s}^{(p)}-\hat{\rho}_{\tau}^{(p)}\right)  \left(  \hat{\rho}_{\tau,s}%
^{(p)}-\hat{\rho}_{\tau}^{(p)}\right)  ^{^{\top}}.
\]
We can obtain the asymptotically pivotal distribution using $\hat{V}_{\tau, p}
$ as the asymptotically random normalization. For testing the null of no
directional predictability, we define the self-normalized omnibus test
statistic
\[
\hat{S}_{\tau}^{(p)}=T\hat{\rho}_{\tau}^{(p)^{\top}}\hat{V}_{\tau, p}^{-1}%
\hat{\rho}_{\tau}^{(p)}.
\]
The following theorem shows that $\hat{S}_{\tau}^{(p)}$ is asymptotically
pivotal. To distinguish the process used in the following theorem from the one
used in the previous section, let $\{\bar{\mathbf{B}}^{(p)}(\cdot)\}$ denote a
$p$-dimensional, standard Brownian motion on $( \ell([0,1]) )^{p}$ equipped
with the uniform topology.

\begin{theorem}
\label{theorem:self-norm} Suppose that Assumptions A1-A5 hold. Then, for each
$\tau\in\mathcal{T}$,
\[
\hat{S}_{\tau}^{(p)}\to^{d} \bar{\mathbf{B}}^{(p)}(1)^{^{\top}}\left(
\bar{\mathbf{V}}^{(p)}\right)  ^{-1}\bar{\mathbf{B}}^{(p)}(1),
\]
where $\bar{\mathbf{V}}^{(p)}=\int_{\omega}^{1} \{ \bar{\mathbf{B}}%
^{(p)}(r)-r\bar{\mathbf{B}}^{(p)}(1) \} \{ \bar{\mathbf{B}}^{(p)}%
(r)-r\bar{\mathbf{B}}^{(p)}(1) \} ^{^{\top}}dr$.\bigskip
\end{theorem}

The joint test based on finite multiple quantiles can be constructed in a
similar manner, while the extension of the self-normalized approach to a range
of quantiles is not obvious. The asymptotic null distribution in the above
theorem can be simulated and a critical value, $c_{S,\alpha}$, corresponding
to a significance level $\alpha$ is tabulated by using the $(1-\alpha)100\%$
percentile of the simulated distribution.\footnote{We provide the simulated
critical values in our R package.} In the theorem below, we consider a power
function of the self-normalized omnibus test statistic, $P(\hat{S}_{\tau
}^{(p)}>c_{S,\alpha})$. For a fixed $\tau\in\mathcal{T}$, we consider a fixed
alternative
\begin{equation}
H_{1}:\rho_{\tau}(k)\mathrm{\ }\neq0\mathrm{\ for\ some\ }k\in\{1,\dots,p\},
\label{eq:alt-fix-sub}%
\end{equation}
and a local alternative
\begin{equation}
H_{1T}:\rho_{\tau}(k)=\zeta/\sqrt{T}\ \mathrm{for\ some}\ k\in\{1,\dots,p\},
\label{eq:alt-local-sub}%
\end{equation}
where $\zeta$ is a finite non-zero scalar. This implies that there exists a
$p$-dimensional vector $\zeta_{\tau}^{(p)}$ such that $\rho_{\tau}%
^{(p)}=T^{-1/2}\zeta_{\tau}^{(p)}$ with $\zeta_{\tau}^{(p)}$ having at least
one non-zero element.

\begin{theorem}
\label{theorem:self-norm-power} (a) Suppose that the fixed alternative in
(\ref{eq:alt-fix-sub}) and Assumptions A1-A5 hold. Then,
\[
\lim_{T\rightarrow\infty}P\left(  \hat{S}_{\tau}^{(p)}>c_{S,\alpha}\right)
\rightarrow1.
\]
(b) Suppose that the local alternative in (\ref{eq:alt-local-sub}) is true and
Assumptions A1-A5 hold. Then,
\[
\lim_{T\rightarrow\infty}P\left(  \hat{S}_{\tau}^{(p)}>c_{S,\tau}\right)
=P\left(  \left\{  \bar{\mathbf{B}}^{(p)}(1)+(\Lambda_{\tau}^{(p)}\Delta
_{\tau}^{(p)})^{-1}\zeta_{\tau}^{(p)}\right\}  ^{^{\top}}\left(
\mathbf{V}^{(p)}\right)  ^{-1}\left\{  \bar{\mathbf{B}}^{(p)}(1)+(\Lambda
_{\tau}^{(p)}\Delta_{\tau}^{(p)})^{-1}\zeta_{\tau}^{(p)}\right\}  \geq
c_{S,\alpha}\right)  ,
\]
where $\Delta_{\tau}^{(p)}$ is a $d_{0}p\times d_{0}p$ matrix with
$\Delta_{\tau}^{(p)}(\Delta_{\tau}^{(p)})^{^{\top}}\equiv\Xi^{(p)}(\tau,\tau)$.
\end{theorem}

\section{The Partial Cross-Quantilogram}

We define the partial cross-quantilogram, which measures the relationship
between two events $\{y_{1t}\leq q_{1,t}(\tau_{1})\}$ and $\{y_{2,t-k}\leq
q_{2,t-k}(\tau_{2})\}$, while controlling for intermediate events between $t$
and $t-k$ as well as whether some state variables exceed a given quantile. Let
$\mathbf{z}_{t}\equiv\lbrack\psi_{\tau_{3}}(y_{3t}-q_{3,t}(\tau_{3}%
)),\dots,\psi_{\tau_{l}}(y_{lt}-q_{l,t}(\tau_{l}))]^{^{\top}}$ be an
$(l-2)\times1$ vector for $l\geq3$, where $q_{i,t}(\tau_{i})=x_{it}^{\top
}\beta_{i}(\tau_{i})$ for $\tau_{i}$ and a $d_{i}\times1$ vector $x_{it}$
($i=3,\dots,l$), and $\mathbf{z}_{t}$ may include the quantile-hit processes
based on some of the lagged predicted variables $\{y_{1,t-1},\dots
,y_{1,t-k}\}$, the intermediate predictors $\{y_{2,t-1},\dots,y_{1,t-k-1}\}$
and some state variables that may reflect some historical events up to
$t$.\footnote{In principle, the intermediate predictors and state variables do
not need to be transformed into quantile hits. As emphasized earlier, however,
one of the main advantages of considering qauntile hits is its applicability
to more general time series, being robust to the existence of moments. If
needed, it is straightforward to extend the results here to the case of the
original variables in $\mathbf{z}_{t}$ with additional moment conditions. We
thank an anonymous referee for pointing this out.}
%For each $i=3, \dots, l$,
%we use $q_{i,t}(\tau _{i})$
%to denote the conditional quantile function of $y_{it}$
%given a $d_{i} \times 1$ vector $x_{it}$
%for $\tau _{i}\in (0,1)$
%and consider the case where
%each conditional quantile function is represented by a linear model.
%Define
%$\bar{\mathbf{q}}_{t}(\bar{\tau})=
%[q_{3t}(\tau _{3}),\dots
%,q_{lt}(\tau _{l})]^{^{\top }}$, with
%$\bar{\tau} =(\tau _{3},\dots ,\tau _{l})^{^{\top }}$
%and
%let
%$\bar{\mathbf{x}}_{t} = [x_{3t}^{\top}, \dots, x_{lt}^{\top}]^{\top}$.

For simplicity, we present the results for a single set of quantiles
$\bar{\tau}=(\tau_{1},\dots,\tau_{l})^{^{\top}}$ and a single lag $k$,
although the results can be extended to the case of a range of quantiles and
multiple lags in an obvious way. To ease the notational burden in the rest of
this section, we consider the case for which a lag $k=0$ without loss of
generality and suppress the dependence on $k$. Let $\bar{\mathbf{y}}%
_{t}=[y_{1t},\dots,y_{lt}]^{\top}$ and $\bar{\mathbf{x}}_{t}=[x_{1t}^{\top
},\dots,x_{lt}^{\top}]^{\top}$.

%% partial cross-correlation matrix
We introduce the correlation matrix of the hit processes and its inverse
matrix
\[
R_{\bar{\tau}}=E\left[  h_{t}(\bar{\tau})h_{t}(\bar{\tau})^{^{\top}}\right]
\ \ \mathrm{and}\ \ P_{\bar{\tau}}=R_{\bar{\tau}}^{-1},
\]
where an $l\times1$ vector of the hit process is denoted by $h_{t}(\bar{\tau
})=[\psi_{\tau_{1}}(y_{1t}-q_{1,t}(\tau_{1})),\dots,\psi_{\tau_{l}}%
(y_{lt}-q_{l,t}(\tau_{l}))]^{^{\top}}$. For $i,j\in\{1,\dots,l\}$, let
$r_{\bar{\tau},ij}$ and $p_{\bar{\tau},ij}$ be the $(i,j)$ element of
$R_{\bar{\tau}}$ and $P_{\bar{\tau}}$, respectively. Notice that the
cross-quantilogram is $r_{\bar{\tau},12}/\sqrt{r_{\bar{\tau},11}r_{\bar{\tau
},22}},$ and the partial cross-quantilogram is defined as
\[
\rho_{\bar{\tau}|\mathbf{z}}=-\frac{p_{\bar{\tau},12}}{\sqrt{p_{\bar{\tau}%
,11}p_{\bar{\tau},22}}}.
\]
The partial cross-correlation also has a form
\[
\rho_{\bar{\tau}|\mathbf{z}}=\delta\sqrt{\frac{\tau_{1}(1-\tau_{1})}{\tau
_{2}(1-\tau_{2})}},
\]
where $\delta$ is a scalar parameter defined in the following regression:
\[
\psi_{\tau_{1}}(y_{1t}-q_{1,t}(\tau_{1}))=\delta\psi_{\tau_{2}}(y_{2t}%
-q_{2,t}(\tau_{2}))+\gamma^{\top}\mathbf{z}_{t}+u_{t},
\]
with a $(l-2)\times1$ vector $\gamma$ and an error term $u_{t}$. Thus, testing
the null hypothesis of $\rho_{\bar{\tau}|\mathbf{z}}=0$ can be viewed as
testing predictability between two quantile hits with respect to information
$\bar{z}$ as in Granger causality test based on the regression form (Granger,
1969). By choosing relevant variables $\bar{z}$, one can use $\rho_{\bar{\tau
}|\mathbf{z}}$ for the purpose of testing Granger causality (Pierce and Haugh,
1977). See also Hong et al. (2009) for testing Granger causality in tail distribution.

%% estimator
To obtain the sample analogue of the partial cross-quantilogram, we first
construct a vector of hit processes, $\hat{h}_{t}(\bar{\tau})$, by replacing
the population conditional quantiles in $h_{t}(\bar{\tau})$ by the sample
analogues $\{\hat{q}_{1,t}(\tau_{1}),\dots,\hat{q}_{l,t}(\tau_{l})\}$. Then,
we obtain the estimator for the correlation matrix and its inverse as
\[
\hat{R}_{\bar{\tau}}=\frac{1}{T}\sum_{t=1}^{T}\hat{h}_{t}(\bar{\tau})\hat
{h}_{t}(\bar{\tau})^{^{\top}}\ \ \mathrm{and}\ \ \hat{P}_{\bar{\tau}}=\hat
{R}_{\bar{\tau}}^{-1},
\]
which leads to the sample analogue of the partial cross-quantilogram
\begin{equation}
\hat{\rho}_{\bar{\tau}|\mathbf{z}}=-\frac{\hat{p}_{\bar{\tau},12}}{\sqrt
{\hat{p}_{\bar{\tau},11}\hat{p}_{\bar{\tau},22}}}, \label{eq:pcq-sample}%
\end{equation}
where $\hat{p}_{\bar{\tau},ij}$ denotes the $(i,j)$ element of $\hat{P}%
_{\bar{\tau}}$ for $i,j\in\{1,\dots,l\}$.

%% inference
In Theorem \ref{theorem:pcq} below, we show that $\hat{\rho}_{\bar{\tau
}|\mathbf{z}}$ asymptotically follows a normal distribution, while the
asymptotic variance depends on nuisance parameters as in the previous section.
To address the issue of the nuisance parameters, we may employ the stationary
bootstrap or the self-normalization technique. For the bootstrap, we can use
pairs of variables $\{(\bar{\mathbf{y}}_{t},\bar{\mathbf{x}}_{t})\}_{t=1}^{T}$
to generate the stationary bootstrap resample $\{(\bar{\mathbf{y}}_{t}^{\ast
},\bar{\mathbf{x}}_{t}^{\ast})\}_{t=1}^{T}$ and then obtain the stationary
bootstrap version of the partial cross-quantilogram, denoted by $\hat{\rho
}_{\bar{\tau}|\mathbf{z}}^{\ast}$, using the formula in (\ref{eq:pcq-sample}).
When we use the self-normalized test statistics, we estimate the partial
cross-quantilogram $\rho_{\bar{\tau},s|\mathbf{z}}$ based on the subsample up
to $s$, recursively and then use
\[
\hat{V}_{\bar{\tau}|\mathbf{z}}=T^{-2}\sum_{s=[T\omega]}^{T}s^{2}\left(
\hat{\rho}_{\bar{\tau},s|\mathbf{z}}-\hat{\rho}_{\bar{\tau},T|\mathbf{z}%
}\right)  ^{2},
\]
to normalize the cross-quantilogram, thereby obtaining the asymptotically
pivotal statistics.

To obtain the asymptotic results, we impose the following conditions on the
conditional distribution function $F_{y_{i}|x_{i}}(\cdot|x_{it})$ and its
density function $f_{y_{i}|x_{i}}(\cdot|x_{it})$ of each pair of additional
variables $(y_{it},x_{it})$ for $i=1,\dots,l$ and on the pairwise joint
distribution $F_{ij}(v_{1},v_{2}|x_{it},x_{jt})\equiv P(y_{it}\leq
v_{1},y_{jt}\leq v_{2}|x_{it},x_{jt})$ for $(v_{1},v_{2})\in\mathds{R}^{2}$.

\vspace{0.5cm} \noindent\textbf{Assumption A7.} \textbf{(a)} $\{(\bar
{\mathbf{y}}_{t},\bar{\mathbf{x}}_{t})\}_{t\in\mathds{Z}}$ is a strictly
stationary and strong mixing sequence satisfying the condition in Assumption
A1; \textbf{(b)} The conditions in Assumption A2 and A3 hold for the
$F_{y_{i}|x_{i}}(\cdot|x_{it})$ and $f_{y_{i}|x_{i}}(\cdot|x_{it})$ at the
relevant quantile for $t=1,\dots,T$, for $i=1,\dots,l$; \textbf{(c)}
$F_{ij}(\cdot|x_{it},x_{jt})$ satisfies the condition in Assumption A4 and
there exists a vector $\nabla_{r}G_{ij}\equiv\partial/\partial b_{r}%
E[F_{ij}(x_{it}^{\top}b_{1},x_{jt}^{\top}b_{2}|x_{it},x_{jt})]$ evaluated at
$(b_{1},b_{2})=(\beta_{i}(\tau_{i}),\beta_{i}(\tau_{j}))$ for $(r,i,j)\in
\{1,2\}\times\{1,\dots,l\}^{2}$ ; \textbf{(d)} There exist positive definite
matrices $M_{i}$ and $D_{i}(\tau_{i})$ as in Assumption A5 for $i=1,\dots,l$.
\vspace{0.5cm}

Assumption A7(a) requires the same weak dependence property as in Assumption
A1. Assumptions A7(b)-(c) ensure the smoothness of the marginal conditional
distribution, marginal density function and the joint distribution of each
pair $(y_{it},y_{jt})$ given $(x_{it},x_{jt})$ for $1\leq i,j\leq l$.
Assumption A7(d) is used to derive a Bahadur representation of $\hat{q}%
_{it}(\tau_{i})$ for $i=1,\dots,l$.

We now state the asymptotic properties of the partial cross-quantilogram and
the related inference methods.

\begin{theorem}
\label{theorem:pcq} (a) Suppose that Assumption A7 holds. Then,
\[
\sqrt{T}(\hat{\rho}_{\bar{\tau}|\mathbf{z}}-\rho_{\bar{\tau}|\mathbf{z}%
})\rightarrow^{d}N(0,\sigma_{\bar{\tau}|\mathbf{z}}^{2}),
\]
for each $\bar{\tau}\in\lbrack0,1]^{l}$, where $\sigma_{\bar{\tau}|\mathbf{z}%
}^{2}=\sum_{l=-\infty}^{\infty}\mathrm{cov}(\xi_{\bar{\tau}l},\xi_{\bar{\tau
}0})$ with
\[
\xi_{\bar{\tau}t}=-\sum_{\substack{1\leq i,j\leq l\\i\not =j}}p_{\bar{\tau
},1i}p_{\bar{\tau},2j}\psi_{\tau_{i}}(y_{it}-q_{i,t}(\tau_{i}))\psi_{\tau_{j}%
}(y_{jt}-q_{j,t}(\tau_{j}))+\sum_{i=1}^{l}\lambda_{\bar{\tau}i}^{\top}%
D_{i}(\tau_{i})^{-1}x_{it}\psi_{\tau_{i}}(y_{it}-q_{i,t}(\tau_{i})),
\]
and $\lambda_{\bar{\tau}i}=\sum_{\substack{1\leq j\leq l\\j\not =i}}\left(
p_{\bar{\tau},1i}p_{\bar{\tau},2j}+p_{\bar{\tau},2i}p_{\bar{\tau},1j}\right)
\nabla_{1}G_{ij}$. \newline(b) Suppose that Assumption A6 and A7 hold. Then,
\[
\sup_{s\in\mathds{R}}\left\vert P^{\ast}\left(  \hat{\rho}_{\bar{\tau
}|\mathbf{z}}^{\ast}\leq s\right)  -P\left(  \hat{\rho}_{\bar{\tau}%
|\mathbf{z}}\leq s\right)  \right\vert \rightarrow^{p}0,
\]
for each $\bar{\tau}\in\lbrack0,1]^{l}$. \newline(c) Suppose that Assumption
A7 holds. Then, under the null hypothesis that $\rho_{\bar{\tau}|\mathbf{z}%
}=0$, we have
\[
\frac{\sqrt{T}\hat{\rho}_{\bar{\tau}|\mathbf{z}}}{\hat{V}_{\bar{\tau
}|\mathbf{z}}^{1/2}}\rightarrow^{d}\frac{\mathbf{B}(1)}{\left\{  \int_{\omega
}^{1}\{\mathbf{B}(1)-r\mathbf{B}(r)\}^{2}dr\right\}  ^{1/2}},
\]
for each $\bar{\tau}\in\lbrack0,1]^{l}$.
\end{theorem}

We can show that the partial cross-quantilogram has non-trivial local power
against a sequence of $\sqrt{T}$-local alternatives, applying the similar
arguments used in Theorem \ref{theorem:alternative} and Theorem
\ref{theorem:self-norm-power}, and thus we omit the details.

\section{Monte Carlo Simulation}

We investigate the finite sample performance of our test statistics. We adopt
the following simple VAR model with covariates and consider two data
generating processes for the error terms.%
\begin{align*}
y_{1t}  &  =0.1+0.3y_{1,t-1}+0.2y_{2,t-1}+0.3z_{1t}+u_{1t}\\
y_{2t}  &  =0.1+0.2y_{2,t-1}+0.3z_{2t}+u_{2t},
\end{align*}
where $z_{it}\sim iid$ $\chi^{2}(3)/3$ for $i=1,2$.

\noindent\textbf{DGP1}: $\left(  u_{1t},u_{2t}\right)  ^{\top}\sim iid$
$N\left(  0,I_{2}\right)  $ where $I_{2}$ is a $2\times2$ identity matrix. We
let $\left(  u_{1t},u_{2t},z_{1t},z_{2t}\right)  $ be mutually independent.

\noindent\textbf{DGP2}:
\[
\left(
\begin{array}
[c]{c}%
u_{1t}\\
u_{2t}%
\end{array}
\right)  =\left(
\begin{array}
[c]{cc}%
\sigma_{1t} & 0\\
0 & 1
\end{array}
\right)  \left(
\begin{array}
[c]{c}%
\varepsilon_{1t}\\
\varepsilon_{2t}%
\end{array}
\right)
\]
where $\left(  \varepsilon_{1t},\varepsilon_{2t}\right)  ^{\top}\sim$ $iid$
$N\left(  0,I_{2}\right)  $ and $\sigma_{1t}^{2}=0.1+0.2u_{1,t-1}%
^{2}+0.2\sigma_{1,t-1}^{2}+u_{2,t-1}^{2}.$ We let $\left(  \varepsilon
_{1t},\varepsilon_{2t},z_{1t},z_{2t}\right)  $ be mutually
independent.\bigskip

The sample cross-quantilogram defined in (\ref{q2}) adopts conditional
quantiles $\hat{q}_{it}(\tau_{i})=x_{it}^{\top}\hat{\beta}_{i}(\tau_{i}).$ We
first estimate $\beta(\tau)\equiv\lbrack\beta_{1}(\tau_{1})^{\top},\beta
_{2}(\tau_{2})^{\top}]^{\top}$ by quantile regression of the above VAR\ model,
where $x_{1t}=\left(  1,y_{1,t-1},y_{2,t-1},z_{1t}\right)  ^{\top}$ and
$x_{2t}=\left(  1,y_{2,t-1},z_{2t}\right)  ^{\top}$ and then obtain the sample
cross-quantilogram using $\hat{q}_{it}(\tau_{i})=x_{it}^{\top}\hat{\beta}%
_{i}(\tau_{i}).$

Under DGP1, there is no predictability from the event $\{y_{2,t-k}\leq
q_{2,t-k}(\tau_{2})\}$ to the event $\{y_{1t}\leq q_{1t}(\tau_{1})\}\ $for all
quantiles $\tau_{1}$\ and $\tau_{2},\ $because $\Pr\left[  y_{1t}\leq
q_{1t}(\tau_{1})\ |\ y_{2,t-k},x_{2,t-k}\right]  =\Pr\left[  u_{1t}\leq
\Phi^{-1}(\tau_{1})\right]  =\tau_{1}\ $for all $t\geq1\ $and$\ \tau_{1}%
\in(0,1),\ $where $\Phi$ denotes the standard normal cdf$.$

Under DGP2, $\left(  u_{1t}\right)  $ is defined as the GARCH-X process, where
its conditional variance is the GARCH(1,1) process with an exogenous
covariate. The GARCH-X process is commonly used for modeling volatility of
economic or financial time series in the literature, see Han (2015) and
references therein. Under DGP2, there exists predictability from
$\{y_{2,t-k}\leq q_{2,t-k}(\tau_{2})\}$ to $\{y_{1t}\leq q_{1t}(\tau_{1})\}$
through $\sigma_{1t}^{2}$\ for all quantiles $(\tau_{1}$,$\tau_{2}%
)\in(0,1)^{2},\ $except the case $\tau_{1}=0.5$ because the conditional
distribution of $u_{1t}$ given $x_{1t}\ $is symmetric around $0.\footnote{To
see this, note that the conditional distribution of $u_{1t}$ given $x_{1t}$
has median zero because $\Pr(u_{1t}\leq0\ |\ x_{1t})=\Pr(\sigma_{1t}%
\varepsilon_{1t}\leq0\ |\ x_{1t})=\Pr(\varepsilon_{1t}\leq0\ |\ x_{1t}%
)=\Pr(\varepsilon_{1t}\leq0)=0.5$ and likewise $\Pr(u_{1t}\geq0\ |\ x_{1t}%
)=0.5.\ $Therefore, letting $\mathcal{F}_{t}=(y_{2,t-k},x_{2,t-k}),$%
\ $\Pr\left(  y_{1t}<q_{1,t}(0.5)\ |\ \mathcal{F}_{t}\right)  =\Pr\left(
u_{1t}<0\ |\ \mathcal{F}_{t}\right)  =\Pr\left(  \varepsilon_{1t}%
<0\ |\ \mathcal{F}_{t}\right)  =0.5.$ This implies that there is no
predictability from $\{y_{2,t-k}\leq q_{2,t-k}(\tau_{2})\}$ to $\{y_{1t}\leq
q_{1,t}(\tau_{1})\}$ at $\tau_{1}=0.5$ under DGP2.}$

\subsection{Results Based on the Bootstrap Procedure}

We first examine the finite-sample performance of the Box-Ljung test
statistics based on the stationary bootstrap procedure. To save space, only
the results for the case where $\tau_{1}=\tau_{2}$ are reported here because
the results for the cases where $\tau_{1}\neq\tau_{2}$ are similar. The
Box-Ljung test statistics $\hat{Q}_{\tau}^{(p)}$ are based on $\hat{\rho
}_{\tau}(k)$ for $\tau_{i}=0.05,0.1,0.2,0.3,0.5,0.7,0.8,0.9$ or $0.95$ and
$k=1,2,\ldots,5$. Tables 1 and 2 report empirical rejection frequencies of the
Box-Ljung test statistics based on the bootstrap critical values at the 5\%
level. The sample sizes considered are $T=$500, 1,000 and 2,000$.$ The number
of simulation repetitions is 1,000. The bootstrap critical values are based on
1,000 bootstrapped replicates. The tuning parameter $\gamma$ is set to be
$0.01.\footnote{Recall that $1/\gamma$ indicates the average block length. We
tried different values for $\gamma$ including one chosen by the data dependent
rule suggested by Politis and White (2004) and the results are still similar
particularly for a large sample. The details of the data dependent rule is
explained in Section 6.}$

In general, our simulation results in Tables 1-3 show that the test has
reasonably good size and power performance in finite samples. Table 1 reports
the simulation results for the DGP1, which show the size performance. The
rejection frequencies are close to $0.05$\ in mid quantiles, while the test
tends to slightly under-reject in low and high quantiles.

Table 2 reports the simulation results for the DGP2, which show the power
performance. Except for the median, the rejection frequencies approach one as
the sample size increases, which shows that our test is consistent. As
expected, the rejection frequencies are close to $0.05$ at the median because
there is no predictability at the median under the DGP2 (see Footnote 10 for
an explanation).

Next, we examine the finite-sample performance of the sup-version of the
Box-Ljung test statistic $\sup_{\tau\in\mathcal{T}}\hat{Q}_{\tau}^{(p)}$ over
a range of quantiles.\footnote{Due to computational burden, we compute the
Box-Ljung test statistic as a maximum over nine quantile levels $\tau
_{i}=0.05,0.1,0.2,0.3,0.5,0.7,0.8,0.9$ and $0.95.$} The simulation results in
Table 3 show that the sup-version test statistic $\sup_{\tau\in\mathcal{T}%
}\hat{Q}_{\tau}^{(p)}$ also has reasonably good finite sample performance,
though it tends to under-reject under DGP1. For DGP2, the rejection
frequencies approach one as the sample size increases.

\subsection{Results for the Self-Normalized Statistics}

We also examine the performance of the self-normalized version of $\hat
{Q}_{\tau}^{(p)}$under the same setup as above. We fix the trimming constant
$\omega$ to be 0.1.\footnote{We also considered 0.03 and 0.05 for $\omega$ and
the results are similar to those for $\omega=0.1.$} The number of repetitions
is 3,000. The empirical sizes of the test are reported in Table 4, where the
underlying process is the VAR model with DGP1. The test generally
under-rejects under the null hypothesis (DGP1), while at the extreme quantiles
($\tau=0.05$ or $0.95)$ the test slightly over-rejects in the small sample
($T=500)$. This finding is not very surprising because the self-normalized
statistic is based on subsamples and at the extreme quantiles there are
effectively not enough observations to compute the test statistic accurately.

Using the GARCH-X process of DGP2, we obtain empirical powers and present the
results in Table 5. With a one-period lag ($p=1$), the
self-normalized{\small \ }quantilogram at $\tau_{1},\tau_{2}\in
\{0.1,0.2,0.8,0.9\}$ rejects the null by about 23.0-30.0\%, 64.3-68.3\% and
91.7-94.0\% for sample sizes 500, 1,000 and 2,000, respectively. In general,
the rejection frequencies increase as the sample size increases, decline as
the maximum number of lags $p$ increases, and are not sensitive to the choice
of the trimming value. Our results suggest that the self-normalized statistics
may have lower power in finite samples compared with the test statistics based
on the stationary bootstrap procedure, see Lobato (2001) for a similar finding.

\section{Empirical Studies}

\subsection{Stock Return Predictability}

We apply the cross-quantilogram to detect directional predictability from an
economic{\small \ }state variable to stock returns. The issue of stock return
predictability has been very important and extensively investigated in the
literature; see Lettau and Ludvigson (2010) for an extensive review. A large
literature has considered predictability of the mean of stock return. The
typical mean return forecast examines whether the mean of an economic state
variable is helpful in predicting the mean of stock return (mean-to-mean
relationship). Recently, Cenesizoglu and Timmermann (2008) considered whether
the mean of an economic state variable is helpful in predicting different
quantiles of stock returns representing left tail, right tail or shoulders of
the return distribution. The cross-quantilogram adds one more dimension to
analyze predictability compared with the linear quantile regression, and so it
provides a more complete picture on the relationship between a predictor and
stock returns. Moreover, we can consider very large lags in the framework of
the quantilogram.

We use daily data from 3 Jan. 1996 to 29 Dec. 2006 with sample size
2,717.\footnote{The working paper version of this paper provides the results
using the monthly data previously analyzed in Goyal and Welch (2008).} Stock
returns are measured by the log price difference of the S\&P 500 index and we
employ stock variance as the predictor. The stock variance is treated as an
estimate of equity risk in the literature. The risk-return relationship is an
important issue in the finance literature; see Lettau and Ludvigson (2010) for
an extensive review. The cross-quantilogram can provide a more complete
relationship from risk to return, which cannot be examined using\ existing
methods. To measure stock variance, we use the realized variance given by the
sum of squared 5-minute returns$.\footnote{The realized variance is obtained
from `Oxford-Man Institute's realised library'.}$\ The autoregressive
coefficient for stock variance is estimated to be $0.68$ and the unit root
hypothesis is clearly rejected. The sample mean and median of stock returns
are $0.0003$ and $0.0005,$ respectively$.$

In Figures 1-3, we provide the sup-type test statistic $\sup_{\tau
\in\mathcal{T}}\hat{Q}_{\tau}^{(p)},$ the cross-quantilogram $\hat{\rho}%
_{\tau}(k)$ and the portmanteau test $\hat{Q}_{\tau}^{(p)}$ (we use the
Box-Ljung versions throughout) to detect directional predictability from stock
variance, representing risk, to stock return. In each graph, we show the 95\%
bootstrap confidence intervals for no predictability based on 1,000
bootstrapped replicates. The tuning parameter $1/\gamma$ is chosen by adapting
the rule suggested by Politis and White (2004) (and later corrected in Patton
et al. (2009)).\footnote{Specifically, $1/\hat{\gamma}=(2\hat{G}^{2}/\hat
{D}_{SB})^{1/3}T^{1/3}$ where $\hat{D}_{SB}=2\hat{g}^{2}(0)$. The definitions
of $\hat{g}$ and $\hat{G}$ are given on page 58 of Politis and White (2004).}
Since it is for univariate data, we apply it separately to each time series
and define $\gamma$ as the average value.

We first examine the sup-version Box-Ljung test statistic $\sup_{\tau
\in\mathcal{T}}\hat{Q}_{\tau}^{(p)}$ and the results are provided in Figure 1.
We consider low and high ranges of quantiles. For the low range, we
set\ $\mathcal{T}=[0.1,0.3]$ and $\tau_{i}=0.1+0.02k$ for $k=0,1,\cdots,10.$
For the high range, we set $\mathcal{T}=[0.7,0.9]$ and $\tau_{i}=0.7+0.02k$
for $k=0,1,\cdots,10.$ In each range, there are eleven different values of
$\tau_{i}$ and we let $\tau_{1}=\tau_{2}$ in calculating $\hat{\rho}_{\tau
}(k)$ for simplicity. Figure 1 clearly shows that there exists predictability
from stock variance to stock return in each range.

Next we investigate the cross-quantilogram $\hat{\rho}_{\tau}(k)$ and the
portmanteau test $\hat{Q}_{\tau}^{(p)}$ for different quantile points in
Figures 2(a)-3(b). For the quantiles of stock return $q_{1}(\tau_{1})$, we
consider $\tau_{1}=0.05,$ $0.1,0.2,0.3,0.5,0.7,0.8,0.9$ and $0.95$. For the
quantiles of stock variance $q_{2}(\tau_{2})$, we consider $\tau_{2}=0.1$ and
$0.9.$ Figures 2(a) and 2(b) are for the case when the stock variance is in
the low quantile, i.e. $\tau_{2}=0.1$. The cross-quantilograms $\hat{\rho
}_{\tau}(k)$ for $\tau_{1}=0.05,$ $0.1,0.2$ and $0.3$ are negative and
significant for many lags. For example, in case of $\tau_{1}=0.05,$ it means
that when risk is very low, it is less likely to have a large negative loss.
On the other hand, the cross-quantilograms for $\tau_{1}=0.7,0.8,0.9$ and
$0.95$ is positive and significant for many lags. For example, in case of
$\tau_{1}=0.95,$ it means that when risk is very low, it is less likely to
have a large positive gain. However, the cross-quantilogram for $\tau_{1}=0.5$
is mostly insignificant, which means that risk is not helpful in predicting
whether stock return is located below or above its median. Figure 2(b) shows
that the Box-Ljung test statistics are mostly significant except for $\tau
_{1}=0.5$.

Figures 3(a) and 3(b) are for the case when stock variance is in the high
quantile, i.e. $\tau_{2}=0.9$. Compared to the previous case of $\tau
_{2}=0.1,$ the cross-quantilograms have similar trends but much larger
absolute values. For $\tau_{1}=0.05,$ the cross-quantilogram $\hat{\rho}%
_{\tau}(1)$ is $-0.193,$ which implies that when risk is higher than its $0.9$
quantile, there is an increased likelihood of having a very large negative
loss in the next day. For $\tau_{1}=0.95,$ the cross-quantilogram $\hat{\rho
}_{\tau}(1)$ is $0.188,$ which implies that when risk is high (higher than its
$0.9$ quantile), there is an increased likelihood of having a very large
positive gain in the next day. The cross-quantilogram for $\tau_{1}=0.5$ is
mostly insignificant and the Box-Ljung test statistics in Figure 3(b) are
mostly significant except for $\tau_{1}=0.5$.

The results in Figures 1-3 show that stock variance is helpful in predicting
stock return and detailed features depend on each quantile of stock variance
and stock return. When stock variance is in high quantile, the absolute value
of the cross-quantilogram is higher and the cross-quantilogram is
significantly different from zero for larger lags. Our results exhibit a more
complete relationship between risk and return and additionally show how the
relationship changes for different lags.

\subsection{Systemic Risk}

The Great Recession of 2007-2009 has motivated researchers to better
understand systemic risk---the risk that the intermediation capacity of the
entire financial system can be impaired, with potentially adverse consequences
for the supply of credit to the real economy. One approach to measure systemic
risk is measuring co-dependence in the tails of equity returns of an
individual financial institution and the financial system.\footnote{Bisias et
al. (2012) categorize the current approaches to measuring systemic risk along
the following lines: 1) tail measures, 2) contingent claims analysis, 3)
network models, and 4) dynamic stochastic macroeconomic models.} Prominent
examples include the work of Adrian and Brunnermeier (2011), Brownlees and
Engle (2012) and White et al. (2012). Since the cross-quantilogram measures
quantile dependence between time series, we apply it to measure systemic risk.

We use the daily CRSP market value weighted index return as the market index
return as in Brownlees and Engle (2012). We consider returns on JP Morgan
Chase (JPM), Morgan Stanley (MS) and AIG as individual financial institutions.
As in Brownlees and Engle (2012), JPM, MS and AIG belong to the Depositories
group, the Broker-Dealers group and the Insurance group, respectively. We
investigate the cross-quantilogram $\hat{\rho}_{\tau}(k)$ between an
individual institution's stock return and the market index return for $k=60$
and $\tau_{1}=\tau_{2}=0.05$. In each graph, we show the 95\% bootstrap
confidence intervals for no quantile dependence based on 1,000 bootstrapped replicates.

The sample period is from 24 Feb. 1993 to 31 Dec. 2014 with sample size
5,505.\footnote{The stock return series of Morgan Stanley are available from
24 Feb. 1993. The stock return series of individual financial institutions are
obtained from Yahoo Finance.} The data including the financial crisis from
2007 and 2009 might not be suitable to be viewed as a strictly stationary
sequence and hence may not fit into our theoretical framework.\footnote{A
rigorous treatment of nonstationary time series in our context is a
challenging issue and will be reported in a future work.} Nevertheless, we
provide the empirical results because it would be practically interesting to
consider a sample period that includes the recent crisis and
post-crisis.\footnote{The results for the sample period from 24 Feb. 1993 to
29 Dec. 2006 are also available from the authors upon request.}

In Figure 4, each graph in the left column shows the cross-quantilogram from
each individual institution to the market. The cross-quantilograms are
positive and generally significant for large lags. The cross-quantilogram from
JPM to the market reaches its peak ($0.146$) at $k=12$ and declines steadily
afterwards. This means that it takes about two weeks for the systemic risk
from JPM to reach its peak once JPM is in distress. From MS to the market, the
cross-quantilogram reaches its peak ($0.127$) at $k=2.$ From AIG to the
market, the cross-quantilogram reaches its peak ($0.127$) at $k=17.$ When
AIG\ is in distress, the systemic risk from AIG takes a longer time (about
three weeks) to reach its peak. When an individual financial institution is in
distress, each institution makes an influence on the market in a different way.

Each graph in the right column of Figure 4 shows the cross-quantilograms from
the market to an individual institution. The cross-quantilogram for this case
is a measure of an individual institution's exposure to system wide distress
and therefore it is similar to the stress tests performed by individual
institutions. From the market to each institutions, the cross-quantilogram at
$k=1$ is relatively low for JPM ($0.062$) and MS ($0.073$) while it is higher
for AIG ($0.104$). Overall, when the market is in distress, each institution
is influenced by its impact in a different way. But the cross-quantilogram
reaches its peak at $k=2$ for all cases. The cross-quantilograms at $k=2$ are
$0.135$, $0.131$ and $0.139$ for JPM, MS and AIG, respectively.

As shown in Figure 4, the cross-quantilogram is a measure\ for either an
institution's systemic risk or an institution's exposure to system wide
distress. Compared to existing methods, one important feature of the
cross-quantilogram is that it provides in a simple manner how such a measure
changes as the lag $k$ increases. For example, White et al. (2012) adopt an
additional impulse response function within the multivariate and
multi-quantile framework to consider tail dependence for a large $k$.
Moreover, another feature of the cross-quantilogram is that it does not
require any modeling. For example, the approach by Brownlees and Engle (2012)
is based on the standard\ multivariate GARCH model and it requires the
modeling of the entire multivariate distribution.

Next, we apply the partial cross-quantilogram to examine the systemic risk
after controlling for an economic state variable. Following Adrian and
Brunnermeier (2011) and Bedljkovic (2010), we adopt the VIX index as the
economic state variable. Since the VIX\ index itself is highly persistent and
can be modeled as an integrated process, we instead use the VIX index change,
the first difference of the VIX\ index level, as the state variable. For the
quantile of the state variable, i.e. $\tau_{3}$ in (\ref{eq:pcq-sample}), we
let $\tau_{3}=0.95$ because a low quantile of a stock return is generally
associated with a rapid increase of the VIX\ index.

Figure 5 shows that the partial cross-quantilograms are still significant in
some cases even if their values are generally lower than the values of the
cross-quantilograms in Figure 4. This indicates that there still remains
systemic risk from an individual institution after controlling for an economic
state variable. These significant partial cross-quantilograms will be of
interest for the management of the systemic risk of an individual financial institution.

\section{Conclusion}

We have established the limiting properties of the cross-quantilogram in the
case of a finite number of lags. Hong (1996) established the properties of the
Box-Pierce statistic in the case that $p=p_{n}\rightarrow\infty:$ after a
location and scale adjustment the statistic is asymptotically normal, see also
Hong et. al. (2009) for a related work. No doubt our results can be extended
to accommodate this case, although in practice the desirability of such a test
is questionable, and the chi-squared type limit in our theory may provide
better critical values for even quite long lags. The cross-quantilogram is
easy to compute and the bootstrap confidence intervals appear to represent
modest enlargements of the Bartlett intervals in the series that we examined.
The statistic shows the cross dependence structure of the time series in a
granular fashion that is more informative than the usual methods.

\newpage

\section*{Appendix}

\singlespacing
%% Equation numbering A1.....
\setcounter{equation}{0} \setcounter{page}{1}
\renewcommand{\theequation}{A-\arabic{equation}}
\renewcommand*{\thepage}{A.\arabic{page}} \renewcommand{\theLemma}{\arabic{Lemma}}

In appendix, we use $C$, $C_{1},C_{2},\dots$ to denote generic positive
constants without further clarification.

%%%%%%%%%%%%%%%%%%%%%%%%%%%%%%%%%%%%%%%%%%%%%%%%%%%%%%%%%%%%%%%%%%%%%%%%%%%%%%

%%%%%%%%%%%%%%%%%%%%%%%%%%%%%
%% setting for appendix: A
%%%%%%%%%%%%%%%%%%%%%%%%%%%%%
\vspace{0.5cm} \renewcommand{\theLemma}{A\arabic{Lemma}}
\renewcommand{\theProposition}{A\arabic{Proposition}} \setcounter{Lemma}{0} \setcounter{Proposition}{0}

\noindent\textbf{Appendix A. Asymptotic Results of Cross-Quantilogram}

%%%%%%%%%%%%%%%%%%%%%%%%%%%%%%%%%%%%%%%%%%%%%%%%%%%%%%%%%%%%%%%%%%%%%%%%%%%%%%

%%%% Lemma: Basic inequalities %%%%%%%%%%%%%%%%%%%%%%%%%%%%%%%%%%
\vspace{0.3cm}

\begin{Lemma}
\label{lemma:basic-ineq-A} Let $\{z_{t}\}_{t\in\mathds{Z}}$ be a strict
stationary, strong mixing sequence of $\mathds{R}^{d}$-valued random variables
for some integer $d\geq1$ with strong mixing coefficients $\{\alpha
_{j}\}_{j\in\mathds{Z}_{+}}$ satisfying $\sum_{j=0}^{\infty}(j+1)^{2s-2}%
\alpha_{j}^{\nu/(2s+\nu)}$ for some integer $s\geq2$ and $\nu\in(0,1)$.
Suppose that $E[z_{1}]=0$ and $\Vert z_{1}\Vert_{2s+\nu}<\infty$. Then,
\[
E\Big \|\sum_{t=1}^{T}z_{t}\Big \|^{2s}\leq T^{s}C\left\{  \left\Vert
z_{1}\right\Vert _{2+\nu}^{2s}+T^{1-s}\left\Vert z_{1}\right\Vert _{2s+\nu
}^{2s}\right\}  .
\]

\end{Lemma}

\begin{proof}
See Supplemental Material.
\end{proof}

\vspace{0.5cm}
%%%%%%%%%%%%%%%%%%%%%%%%%%%%%%%%%%%%%%%%%%%%%%%%%%%%%%%%%%%%%%%%%%%%%%%%%%%%%%

%%%%%%%%%%%%%%%%%%%%%%%%%%%%%%%%%%%%%%%%%%%%%%%%%%%%%%%%%%%%%%%%%%%%%%%%%%%%%%

%%%% Proposition: Stochastic equicontinuity %%%%%%%%%%%%%%%%%%%%%%%%%%%%%%%%%%
We define the process indexed by $\tau\in\mathcal{T}$:
\[
\mathbb{V}_{t,k}(\tau):=\frac{1}{\sqrt{T}}\sum_{t=k+1}^{T}\left\{
1[\mathbf{y}_{t,k}\leq\mathbf{q}_{t,k}(\tau)]-E[F_{\mathbf{y}|\mathbf{x}%
}^{(k)}(\mathbf{q}_{t,k}(\tau)|\mathbf{x}_{t,k})]\right\}  .
\]
Also, define a $d_{i}\times1$ vector of random variables indexed by $\tau
_{i}\in\mathcal{T}_{i}$ for each $i=1,2$:
\[
\mathbb{W}_{i,T}(\tau_{i}):=\frac{1}{\sqrt{T}}\sum_{t=1}^{T}x_{it}\psi
_{\tau_{i}}\left(  y_{it}-q_{i,t}(\tau_{i})\right)  .
\]
The below lemma shows the stochastic equicontinuity of the processes defined
above, using a similar argument in Bai (1996).

%%%%%%%%%%%% Lemma: Stochastic equicontinuity %%%%%%%%%%%%%%%

\begin{Proposition}
\label{Proposition:cont} Suppose Assumption A1-A5 hold. Let $k \in\{1, \dots,
p\}$ and define metrics $\rho_{i}(\tau_{i}, \tau_{i}^{\prime})=|\tau
_{i}^{\prime}- \tau_{i}|$ for $\tau_{i}, \tau_{i}^{\prime}\in\mathcal{T}_{i}$
($i = 1,2$) and a metric $\rho(\tau, \tau^{\prime})=\sum_{i=1}^{2}\rho
_{i}(\tau_{i}, \tau_{i}^{\prime})$ for $\tau, \tau^{\prime}\in\mathcal{T}$. Then,

\begin{description}
\item[(a)] $\mathbb{V}_{T,k}(\tau)$ is stochastically equicontinuous on
$(\mathcal{T},\rho)$;

\item[(b)] $\mathbb{W}_{i,T}(\tau_{i})$ is stochastically equicontinuous on
$(\mathcal{T}_{i}, \rho_{i})$ for each $i = 1,2$.
\end{description}
\end{Proposition}

\begin{proof}
See Supplemental Material.
\end{proof}

\vspace{0.5cm}
%%%%%%%%%%%%%%%%%%%%%%%%%%%%%%%%%%%%%%%%%%%%%%%%%%%%%%%%%%%%%%

%%%%%%%%%%%%%%%%%%%%%%%%%%%%%%%%%%%%%%%%%%%%%%%%%%%%%%%%%%%%%%

%%%%%%%%%%%%%%%%%%%%%%%%%%%%%%%%%%%%%%%%%%%%%%%%%%%%%%%%%%%%%%
Because of the importance of the result, we present the central limit theorem
for strong mixing sequence in the lemma below. The proof can be found in
Corollary 5.1 of Hall and Heyde (1980) or Rio (1997, 2013) among others.
%%%%%%%%%%%%%%%%%%%%%%%%%%%%%%%%%%%%%%%%%%%%%%%%%%%%%%%%%%%%%%

%%%%%%%%%%%%%%%%%%%%%%%%%%%%%%%%%%%%%%%%%%%%%%%%%%%%%%%%%%%%%%

%%%%%%%%%%%%%%%%%%%%%%%%%%%%%%%%%%%%%%%%%%%%%%%%%%%%%%%%%%%%%%

\begin{Lemma}
\label{lemma:CLT-HH} Suppose that the strict stationary sequence
$\{z_{t}\}_{t\in\mathds{Z}}$ satisfies the strong mixing condition with
$E[z_{1}]=0$ and $E|z_{1}|^{2+\varsigma}<\infty$ for some $\varsigma
\in(0,\infty)$, while $\sum_{j=1}^{\infty}\alpha_{j}^{\varsigma/(2+\varsigma
)}<\infty$. Then, $\lim_{T\rightarrow\infty}E[(T^{-1/2}\sum_{t=1}^{T}%
z_{t})^{2}]=\sigma^{2}$ for some $\sigma^{2}\in\lbrack0,\infty)$. If
$\sigma^{2}>0$, then $\sigma^{-1}T^{-1/2}\sum_{t=1}^{T}z_{t}\rightarrow
^{d}N(0,1)$.
\end{Lemma}

\vspace{0.5cm}
%%%%%%%%%%%%%%%%%%%%%%%%%%%%%%%%%%%%%%%%%%%%%%%%%%%%%%%%%%%%%%

%%%%%%%%%%%%%%%%%%%%%%%%%%%%%%%%%%%%%%%%%%%%%%%%%%%%%%%%%%%%%%

%%%%%%%%%%%%%%%%%%%%%%%%%%%%%%%%%%%%%%%%%%%%%%%%%%%%%%%%%%%%%%

Define a $d_{0}\times1$ vector $\mathbb{B}_{T,k}(\tau)=[\mathbb{V}_{T,k}%
(\tau),\mathbb{W}_{1,T}(\tau_{1})^{\top},\mathbb{W}_{2,T}(\tau_{2})^{\top
}]^{\top} $ for $\tau\in\mathcal{T}$ and $k = 1, \dots, p$. The following
proposition shows the weak convergence of the process $\{\mathbb{B}_{T,k}%
(\tau):\tau\in\mathcal{T}\}_{k=1}^{p}$.

%%%% Lemma: weak convergence %%%%%%%%%%%%%%%%%%%%%%%%%%%%%%%%%%%%%%%%%%%%%%%

\begin{Proposition}
\label{Proposition:weak convergence} Suppose Assumptions A1-A5 hold. Then,
\begin{align*}
\left[  \mathbb{B}_{T,1}(\cdot), \dots, \mathbb{B}_{T,p}(\cdot) \right]
^{\top} \Rightarrow\left[  \mathbb{B}_{1}(\cdot), \dots, \mathbb{B}_{p}(\cdot)
\right]  ^{\top}.
\end{align*}

\end{Proposition}

\begin{proof}
Proposition \ref{Proposition:cont} shows that $[\mathbb{B}_{T,1}(\cdot
),\dots,\mathbb{B}_{T,p}(\cdot)]^{\top}$ is stochastic equicontinuous. Thus,
it remains to establish convergence of the finite dimensional distributions.
By the Cramer-Wold device, it suffices to show
\[
\sum_{j=1}^{J}\theta_{j}\sum_{k=1}^{p}\kappa_{k}^{\top}\mathbb{B}_{T,k}\left(
\tau^{(j)}\right)  \rightarrow^{d}N\left(  0,\sigma_{\theta,\kappa}%
^{2}\right)  ,
\]
for any $\{\theta_{j}\in\mathds{R}\}_{j=1}^{J}$, $\{\kappa_{k}\in
\mathds{R}^{d}\}_{k=1}^{p}$, $\{\tau^{(j)}\in\lbrack0,1]^{2}\}_{j=1}^{J}$, and
$J\geq1$, where
\begin{equation}
\sigma_{\theta,\kappa}^{2}=\sum_{j=1}^{J}\sum_{j^{\prime}=1}^{J}\theta
_{j}\theta_{j^{\prime}}\sum_{k=1}^{p}\sum_{k^{\prime}=1}^{p}\kappa_{k}^{\top
}\Xi_{k,k^{\prime}}(\tau^{(j)},\tau^{(j^{\prime})})\kappa_{k^{\prime}}.
\label{eq:asym-var}%
\end{equation}
The original time-series is a stationary sequence satisfying the strong mixing
condition in Assumption A1 and a measurable transformation involving lagged
variables satisfies the same mixing condition if the lag order is finite.
Hence, the central limit theorem for strong-mixing sequences in Lemma
\ref{lemma:CLT-HH} shows that the convergence in distribution to the normal
law with the finite variance. Therefore, we establish the weak convergence.
\end{proof}

\vspace{0.5cm}
%%%% END: Lemma: Weak Convergence %%%%%%%%%%%%%%%%%%%%%%%%%%%%%%%%

%% H(a,v) function
Let $\mathbf{v}=(v_{1},v_{2})\in\mathds{R}^{d_{1}}\times\mathds{R}^{d_{2}}$
and $\mathbf{v}_{t,k}=(v_{1,t},v_{2,t-k})^{\top}\in\mathds{R}^{2}$ with
$v_{i,t}=x_{it}^{\top}v_{i}$ for $i=1,2$ and for $t=1,\dots,T$. Define
\[
\mathbb{V}_{T,k}(\tau,\mathbf{v}):=\frac{1}{\sqrt{T}}\sum_{t=k+1}^{T}\left\{
1[\mathbf{y}_{t,k}\leq\mathbf{q}_{t,k}(\tau)+T^{-1/2}\mathbf{v}_{t,k}%
]-E[F_{\mathbf{y}|\mathbf{x}}^{(k)}(\mathbf{q}_{t,k}(\tau)+T^{-1/2}%
\mathbf{v}_{t,k}|\mathbf{x}_{t,k})]\right\}  ,
\]
and
\[
\mathbb{W}_{i,T}(\tau_{i},v_{i}):=\frac{1}{\sqrt{T}}\sum_{t=1}^{T}%
x_{it}\left\{  1[y_{it}\leq q_{i,t}(\tau_{i})+T^{-1/2}v_{i,t}]-F_{y_{i}|x_{i}%
}(q_{i,t}(\tau_{i})+T^{-1/2}v_{i,t}|x_{it})\right\}  .
\]

%%%%%%%%%%%%%%%%%%%%%%%%%%%%%%%%%%%%%%%%%%%%%%%%%%%%%%%%%%%%%%
%%%% Lemma: Approximation %%%%%%%%%%%%%%%%%%%%%%%%%%%%%%%%%%%%

\begin{Proposition}
\label{Proposition:approx} Suppose Assumption A1-A5 hold. Then,

\begin{description}
\item[(a)] $\sup_{\tau\in\mathcal{T}}\sup_{\mathbf{v}\in\mathcal{V}_{M}%
}|\mathbb{V}_{T,k}(\tau,\mathbf{v})-\mathbb{V}_{T,k}(\tau)|=o_{p}(1)$ for
every $M>0$;

\item[(b)] $\sup_{\tau_{i}\in\mathcal{T}_{i}}\sup_{v_{i}\in\mathcal{V}_{i,M}%
}\Vert\mathbb{W}_{i,T}(\tau_{i},v_{i})-\mathbb{W}_{i,T}(\tau_{i})\Vert
=o_{p}(1)$ for every $M>0$ and $i=1,2$,
\end{description}
\end{Proposition}

\noindent\textit{where }$\mathcal{V}_{M}=\mathcal{V}_{1,M}\times
\mathcal{V}_{2,M}$\textit{\ with }$\mathcal{V}_{i,M}=\{v_{i}\in R^{d_{i}%
}:\Vert v_{i}\Vert\leq M\}$\textit{\ for }$i=1,2$\textit{.}\vspace{0.3cm}

\begin{proof}
See Supplemental Material.
\end{proof}

%%%%%%%%%%%%%%% END - Lemma: approximation    %%%%%%%%%%%%%%%%%%%%%%%%%%%%%%%%%

%%%%%%%%%%%%%%%%%%%%%%%%%%%%%%%%%%%%%%%%%%%%%%%%%%%%%%%%%%%%%%%%%%%%%%%%%%%%%%%

%%%% Bahadur Representation %%%%%%%%%%%%%%%%%%%%%%%%%%%%%%%%%%%%%%%%%%%%%%%%%%%

\begin{Proposition}
\label{Proposition:uniform-bahadur} Suppose Assumption A1-A5 hold. Then, for
$i=1,2$
\begin{align*}
\sqrt{T} \{ \hat{\beta}_{i}(\tau_{i}) - \beta_{i}(\tau_{i}) \} = -D_{i}%
^{-1}(\tau_{i}) \mathbb{W}_{i, T}(\tau_{i}) + o_{p}(1),
\end{align*}
uniformly in $\tau\in\mathcal{T}_{i}$.
\end{Proposition}

\begin{proof}
See Supplemental Material.
\end{proof}

%%%%%%%%%%%%%%%%%%%%%%%%%%%%%%%%%%%%%%%%%%%%%%%%%%%%%%%%%%%%%%%%%%%%%%%%%%%%%%%

%%%%%%%%%%%%%%%%%%%%%%%%%%%%%%%%%%%%%%%%%%%%%%%%%%%%%%%%%%%%%%%%%%%%%%%%%%%%%%%

%%%%%%%%%%%%%%%%%%%%%%%%%%%%%%%%%%%%%%%%%%%%%%%%%%%%%%%%%%%%%%%%%%%%%%%%%%%%%%%
\vspace{0.5cm}

The below lemma shows that the limiting behavior of the cross-quantilogram
process reflects the contributions of estimation errors due to the estimation
of the conditional quantile function.

\begin{Proposition}
\label{Proposition:u-l-approx} Suppose that Assumption A1-A5 hold. Then, for
each $k\in\{1,\dots,p\}$,
\[
\sqrt{T} \left\{  \hat{\rho}_{\tau}(k) - \rho_{\tau}(k) \right\}  = \frac{
\mathbb{V}_{T,k}(\tau) + \nabla G^{(k)}(\tau)^{\top} \sqrt{T} \{ \hat{\beta
}(\tau) - \beta(\tau) \} }{ \sqrt{\tau_{1}(1-\tau_{1})\tau_{2}(1-\tau_{2})}
}+o_{p}(1),
\]
uniformly in $\tau\in\mathcal{T}$.
\end{Proposition}

\begin{proof}
Let $\hat{\gamma}_{\tau,k}=T^{-1}\sum_{t=k+1}^{T}\psi_{\tau_{1}}(y_{1t}%
-\hat{q}_{1,t}(\tau_{1}))\psi_{\tau_{2}}(y_{2,t-k}-\hat{q}_{2,t-k}(\tau_{2}))$
and $\gamma_{\tau,k}=E[\psi_{\tau_{1}}(y_{1t}-q_{1,t}(\tau_{1}))\psi_{\tau
_{2}}(y_{2,t-k}-q_{2,t-k}(\tau_{2}))]$. Using a similar argument in Lemma 2.1
of Arcones (1998), we can show $\sup_{\tau_{i}\in\mathcal{T}_{i}}|T^{-1/2}%
\sum_{t=1}^{T}\psi_{\tau_{i}}(y_{it}-\hat{q}_{i,t}(\tau_{i}))|=o_{p}(1)$ for
$i=1,2$, because $x_{it}$ includes a constant term. It follows that, uniformly
in $\tau\in\mathcal{T}$,
\begin{equation}
T^{-1}\sum_{t=1}^{T}\psi_{\tau_{i}}^{2}(y_{it}-\hat{q}_{i,t}(\tau_{i}%
))=\tau_{i}(1-\tau_{i})+o_{p}(1),\ \ \mathrm{for}\ i=1,2,\label{eq:1-den}%
\end{equation}
and
\[
\sqrt{T}\left(  \hat{\gamma}_{\tau,k}-\gamma_{\tau,k}\right)  =T^{-1/2}%
\sum_{t=k+1}^{T}\big \{1[\mathbf{y}_{t,k}\leq\hat{\mathbf{q}}_{t,k}%
(\tau)]-E\big[F_{\mathbf{y}|\mathbf{x}}^{(k)}(\mathbf{q}_{t,k}(\tau
)|\mathbf{x}_{t,k})\big]\big \}+o_{p}(1).
\]
Define $\mathcal{V}_{M}=\{\mathbf{v}\equiv(v_{1},v_{2})\in\mathds{R}^{d_{1}%
}\times\mathds{R}^{d_{2}}:\max_{i=1,2}\Vert v_{i}\Vert\leq M\}$ for some $M>0$
and let $\mathbf{v}_{t,k}=(x_{1t}^{\top}v_{1},x_{2,t-k}^{\top}v_{2})^{\top}$.
Then, Proposition \ref{Proposition:approx} implies
\begin{align*}
&  T^{-1/2}\sum_{t=k+1}^{T}\big \{1[\mathbf{y}_{t,k}\leq\mathbf{q}_{t,k}%
(\tau)+T^{-1/2}\mathbf{v}_{t,k}]-E\big [F_{\mathbf{y}|\mathbf{x}}%
^{(k)}(\mathbf{q}_{t,k}(\tau)|\mathbf{x}_{t,k})\big]\big \}\\
&  \hspace{2cm}=\mathbb{V}_{T,k}(\tau)+\sqrt{T}E\big[F_{\mathbf{y}|\mathbf{x}%
}^{(k)}(\mathbf{q}_{t,k}(\tau)+T^{-1/2}\mathbf{v}_{t,k}|\mathbf{x}%
_{t,k})-F_{\mathbf{y}|\mathbf{x}}^{(k)}(\mathbf{q}_{t,k}(\tau)|\mathbf{x}%
_{t,k})\big]+o_{p}(1),
\end{align*}
uniformly in $(\tau,\mathbf{v})\in\mathcal{T}\times\mathcal{V}_{M}$ for any
$M>0$. Also, the mean-value theorem implies $\sqrt{T}E[F_{\mathbf{y}%
|\mathbf{x}}^{(k)}({\mathbf{q}}_{t,k}(\tau)+T^{-1/2}\mathbf{v}_{t,k}%
|\mathbf{x}_{t,k})-F_{\mathbf{y}|\mathbf{x}}^{(k)}(\mathbf{q}_{t,k}%
(\tau)|\mathbf{x}_{t,k})]=\nabla G^{(k)}(\tau)^{\top}\mathbf{v}+o(1)$
uniformly in $(\tau,\mathbf{v})\in\mathcal{T}\times\mathcal{V}_{M}$. Thus, for
any $M>0$,
\begin{equation}
\sup_{(\tau,\mathbf{v})\in\mathcal{T}\times\mathcal{V}_{M}}|R_{T}%
(\tau,\mathbf{v})|=o_{p}(1),\label{eq:sup-op}%
\end{equation}
where
\begin{align*}
R_{T}(\tau,\mathbf{v):=} &  T^{-1/2}\sum_{t=k+1}^{T}\big \{1[\mathbf{y}%
_{t,k}\leq\mathbf{q}_{t,k}(\tau)+T^{-1/2}\mathbf{v}_{t,k}]-E\big[F_{\mathbf{y}%
|\mathbf{x}}^{(k)}(\mathbf{q}_{t,k}(\tau)|\mathbf{x}_{t,k})\big]\big\}\\
&  -\left(  \mathbb{V}_{T,k}(\tau)+\nabla G^{(k)}(\tau)^{\top}\mathbf{v}%
\right)  .
\end{align*}
Let $\epsilon$ be an arbitrary positive constant. Proposition
\ref{Proposition:weak convergence} and \ref{Proposition:uniform-bahadur} imply
that there exists a constant $M>0$ such that $P(\sup_{\tau\in\mathcal{T}}%
\Vert\hat{\beta}(\tau)-\beta(\tau)\Vert>M/\sqrt{T})<\epsilon$ for a
sufficiently large $T$. It follows that there exists an $M>0$ such that
\[
P\left(  \sup_{\tau\in\mathcal{T}}\left\vert R_{T}(\tau,\sqrt{T}\{\hat{\beta
}(\tau)-\beta(\tau)\}\mathbf{)}\right\vert \mathbf{>\epsilon}\right)
<\epsilon+P\Big(\sup_{(\tau,\mathbf{v})\in\mathcal{T}\times\mathcal{V}_{M}%
}\left\vert R_{T}(\tau,\mathbf{v)}\right\vert \mathbf{>\epsilon}\Big),
\]
for a sufficiently large $T$. Thus, (\ref{eq:sup-op}) yields
\[
\sqrt{T}\left(  \hat{\gamma}_{\tau,k}-\gamma_{\tau,k}\right)  =\mathbb{V}%
_{T,k}(\tau)+\nabla G^{(k)}(\tau)^{\top}\sqrt{T}\{\hat{\beta}(\tau)-\beta
(\tau)\}+o_{p}(1),
\]
uniformly in $\tau\in\mathcal{T}$. This together with (\ref{eq:1-den}) yields
the desired result.
\end{proof}

%%%% END: Lemma: Uniform Linear Approximation  %%%%%%%%%%%%%%%%%%%

%%%%%%%%%%%%%%%%%%%%%%%%%%%%%%%%%%%%%%%%%%%%%%%%%%%%%%%%%%%%%%%%%%

%%%% Theorem: Limiting Behavior for One dimensional %%%%%%%%%%%%%%
\vspace{0.5cm}

\begin{proof}
[\textbf{Proof of Theorem \ref{theorem:lim-p}}]For each $i=1,2$, Proposition
\ref{Proposition:uniform-bahadur} yields an asymptotic linear approximation,
$\sqrt{T}\{\hat{\beta}_{i}(\tau_{i})-\beta_{i}(\tau_{i})\}=-D_{i}^{-1}%
(\tau_{i})\mathbb{W}_{i,T}(\tau_{i})+o_{p}(1)$ uniformly in $\tau_{i}%
\in\mathcal{T}_{i}$, which with Proposition \ref{Proposition:u-l-approx} shows
that $\sqrt{T}\left\{  \hat{\rho}_{\tau}(k)-\rho_{\tau}(k)\right\}
=\lambda_{\tau,k}^{\top}\mathbb{B}_{T,k}(\tau)+o_{p}(1)$ uniformly in $\tau
\in\mathcal{T}$. For a finite $p>0$, we have
\begin{equation}
\sqrt{T}\left(  \hat{\rho}_{\tau}^{(p)}-\rho_{\tau}^{(p)}\right)
=\Lambda_{\tau}^{(p)}\mathbb{B}_{T}^{(p)}(\tau)+o_{p}(1),
\label{eq:expansion-p}%
\end{equation}
uniformly in $\tau\in\mathcal{T}$. The desired result is obtained from
Proposition \ref{Proposition:weak convergence} with the continuous mapping theorem.
\end{proof}

%%%% END:: Theorem: Limiting Behavior for One dimensional %%%%%%%%%%%%%%%%%%%%%%%%

%%%%%%%%%%%%%%%%%%%%%%%%%%%%%%%%%%%%%%%%%%%%%%%%%%%%%%%%%%%%%%%%%%%%%%%%%%%%%%%%%%

%%%%%%%%%%%%%%%%%%%%%%%%%%%%%%%%%%%%%%%%%%%%%%%%%%%%%%%%%%%%%%%%%%%%%%%%%%%%%%%%%%
%%%%  Bootstrap
%%%%%%%%%%%%%%%%%%%%%%%%%%%%%%%%%%%%%%%%%%%%%%%%%%%%%%%%%%%%%%%%%%%%%%%%%%%%%%%%%%

\vspace{0.5cm}

\noindent\textbf{Appendix B. Stationary Bootstrap} \vspace{0.5cm}

%%%%%%%%%%%%%%%%%%%%%%%%%%%%%
%% setting for appendix: B
%%%%%%%%%%%%%%%%%%%%%%%%%%%%%
\renewcommand{\theLemma}{B\arabic{Lemma}}
\renewcommand{\theProposition}{B\arabic{Proposition}}
\setcounter{Proposition}{0} \setcounter{Lemma}{0}

A positive integer valued, possibly infinite random variable $\mu$ is said to
be a \textit{stopping time} with respect to a filtration $\{\mathcal{F}%
_{n},n\geq1\}$ if $\{\mu=n\}\in\mathcal{F}_{n},\forall n\in\mathds{N}$. Given
random block lengths $\{L_{i}\}_{i\in\mathds{N}}$ under the stationary
bootstrap, define $N=\inf\{i\in\mathds{N}:L_{1}+\dots+L_{i}\geq n\}$. Then,
$N$ is a stopping time with respect to $\{\sigma(L_{1},\dots,L_{i}):1\leq
i\leq n\}$. In the following lemma, we present a moment inequality using ideas
found in the literature on the sopped random walk process. See Gut (2009) for
a comprehensive treatment.

%%%% Lemma: Preliminary  %%%%%%%%%%%%%%%%%%%%%%%%%%%%%%%%%%%%%%%%%%%%%%%%%%%%

\begin{Lemma}
\label{lemma:stop-moment} Let $\{z_{t}\}_{t\in\mathds{Z}}$ be a strict
stationary, strong mixing sequence of $\mathds{R}^{d}$-valued random variables
for some integer $d\geq1$ with strong mixing coefficients $\{\alpha
_{j}\}_{j\in\mathds{Z}_{+}}$ satisfying $\sum_{j=0}^{\infty}(j+1)^{2s-2}%
\alpha_{j}^{\nu/(2s+\nu)}$ for some integer $s\geq2$ and $\nu\in(0,1)$.
Suppose that $\Vert z_{1}\Vert_{2s+\nu}<\infty$ and a stationary bootstrap
resample, $\{z_{t}^{\ast}\}_{t=1}^{T}$, from $\{z_{t}\}_{t=1}^{T}$ satisfies
Assumption A6 with the sample size $T>0$. Define $S_{k,l}=\sum_{t=k}%
^{k+l-1}z_{t}$ and $S_{k,l}^{\ast}=\sum_{t=k}^{k+l-1}z_{t}^{\ast}$. Then,
\[
E\left\Vert S_{1,T}^{\ast}-E^{\ast}S_{1,T}^{\ast}\right\Vert ^{2s}\leq
C\Big \{(T\gamma)^{s}\sum_{l=1}^{\infty}\pi_{l}E\big \|\tilde{S}%
_{1,l}\big \|^{2s}+E\big \|\tilde{S}_{1,T}\big \|^{2s}\Big \},
\]
where $\tilde{S}_{k,l}=\sum_{t=k}^{k+l-1}(z_{t}-Ez_{t})$ for $k,l\in
\mathds{N}$.
\end{Lemma}

\begin{proof}
See Supplemental Material.
\end{proof}

%%%% Lemma: Preliminary  %%%%%%%%%%%%%%%%%%%%%%%%%%%%%%%%%%%%%%%%%%%%%%%%%%%%

%%%%%%%%%%%%%%%%%%%%%%%%%%%%%%%%%%%%%%%%%%%%%%%%%%%%%%%%%%%%%%%%%%%%%%%%%%%%%%%%%%

%%%% Lemma: Moment inequality %%%%%%%%%%%%%%%%%%%%%%%%%%%%%%%%%%%%%%%%%%%%%%%%%%%

\begin{Lemma}
\label{lemma:ineq-moment-boot} Suppose that the same conditions assumed in
Lemma \ref{lemma:stop-moment} hold. Then,%
\[
E\left\Vert S_{1,T}^{\ast}-E^{\ast}S_{1,T}^{\ast}\right\Vert ^{2s}\leq
T^{s}C\left(  \Vert z_{1}\Vert_{2+\nu}^{2s}+\gamma^{s-1}\Vert z_{1}%
\Vert_{2s+\nu}^{2s}\right)  ,
\]
for a sufficiently large $T$.
\end{Lemma}

\begin{proof}
See Supplemental Material.
\end{proof}

%%%% Lemma: Moment inequality %%%%%%%%%%%%%%%%%%%%%%%%%%%%%%%%%%%%%%%%%%%%%%%%%%%

%%%%%%%%%%%%%%%%%%%%%%%%%%%%%%%%%%%%%%%%%%%%%%%%%%%%%%%%%%%%%%%%%%%%%%%%%%%%%%%%%%
\vspace{0.5cm}

%%%%%%%%%%%% Lemma: Stochastic equicontinuity -SB %%%%%%%%%%%%%%%%%%%%%%%%%%%%%%%%
We now turn to the asymptotic results of cross-quantilogram based on the
stationary bootstrap. Define
\[
\mathbb{V}_{T,k}^{\ast}(\tau):=\frac{1}{\sqrt{T}}\sum_{t=k+1}^{T}\left\{
1[\mathbf{y}_{t,k}^{\ast}\leq\mathbf{q}_{t,k}^{\ast}(\tau)]-1[\mathbf{y}%
_{t,k}\leq\mathbf{q}_{t,k}(\tau)]\right\}
\]
and
\[
\mathbb{W}_{i,T}^{\ast}(\tau_{i}):=\frac{1}{\sqrt{T}}\sum_{t=k+1}^{T}\left\{
x_{it}^{\ast}\psi_{\tau_{i}}\left(  y_{it}^{\ast}-q_{i,t}^{\ast}(\tau
_{i})\right)  -x_{it}\psi_{\tau_{i}}\left(  y_{it}-q_{i,t}(\tau_{i})\right)
\right\}
\]
for each $i=1,2$. The lemma below shows the stochastic equicontinuity of the
processes, $\mathbb{V}_{T,k}^{\ast}(\cdot)$ and $\mathbb{W}_{i,T}^{\ast}%
(\cdot)$, \textit{unconditional} on the original sample.

%%%% Lemma: equicont %%%%%%%%%%%%%%%%%%%%%%%%%%%%%%%%%%%%%%%%
\vspace{0.5cm}

\begin{Proposition}
\label{Proposition:cont-boot} Suppose Assumption A1-A6 hold. Let
$k\in\{1,\dots,p\}$ and define metrics $\rho_{i}(\cdot,\cdot)$ for $i=1,2$ and
a metric $\rho(\cdot,\cdot)$ as in Proposition \ref{Proposition:cont}. Then,

\begin{description}
\item[(a)] $\mathbb{V}_{T,k}^{\ast}(\tau)$ is stochastically equicontinuous on
$(\mathcal{T},\rho)$;

\item[(b)] $\mathbb{W}_{i,T}^{\ast}(\tau_{i})$ is stochastically
equicontinuous on $(\mathcal{T}_{i},\rho_{i})$ for each $i=1,2$.
\end{description}
\end{Proposition}

\begin{proof}
See Supplemental Material.
\end{proof}

\vspace{0.5cm}
%%%% END:: Lemma: equicont Bootstrap  %%%%%%%%%%%%%%%%%%%%%%%%%%%%%%%%%%%%%%%%

%%%%%%%%%%%%%%%%%%%%%%%%%%%%%%%%%%%%%%%%%%%%%%%%%%%%%%%%%%%%%%%%%%%%%%%%%%%%%%

Let $\mathbb{B}_{T,k}^{\ast}(\tau)= [\mathbb{V}_{T,k}^{\ast}(\tau
),\mathbb{W}_{1,T}^{\ast}(\tau_{1})^{\top},\mathbb{W}_{2,T}^{\ast}(\tau
_{2})^{\top}]^{\top} $ for $(k,\tau)\in\{1, \dots, p\} \times\mathcal{T}$ and
define $\mathbb{B}_{T,k}^{(p) \ast}(\tau) := [ \mathbb{B}_{T,1}^{\ast}(\tau),
\dots, \mathbb{B}_{T,p}^{\ast}(\tau) ]^{\top} $. As a norm that introduces the
topology of $(\ell^{\infty}(\mathcal{T}))^{pd_{0}}$, we use $\sup_{\tau
\in\mathcal{T}} \|\cdot\|$ defined on $(\ell^{\infty}(\mathcal{T}))^{pd_{0}}$,
so that $\sup_{\tau\in\mathcal{T}} \|f(\tau)\|$ for any $f \in(\ell^{\infty
}(\mathcal{T}))^{pd_{0}}$. Let $BL_{1}$ be the set of all Lipschitz
continuous, real-valued functions on $(\ell^{\infty}(\mathcal{T}))^{pd_{0}}$
with a Lipschitz constant bounded by 1. We prove the following proposition by
modifying the argument used in Theorem 2 of Galvao et.~al.~(2014), where the
approach of van der Vaart and Wellner (1996, Theorem 2.9.6) is extended for
the dependent process but their setup differs from the one here.

%%%% Lemma; Weak Convergence Bootstrap %%%%%%%%%%%%%%%%%%%%%%%%%%%%%%%%%%%%%%
\vspace{0.5cm}

\begin{Proposition}
\label{Proposition:weak convergence-boot} Suppose Assumptions A1-A6 hold.
Then,
\begin{align*}
\sup_{h \in BL_{1}} \left|  E^{\ast} \big [ h(\mathbb{B}_{T}^{(p) \ast})
\big ] - E \big [ h(\mathbb{B}^{(p)}) \big ] \right|  \to^{p} 0.
\end{align*}

\end{Proposition}

\begin{proof}
Let $\delta>0$. Given the compact set $\mathcal{T}$ in $[0,1]^{2}$, there
exists a finite partition $\{\mathcal{T}^{(j)}\}_{j=1}^{J}$ such that
$\max_{1\leq j\leq J}\sup_{\tau^{\prime},\tau^{\prime\prime}\in\mathcal{T}%
^{(j)}}\Vert\tau^{\prime\prime}-\tau^{\prime}\Vert\leq\delta$. Pick up
$\tau^{(j)}\equiv(\tau_{1}^{(j)},\tau_{2}^{(j)})^{\top}\in\mathcal{T}^{(j)}$
for $j=1,\dots,J$ and let $\Pi_{\delta}$ be a map from $\mathcal{T}$ to
$\{\tau^{(j)}\}_{j=1}^{J}$ so that $\Pi_{\delta}(\tau)=\tau^{(j)}$ if $\tau
\in\mathcal{T}^{(j)}$. Define $\mathbb{B}_{T}^{(p)\ast}\circ\Pi_{\delta}$ and
$\mathbb{B}^{(p)}\circ\Pi_{\delta}$ as the stochastic processes on
$\mathcal{T}$, given by $\mathbb{B}_{T}^{(p)\ast}\circ\Pi_{\delta}%
(\tau)=\mathbb{B}_{T}^{(p)\ast}(\Pi_{\delta}(\tau))$ and $\mathbb{B}%
^{(p)}\circ\Pi_{\delta}(\tau)=\mathbb{B}^{(p)}(\Pi_{\delta}(\tau))$ for
$\tau\in\mathcal{T}$. It follows from the triangle inequality that, for any
$h\in BL_{1}$,
\begin{align}
\left\vert E^{\ast}\big [h(\mathbb{B}_{T}^{(p)\ast})\big ]-E\big [h(\mathbb{B}%
^{(p)})\big ]\right\vert  &  \leq\left\vert E^{\ast}\big [h(\mathbb{B}%
_{T}^{(p)\ast})\big ]-E^{\ast}\big [h(\mathbb{B}_{T}^{(p)\ast}\circ\Pi
_{\delta})\big ]\right\vert \label{eq:BL1}\\
&  \text{ \ \ }+\left\vert E^{\ast}\big [h(\mathbb{B}_{T}^{(p)\ast}\circ
\Pi_{\delta})\big ]-E\big [h(\mathbb{B}^{(p)}\circ\Pi_{\delta}%
)\big ]\right\vert \label{eq:BL2}\\
&  \text{ \ \ }+\left\vert E\big [h(\mathbb{B}^{(p)}\circ\Pi_{\delta
})\big ]-E\big [h(\mathbb{B}^{(p)})\big ]\right\vert . \label{eq:BL3}%
\end{align}
It suffices to show that (\ref{eq:BL1}) - (\ref{eq:BL3}) are $o_{p}(1)$
uniformly in $h\in BL_{1}$.

%%%%%%%%%%
%% Step 1
%%%%%%%%%%
We first consider (\ref{eq:BL1}). We have
\begin{align*}
E \left[  \sup_{h \in BL_{1}} \left|  E^{\ast} \big [ h(\mathbb{B}_{T}^{(p)
\ast}) \big ] - E^{\ast} \big [ h(\mathbb{B}_{T}^{(p) \ast} \circ\Pi_{\delta})
\big ] \right|  \right]  \le E \left[  \sup_{h \in BL_{1}} \left|
h(\mathbb{B}_{T}^{(p) \ast}) - h(\mathbb{B}_{T}^{(p) \ast} \circ\Pi_{\delta})
\right|  \right]  .
\end{align*}
Let $I_{T, \delta, \epsilon}^{\ast}:=1 [ \sup_{\tau\in\mathcal{T}}
\|\mathbb{B}_{T}^{(p) \ast}(\tau) - \mathbb{B}_{T}^{(p) \ast} \circ\Pi
_{\delta}(\tau) \| > \epsilon]$ for $\epsilon>0$. Proposition
\ref{Proposition:cont-boot} implies that $\lim_{\delta\downarrow0} \lim_{T
\to\infty} E[I_{T, \delta, \epsilon}^{\ast}]< \epsilon$ for every $\epsilon
>0$. Also $\sup_{h \in BL_{1}} | h(\mathbb{B}_{T}^{(p) \ast}) - h(\mathbb{B}%
_{T}^{(p) \ast} \circ\Pi_{\delta}) | \le2 $ because the range of a function
$h$ is $[-1,1]$. It follows that
\begin{align*}
\lim_{\delta\downarrow0} \lim_{T \to\infty} E \left[  \sup_{h \in BL_{1}}
\left|  h(\mathbb{B}_{T}^{(p) \ast}) - h(\mathbb{B}_{T}^{(p) \ast} \circ
\Pi_{\delta}) \right|  \cdot I_{T, \delta, \epsilon}^{\ast} \right]  \le2
\epsilon.
\end{align*}
Since $\sup_{h \in BL_{1}} | h(\mathbb{B}_{T}^{(p) \ast}) - h(\mathbb{B}%
_{T}^{(p) \ast} \circ\Pi_{\delta}) | \le\sup_{\tau\in\mathcal{T}}
\|\mathbb{B}_{T}^{(p) \ast}(\tau) - \mathbb{B}_{T}^{(p) \ast} \circ\Pi
_{\delta}(\tau) \| $, we have
\begin{align*}
E \left[  \sup_{h \in BL_{1}} \left|  h(\mathbb{B}_{T}^{(p) \ast}) -
h(\mathbb{B}_{T}^{(p) \ast} \circ\Pi_{\delta}) \right|  \cdot(1 - I_{T,
\delta, \epsilon}^{\ast}) \right]  \le\epsilon.
\end{align*}
Thus, $\lim_{\delta\downarrow0} \lim_{T \to\infty} E [ \sup_{h \in BL_{1}} |
h(\mathbb{B}_{T}^{(p) \ast}) - h(\mathbb{B}_{T}^{(p) \ast} \circ\Pi_{\delta})
| ] \le3 \epsilon$. An application of the Markov inequality yields that
(\ref{eq:BL1}) is $o_{p}(1)$ uniformly in $h \in BL_{1}$.

%%%%%%%%%%
%% Step 2
%%%%%%%%%%
Next we shall show that $\sup_{h\in BL_{1}}|E^{\ast}[h(\mathbb{B}_{T}%
^{(p)\ast}\circ\Pi_{\delta})]-E[h(\mathbb{B}^{(p)}\circ\Pi_{\delta
})]|\rightarrow^{p}0$ for any $\delta>0$. It suffices to show that
$\{\mathbb{B}_{T}^{(p)\ast}(\tau^{(j)})\}_{j=1}^{J}\rightarrow^{d}%
\{\mathbb{B}^{(p)}(\tau^{(j)})\}_{j=1}^{J}$ conditional on the original
sample, for almost every sequence.
%% Finite dimensional
To this end, we use the Cramer-Wold device and consider $\sum_{j=1}^{J}%
\theta_{j}\sum_{k=1}^{p}\kappa_{k}^{\top}\mathbb{B}_{T,k}^{\ast}\left(
\tau^{(j)}\right)  $ for some $\{\theta_{j}\in\mathds{R}\}_{j=1}^{J}$ and
$\{\kappa_{k}\in\mathds{R}^{d}\}_{k=1}^{p}$. Let $v_{t}^{\ast}=\sum_{j=1}%
^{J}\theta_{j}\sum_{k=1}^{p}\kappa_{k}^{\top}\xi_{t,k}^{\ast}(\tau^{(j)})$ and
$v_{t}=\sum_{j=1}^{J}\theta_{j}\sum_{k=1}^{p}\kappa_{k}^{\top}\xi_{t,k}%
(\tau^{(j)}),$ where $\xi_{t,k}(\cdot)$ is defined in Section 3 and $\xi
_{t,k}^{\ast}(\cdot)$ is its bootstrap counterpart. Then, we can write
\[
\sum_{j=1}^{J}\theta_{j}\sum_{k=1}^{p}\kappa_{k}^{\top}\mathbb{B}_{T,k}^{\ast
}\left(  \tau^{(j)}\right)  =T^{-1/2}\sum_{t=k+1}^{T}(v_{t}^{\ast}-v_{t}).
\]
As discussed in Proposition \ref{Proposition:weak convergence}, $\{v_{t}%
\}_{t\in\mathds{N}}$ is a stationary time-series satisfying Assumption 1. As
shown in p.~1237 of Kunsch (1989), the moment and strong-mixing assumption
imposed on the original time series implies the condition imposed on the forth
joint cumulant in (8) of Politis and Romano (1994). Hence, Theorems 1 and 2 of
Politis and Romano (1994) imply that the bootstrap estimate of the variance
converges to $\sigma_{\theta,\kappa}^{2}$ in probability, where $\sigma
_{\theta,\kappa}^{2}$ is defined in (\ref{eq:asym-var}), and that we obtained
the distribution convergence conditional on the original sample.

%%%%%%%%%%
%% Step 3
%%%%%%%%%%
Finally, consider (\ref{eq:BL3}). The process $\mathbb{B}^{(p) }$ is uniformly
continuous on $\mathcal{T}$, which with the dominated convergence theorem
yields that $\lim_{\delta\downarrow0} \sup_{h \in BL_{1}} \left|  E \big [
h(\mathbb{B}^{(p)} \circ\Pi_{\delta}) \big ]
- E \big [
h(\mathbb{B}^{(p) } ) \big ]
\right|  = 0 $. Hence, we obtain the desired conclusion.
\end{proof}

%%%%%%%%%%%%%%% END Proof: Bootstrap Consistency - Numerator %%%%%%%%%%%%%%%%%

%%%%%%%%%%%%%%%%%%%%%%%%%%%%%%%%%%%%%%%%%%%%%%%%%%%%%%%%%%%%%%%%%%%%%%%%%%%%%%
\vspace{0.5cm}

%%%%%%%%%%%%%%%%%%%%%%%%%%%%%%%%%%%%%%%%%%%%%%%%%%%%%%%%%%%%%%%%%%%%%%%%%%%%%%
%% H(a,v) function
For $\mathbf{v}=(v_{1},v_{2})\in\mathds{R}^{d_{1}}\times\mathds{R}^{d_{2}}$,
let $\mathbf{v}_{t,k}^{\ast}=(v_{1,t}^{\ast},v_{2,t-k}^{\ast})^{\top}$ with
$v_{i,t}^{\ast}=x_{it}^{\ast\top}v_{i}$ for $i=1,2$. Define
\[
\mathbb{V}_{T,k}^{\ast}(\tau,\mathbf{v}):=T^{-1/2}\sum_{t=k+1}^{T}\left\{
1[\mathbf{y}_{t,k}^{\ast}\leq\mathbf{q}_{t,k}^{\ast}(\tau)+T^{-1/2}%
\mathbf{v}_{t,k}^{\ast}]-1[\mathbf{y}_{t,k}\leq\mathbf{q}_{t,k}(\tau
)+T^{-1/2}\mathbf{v}_{t,k}]\right\}  ,
\]
and
\[
\mathbb{W}_{i,T}^{\ast}(\tau_{i},v_{i}):=T^{-1/2}\sum_{t=k+1}^{T}\left\{
x_{it}^{\ast}\psi_{\tau_{i}}\left(  y_{it}^{\ast}-q_{i,t}^{\ast}(\tau
_{i})-T^{-1/2}v_{i,t}^{\ast}\right)  -x_{it}\psi_{\tau_{i}}\left(
y_{it}-q_{i,t}(\tau_{i})-T^{-1/2}v_{i,t}\right)  \right\}  .
\]

%%%%%%%%%%%%%%%%%%%%%%%%%%%%%%%%%%%%%%%%%%%%%%%%%%%%%%%%%%%%%%
%%%% Lemma: Approximation %%%%%%%%%%%%%%%%%%%%%%%%%%%%%%%%%%%%

\begin{Proposition}
\label{Proposition:approx-boot} Suppose Assumption A1-A6 hold. Then,

\begin{description}
\item[(a)] $\sup_{\tau\in\mathcal{T}} \sup_{\mathbf{v} \in\mathcal{V}_{M}} |
\mathbb{V}_{T,k}^{\ast}(\tau, \mathbf{v}) - \mathbb{V}_{T,k}^{\ast}(\tau) |
=o_{p}(1) $ for every $M>0$;

\item[(b)] $\sup_{\tau_{i} \in\mathcal{T}_{i}} \sup_{v_{i} \in\mathcal{V}%
_{i,M}} \| \mathbb{W}_{i,T}^{\ast}(\tau_{i}, v_{i}) - \mathbb{W}_{i, T}^{\ast
}(\tau_{i}) \| = o_{p}(1) $ for every $M>0$ and $i = 1,2$,
\end{description}
\end{Proposition}

\noindent\textit{where }$\mathcal{V}_{M}=\mathcal{V}_{1,M}\times
\mathcal{V}_{2,M}$\textit{\ with }$\mathcal{V}_{i,M}=\{v_{i}\in R^{d_{i}%
}:\Vert v_{i}\Vert\leq M\}$\textit{\ for }$i=1,2$\textit{.}\vspace{0.3cm}

\begin{proof}
See Supplemental Material.
\end{proof}

\vspace{0.5cm}
%%%%%%%%%%%%%%%%%%%%%%%%%%%%%%%%%%%%%%%%%%%%%%%%%%%%%%%%%%%%%%%%%%%%%%%%%%%%%%%%%%

%%%%%%%%%%%%%%%%%%%%%%%%%%%%%%%%%%%%%%%%%%%%%%%%%%%%%%%%%%%%%%%%%%%%%%%%%%%%%%%%%%

%%%%%%%%%%%%%%%%%%%%%%%%%%%%%%%%%%%%%%%%%%%%%%%%%%%%%%%%%%%%%%%%%%%%%%%%%%%%%%%%%%

\begin{Proposition}
\label{Proposition:uniform-bahadur-bootstrap} Suppose Assumption A1-A6 hold.
Then, for $i=1,2$,
\[
\sqrt{T}\{\hat{\beta}_{i}^{\ast}(\tau_{i})-\beta_{i}(\tau_{i})\}=-D_{i}%
^{-1}(\tau_{i})\frac{1}{\sqrt{T}}\sum_{t=k+1}^{T}x_{it}^{\ast}\psi_{\tau_{i}%
}(y_{it}^{\ast}-q_{i,t}^{\ast}(\tau_{i}))+o_{p}(1),
\]
uniformly in $\tau_{i}\in\mathcal{T}_{i}$.
\end{Proposition}

\begin{proof}
A similar argument used in Proposition \ref{Proposition:uniform-bahadur}
completes the proof and thus the details are omitted.
\end{proof}

%%%%%%%%%%%%%%%%%%%%%%%%%%%%%%%%%%%%%%%%%%%%%%%%%%%%%%%%%%%%%%%%%%%%%%%%%%%%%%%%%%

%%%%%%%%%%%%%%%%%%%%%%%%%%%%%%%%%%%%%%%%%%%%%%%%%%%%%%%%%%%%%%%%%%%%%%%%%%%%%%%%%%

%%%%%%%%%%%%%%%%%%%%%%%%%%%%%%%%%%%%%%%%%%%%%%%%%%%%%%%%%%%%%%%%%%%%%%%%%%%%%%%%%%
\vspace{0.5cm}

\begin{Proposition}
\label{Proposition:u-l-approx-boot} Suppose that Assumption A1-A6 hold. Then,
for each $k\in\{1,\dots,p\}$,
\[
\sqrt{T} \left\{  \hat{\rho}_{\tau}^{\ast}(k) - \hat{\rho} _{\tau}(k)
\right\}  = \frac{ \mathbb{V}_{T,k}^{\ast}(\tau) + \nabla G^{(k)}(\tau)^{\top}
\sqrt{T}\{\hat{\beta}^{\ast}(\tau) - \hat{\beta}(\tau)\} }{ \sqrt{\tau
_{1}(1-\tau_{1})\tau_{2}(1-\tau_{2})} }+ o_{p}(1),
\]
uniformly in $\tau\in\mathcal{T}$.
\end{Proposition}

\begin{proof}
Let $\hat{\gamma}_{\tau,k}^{\ast}=T^{-1}\sum_{t=k+1}^{T}\psi_{\tau_{1}}%
(y_{1t}^{\ast}-\hat{q}_{1,t}^{\ast}(\tau_{1}))\psi_{\tau_{2}}(y_{2,t-k}^{\ast
}-\hat{q}_{2,t-k}^{\ast}(\tau_{2}))$. Using a similar argument used to show
Lemma 2.1 of Arcones (1998), we can show $\sup_{\tau_{i}\in\mathcal{T}_{i}%
}|T^{-1/2}\sum_{t=1}^{T}\psi_{\tau_{i}}(y_{it}^{\ast}-\hat{q}_{i,t}^{\ast
}(\tau_{i}))|=o_{p}(1)$ for $i=1,2$. It follows that
\[
T^{-1}\sum_{t=k+1}^{T}\psi_{\tau_{i}}^{2}(y_{it}^{\ast}-\hat{q}_{it}^{\ast
}(\tau_{i}))=\tau_{i}(1-\tau_{i})+o_{p}(1),\ \ \ \mathrm{for}\ i=1,2,
\]
and
\[
\sqrt{T}\left(  \hat{\gamma}_{\tau,k}^{\ast}-\hat{\gamma}_{\tau,k}\right)
=T^{-1/2}\sum_{t=k+1}^{T}\left\{  1[\mathbf{y}_{t,k}^{\ast}\leq\hat
{\mathbf{q}}_{t,k}^{\ast}(\tau)]-1[\mathbf{y}_{t,k}\leq\hat{\mathbf{q}}%
_{t,k}(\tau)]\right\}  +o_{p}(1),
\]
uniformly in $\tau_{i}\in\mathcal{T}_{i}$ and $\tau\in\mathcal{T}$,
respectively. As in Proposition \ref{Proposition:u-l-approx}, we can show
\[
\frac{1}{\sqrt{T}}\sum_{t=k+1}^{T}\big \{1[\mathbf{y}_{t,k}\leq\hat
{\mathbf{q}}_{t,k}(\tau)]-E\big [F_{\mathbf{y}|\mathbf{x}}^{(k)}%
(\mathbf{q}_{t,k}(\tau)\mathbf{x}_{t,k})\big ]\big \}=\mathbb{V}_{T,k}%
(\tau)+\nabla G^{(k)}(\tau)^{\top}\sqrt{T}\{\hat{\beta}(\tau)-\beta
(\tau)\}+o_{P}(1),
\]
uniformly in $\tau\in\mathcal{T}$. A similar argument used in Proposition
\ref{Proposition:u-l-approx} together with Proposition
\ref{Proposition:approx-boot} and \ref{Proposition:approx} yields that,
uniformly in $\tau\in\mathcal{T}$,
\begin{align*}
\frac{1}{\sqrt{T}}\sum_{t=k+1}^{T}\big \{1[\mathbf{y}_{t,k}^{\ast}\leq
\hat{\mathbf{q}}_{t,k}^{\ast}(\tau)]-E\left[  F_{\mathbf{y}|\mathbf{x}}%
^{(k)}(\mathbf{q}_{t,k}(\tau)|\mathbf{x}_{t,k})\right]  \big \}  &
=\mathbb{V}_{T,k}(\tau)^{\ast}+\mathbb{V}_{T,k}(\tau)\\
&  +\nabla G^{(k)}(\tau)^{\top}\sqrt{T}\{\hat{\beta}^{\ast}(\tau)-\beta
(\tau)\}+o_{p}(1).
\end{align*}
It follows that $\sqrt{T}(\hat{\gamma}_{\tau,k}^{\ast}-\hat{\gamma}_{\tau
,k})=\mathbb{V}_{T,k}^{\ast}(\tau)+\nabla G^{(k)}(\tau)^{\top}\sqrt{T}%
\{\hat{\beta}^{\ast}(\tau)-\hat{\beta}(\tau)\}+o_{p}(1)$ uniformly in $\tau
\in\mathcal{T}$. Thus, we obtained the desired result.
\end{proof}

%%%% Proposition: linear approximation-bootstrap %%%%%%%%%%%%%%%%%%%%%%%%%%%%%%%%%

%%%%%%%%%%%%%%%%%%%%%%%%%%%%%%%%%%%%%%%%%%%%%%%%%%%%%%%%%%%%%%%%%%%%%%%%%%%%%%%%%%

%%%% Theorem: Limiting Behavior for One dimensional %%%%%%%%%%%%%%%%%%%%%%%%%%%%%
\vspace{0.5cm}

\begin{proof}
[\textbf{Proof of Theorem \ref{theorem:boostrap validity}}]\noindent
\textbf{(a)} Define the processes $\hat{\mathbb{G}}_{T}^{(p)\ast}(\tau
):=\sqrt{T}(\hat{\rho}_{\tau}^{\ast(p)}-\hat{\rho}_{\tau}^{(p)})$ and
$\mathbb{G}^{(p)}(\tau):=\Lambda_{\tau}^{(p)}\mathbb{B}^{(p)}(\tau)$ for
$\tau\in\mathcal{T}$ and for an integer $p>0$. Let $\widetilde{BL}_{1}$ denote
the set of all Lipschitz continuous, real-valued functions on $(\ell^{\infty
}(\mathcal{T}))^{p}$ with a Lipschitz constant bounded by 1. It suffices to
show that
\[
\sup_{h\in\widetilde{BL}_{1}}\big |E^{\ast}\big [h(\mathbb{G}_{T}^{(p)\ast
})\big ]-E\big [h(\mathbb{G}^{(p)})\big ]\big |\rightarrow^{p}0.
\]
Let $\mathbb{G}_{T}^{(p)\ast}(\tau):=\Lambda_{\tau}^{(p)}\mathbb{B}%
_{T}^{(p)\ast}(\tau)$. We can write
\begin{align*}
\sup_{h\in\widetilde{BL}_{1}}\big |E^{\ast}\big [h(\hat{\mathbb{G}}%
_{T}^{(p)\ast})\big ]-E\big [h(\mathbb{G}^{(p)})\big ]\big |  &  \leq
\sup_{h\in\widetilde{BL}_{1}}\big |E^{\ast}\big [h(\hat{\mathbb{G}}%
_{T}^{(p)\ast})\big ]-E^{\ast}\big [h(\mathbb{G}_{T}^{(p)\ast})\big ]\big |\\
&  \text{ \ \ }+\sup_{h\in\widetilde{BL}_{1}}\big |E^{\ast}\big [h(\mathbb{G}%
_{T}^{(p)\ast})\big ]-E\big [h(\mathbb{G}^{(p)})\big ]\big |.
\end{align*}
Propositions \ref{Proposition:uniform-bahadur} and
\ref{Proposition:uniform-bahadur-bootstrap} imply that $\sqrt{T}\{\hat{\beta
}_{i}^{\ast}(\tau_{i})-\hat{\beta}_{i}(\tau_{i})\}=-D_{i}^{-1}(\tau
_{i})\mathbb{W}_{1,T}^{\ast}(\tau_{i})+o_{p}(1)$ uniformly in $\tau_{i}%
\in\mathcal{T}_{i}$ for each $i=1,2$. It follows from Proposition
\ref{Proposition:approx-boot} that
\[
\sqrt{T}\left\{  \hat{\rho}_{\tau}^{\ast}(k)-\hat{\rho}_{\tau}(k)\right\}
=\lambda_{\tau,k}^{\top}\mathbb{B}_{T,k}^{\ast}(\tau)+o_{p}(1),
\]
uniformly in $\tau\in\mathcal{T}$, where $\lambda_{\tau,k}$ is defined in
(\ref{eq:lambda-def}). This leads to
\[
\sup_{\tau\in\mathcal{T}}\big \|\hat{\mathbb{G}}_{T}^{(p)\ast}(\tau
)-\mathbb{G}_{T}^{(p)\ast}(\tau)\big \|=o_{p}(1).
\]
This implies that $\sup_{h\in\widetilde{BL}_{1}}|h(\hat{\mathbb{G}}%
_{T}^{(p)\ast})-h(\mathbb{G}_{T}^{(p)\ast})|=o_{p}(1)$, because $\sup
_{h\in\widetilde{BL}_{1}}|h(\hat{\mathbb{G}}_{T}^{(p)\ast})-h(\mathbb{G}%
_{T}^{(p)\ast})|\leq\sup_{\tau\in\mathcal{T}}\Vert\hat{\mathbb{G}}%
_{T}^{(p)\ast}(\tau)-\mathbb{G}_{T}^{(p)\ast}(\tau)\Vert$. It follows from the
dominated convergence theorem that $\lim_{T\rightarrow\infty}E\sup
_{h\in\widetilde{BL}_{1}}|E^{\ast}[h(\hat{\mathbb{G}}_{T}^{(p)\ast})]-E^{\ast
}[h(\mathbb{G}_{T}^{(p)\ast})]|=0$. An application of the Markov inequality
shows that $\sup_{h\in\widetilde{BL}_{1}}|E^{\ast}[h(\hat{\mathbb{G}}%
_{T}^{(p)\ast})]-E^{\ast}[h(\mathbb{G}_{T}^{(p)\ast})]|=o_{p}(1).$

%% Weak convergence
Under Assumption A4 and A5, $\Lambda_{\tau}^{(p)}$ is bounded uniformly in
$\tau\in\mathcal{T}$, we have
\begin{align*}
\sup_{h \in\widetilde{BL}_{1}} \big | E^{\ast} \big [ h ( \mathbb{G}_{T}^{(p)
\ast} ) \big ] - E \big [ h ( \mathbb{G}^{(p)} ) \big ] \big | \le C_{1}
\sup_{g \in BL_{1}} \big | E^{\ast} \big [ g ( \mathbb{B}_{T}^{(p) \ast} )
\big ] - E \big [ g ( \mathbb{B}^{(p)} ) \big ] \big |,
\end{align*}
where the right-hand side is negligible in probability from Proposition
\ref{Proposition:weak convergence-boot}. Hence, we obtain the desired result.

\textbf{(b)} From the continuous mapping theorem, the result in (a) of this
theorem yields the desired result. See Theorem 10.8 of Kosorok (2007) for a
general argument.
\end{proof}

%%%% END:: Theorem: Limiting Behavior for One dimensional %%%%%%%%%%%%%%%%%%%%%%%%

%%%%%%%%%%%%%%%%%%%%%%%%%%%%%%%%%%%%%%%%%%%%%%%%%%%%%%%%%%%%%%%%%%%%%%%%%%%%%%%%%%

%%%% Theorem: Alternative %%%%%%%%%%%%%%%%%%%%%%%%%%%%%%%%%%%%%%%%%%%%%%%%%%%%%%%%
\vspace{0.5cm}

\begin{proof}
[\textbf{Proof of Theorem \ref{theorem:alternative}}]As shown in
(\ref{eq:expansion-p}), under both fixed and local alternatives,
\[
\sqrt{T}\left(  \hat{\rho}_{\tau}^{(p)}-\rho_{\tau}^{(p)}\right)
=\Lambda_{\tau}^{(p)}\mathbb{B}_{T}^{(p)}(\tau)+o_{p}(1)
\]
uniformly in $\tau\in\mathcal{T}$, and it follows from Theorem 1 that
$\Lambda_{\tau}^{(p)}\mathbb{B}_{T}^{(p)}(\tau)=O_{P}(1)$ uniformly in
$\tau\in\mathcal{T}$.

\textbf{(a)} Under the fixed alternative, there is some $\tau\in\mathcal{T} $
such that $\rho_{\tau}^{(p)}$ is some non-zero constant and then $\sqrt{T}%
\hat{\rho}_{\tau}^{(p)}=\Lambda_{\tau}^{(p)} \mathbb{B}_{T}^{(p)}(\tau)
+\sqrt{T}\rho_{\tau}^{(p)}+o_{p}(1) $ uniformly in $\tau\in\mathcal{T}$. This
implies that, under the fixed alternative, $\sup_{\tau\in\mathcal{T}}\hat
{Q}_{\tau}^{(p)}=T\sup_{\tau\in\mathcal{T}}\Vert\rho_{\tau}^{(p)}\Vert^{2}(
1+o_{p}(1)) $. Thus, $\sup_{\tau\in\mathcal{T}}\hat{Q}_{\tau}^{(p)}%
\rightarrow^{p}\infty$ under the fixed alternative, whereas the critical value
$c_{Q,\tau}^{\ast}$ is bounded in probability from Theorem 2. Therefore,
$\lim_{T\rightarrow\infty}P(\sup_{\tau\in\mathcal{T}}\hat{Q}_{\tau}%
^{(p)}>c_{Q,\tau}^{\ast})=1$. Therefore, our test is shown to be consistent
under the fixed alternative.

\textbf{(b)} Under the local alternative, we can write $\rho_{\tau}%
^{(p)}=\zeta_{\tau}^{(p)}/\sqrt{T}$, where $\zeta_{\tau}^{(p)}$ is a
$p$-dimensional constant vector, at least one of elements is non-zero. Thus,
we have
\[
\hat{Q}_{\tau}^{(p)}=\Vert\Lambda_{\tau}^{(p)}\mathbb{B}_{T}^{(p)}(\tau
)+\zeta_{\tau}^{(p)}\Vert^{2}+o_{P}(1),
\]
uniformly in $\tau\in\mathcal{T}$. From Theorem 1 and the continuous mapping
theorem,
\[
\sup_{\tau\in\mathcal{T}}\hat{Q}_{\tau}^{(p)}\Rightarrow\sup_{\tau
\in\mathcal{T}}\Vert\Lambda_{\tau}^{(p)}\mathbb{B}^{(p)}(\tau)+\zeta_{\tau
}^{(p)}\Vert^{2}.
\]
Also, Theorem 2 implies $\sup_{\tau\in\mathcal{T}}\hat{Q}_{\tau}^{(p)\ast
}\Rightarrow^{\ast}\sup_{\tau\in\mathcal{T}}\Vert\Lambda_{\tau}^{(p)}%
\mathbb{B}^{(p)}(\tau)\Vert^{2} $ in probability. Thus, the desired result follows.
\end{proof}

%%%% END Theorem: Alternative %%%%%%%%%%%%%%%%%%%%%%%%%%%%%%%%%%%%%%%%%%%%%%%

%%%%%%%%%%%%%%%%%%%%%%%%%%%%%%%%%%%%%%%%%%%%%%%%%%%%%%%%%%%%%%%%%%%%%%%%%%%%%

\bigskip
%%%%%%%%%%%%%%%%%%%%%%%%%%%%%%%%%%%%%%%%%%%%%%%%%%%%%%%%%%%%%%%%%%%%%%%%%%%%%
%%%%%%%%%%%%%%%%%%%%%%%%%%%%%
%% setting for appendix: C
%%%%%%%%%%%%%%%%%%%%%%%%%%%%%
\renewcommand{\theLemma}{C\arabic{Lemma}}
\renewcommand{\theProposition}{C\arabic{Proposition}} \setcounter{Lemma}{0} \setcounter{Proposition}{0}

\vspace{0.5cm} \noindent\textbf{Appendix C. Self-Normalized
Cross-Quantilogram}

%%%%%%%%%%%%%%%%%%%%%%%%%%%%%%%%%%%%%%%%%%%%%%%%%%%%%%%%%%%%%%%%%%%%%%%%%%%%%%

%%%%%%%%%%%%%%%%%%%%%%%%%%%%%%%%%%%%%%%%%%%%%%%%%%%%%%%%%%%%%%%%%%%%%%%%%%%%%%

%%%%%%%%%%%%%%%%%%%%%%%%%%%%%%%%%%%%%%%%%%%%%%%%%%%%%%%%%%%%%%%%%%%%%%%%%%%%%%
\vspace{0.5cm}

\begin{Lemma}
\label{lemma:basic-ineq-C} Let $\{z_{t}\}_{t\in\mathds{Z}}$ be a strict
stationary, strong mixing sequence of $\mathds{R}^{d}$-valued random variables
for some integer $d\geq1$ with strong mixing coefficients $\{\alpha
_{j}\}_{j\in\mathds{Z}_{+}}$ satisfying $\sum_{j=0}^{\infty}(j+1)^{2s-2}%
\alpha_{j}^{\nu/(2s+\nu)}$ for some integer $s\geq2$ and $\nu\in(0,1)$.
Suppose that $E[z_{1}]=0$ and $\Vert z_{1}\Vert_{2s+\nu}<\infty$. Then,
\[
E\bigg [\sup_{r\in\lbrack0,1]}\Big \|\sum_{t=1}^{[Tr]}z_{t}\Big \|^{2s}%
\bigg ]\leq CT^{s}\big (\big \|z_{1}\big \|_{2+\nu}^{2s}+T^{1-s}%
\big \|z_{1}\big \|_{2s+\nu}\big ).
\]

\end{Lemma}

\begin{proof}
The desired result follows from Theorem 6.3 and Annexes C of Rio (2013) as in
Lemma \ref{lemma:basic-ineq-A}.
\end{proof}

\vspace{0.5cm}
%%%%%%%%%%%%%%%%%%%%%%%%%%%%%%%%%%%%%%%%%%%%%%%%%%%%%%%%%%%%%%%%%%%%%%%%%%%%%%

%%%%%%%%%%%%%%%%%%%%%%%%%%%%%%%%%%%%%%%%%%%%%%%%%%%%%%%%%%%%%%%%%%%%%%%%%%%%%%

We define the process indexed by $r\in\lbrack0,1]$
\[
\bar{\mathbb{V}}_{T,k,\tau}(r):=\frac{1}{\sqrt{T}}\sum_{t=k+1}^{[Tr]}\left\{
1[\mathbf{y}_{t,k}\leq\mathbf{q}_{t,k}(\tau)]-E[F_{\mathbf{y}|\mathbf{x}%
}^{(k)}(\mathbf{q}_{t,k}(\tau)|\mathbf{x}_{t,k})]\right\}  ,
\]
and
\[
\bar{\mathbb{W}}_{i,T,\tau_{i}}(r):=\frac{1}{\sqrt{T}}\sum_{t=1}^{[Tr]}%
x_{it}\left\{  1[y_{it}\leq q_{i,t}(\tau_{i})]-\tau_{i}\right\}  ,
\]
for each $i=1,2$. The following proposition shows the stochastic
equicontinuity of the processes defined above.

%%%%%%%%%%%% Lemma: Stochastic equicontinuity %%%%%%%%%%%%%%%

\begin{Proposition}
\label{Proposition:cont-self} Suppose Assumption A1-A5 hold. Let $k \in\{1,
\dots, p\}$ and define metrics $\bar{\rho}(r, r^{\prime})=|r^{\prime}- r|$ for
$r,r^{\prime}\in[0,1]$. Then,

\begin{description}
\item[(a)] $\bar{\mathbb{V}}_{T,k,\tau}(r)$ is stochastically equicontinuous
on $([0,1],\bar{\rho} )$.

\item[(b)] $\bar{\mathbb{W}}_{i,T,\tau_{i}}(r)$ is stochastically
equicontinuous on $([0,1],\bar{\rho})$ for each $i = 1,2$.
\end{description}
\end{Proposition}

\begin{proof}
See Supplemental Material.
\end{proof}

%%%% END: Lemma: equicont %%%%%%%%%%%%%%%%%%%%%%%%%%%%%%%%%%%%%%%%%%%%%%%%%%%%%%

%%%%%%%%%%%%%%%%%%%%%%%%%%%%%%%%%%%%%%%%%%%%%%%%%%%%%%%%%%%%%%%%%%%%%%%%%%%%%%%%
\vspace{0.5cm}

%%%% Lemma; Weak Convergence Self %%%%%%%%%%%%%%%%%%%%%%%%%%%%%%%%%%%%%%%%%%%%%%
Define a $d_{0} \times1$ vector $\bar{\mathbb{B}}_{T,k, \tau}(r) = [
\bar{\mathbb{V}}_{T,k,\tau}(r), \bar{\mathbb{W}}_{1, T,\tau_{1}}(r)^{\top},
\bar{\mathbb{W}}_{2, T,\tau_{2}}(r)^{\top}]^{\top}. $ for $r \in[0,1]$ and $k
= 1, \dots, p$. The following proposition shows the weak convergence of the
process $\{\bar{\mathbb{B}}_{T,k,\tau}(r): r \in[0,1] \}_{k=1}^{p}$.

\begin{Proposition}
\label{Proposition:weak convergence-self} Suppose Assumptions A1-A5 hold.
Then,
\begin{align*}
\left[  \bar{\mathbb{B}}_{T,1, \tau}(\cdot), \dots, \bar{\mathbb{B}}_{T,p,
\tau}(\cdot) \right]  ^{\top} \Rightarrow\left[  \bar{\mathbb{B}}_{1,\tau
}(\cdot), \dots, \bar{\mathbb{B}}_{p,\tau}(\cdot) \right]  ^{\top}.
\end{align*}

\end{Proposition}

\begin{proof}
Proposition \ref{Proposition:cont-self} establishes the stochastic
equicontinuity of $\left[  \bar{\mathbb{B}}_{T,1,\tau}(\cdot),\dots
,\bar{\mathbb{B}}_{T,p,\tau}(\cdot)\right]  ^{\top}$ and it suffices to show
convergence of the finite dimensional distributions. Since the finite
dimensional convergences can be shown by a similar argument used in
Proposition \ref{Proposition:weak convergence}, we omit the details.
\end{proof}

\vspace{0.5cm}

%%%%%%%%%%%%%%% END Proof: Bootstrap Consistency - Numerator %%%%%%%%%%%%%%%%%

%%%%%%%%%%%%%%%%%%%%%%%%%%%%%%%%%%%%%%%%%%%%%%%%%%%%%%%%%%%%%%%%%%%%%%%%%%%%%%

%%%%%%%%%%%%%%%%%%%%%%%%%%%%%%%%%%%%%%%%%%%%%%%%%%%%%%%%%%%%%%%%%%%%%%%%%%%%%%
%% H(a,v) function
For $\mathbf{v}=(v_{1},v_{2})\in\mathds{R}^{d_{1}}\times\mathds{R}^{d_{2}}$,
we define
\[
\bar{\mathbb{V}}_{T,k,\tau}(r,\mathbf{v}):=\frac{1}{\sqrt{T}}\sum
_{t=k+1}^{[Tr]}\left\{  1[\mathbf{y}_{t,k}\leq\mathbf{q}_{t,k}(\tau
)+T^{-1/2}\mathbf{v}_{t,k}]-E\left[  F_{\mathbf{y}|\mathbf{x}}^{(k)}\left(
\mathbf{q}_{t,k}(\tau)+T^{-1/2}\mathbf{v}_{t,k}|\mathbf{x}_{t,k}\right)
\right]  \right\}  ,
\]
and
\[
\bar{\mathbb{W}}_{i,T,\tau_{i}}(r,v_{i}):=\frac{1}{\sqrt{T}}\sum
_{t=k+1}^{[Tr]}x_{it}\left\{  1[y_{it}\leq q_{i,t}(\tau_{i})+T^{-1/2}%
v_{i,t}]-F_{y_{i}|x_{i}}\left(  q_{i,t}(\tau_{i})+T^{-1/2}v_{i,t}%
|x_{it}\right)  \right\}  .
\]

%%%% Lemma: Approximation %%%%%%%%%%%%%%%%%%%%%%%%%%%%%%%%%%%%

%%%%%%%%%%%%%%%%%%%%%%%%%%%%%%%%%%%%%%%%%%%%%%%%%%%%%%%%%%%%%%%%%%%%%%%%%%%%%%

\begin{Proposition}
\label{Proposition:approx-self} Suppose Assumption A1-A5 hold. Then,

\begin{description}
\item[(a)] $\sup_{\omega\le r \le1} \sup_{\mathbf{v} \in\mathcal{V}_{M}} |
\bar{\mathbb{V}}_{T,k,\tau}(r, \mathbf{v}) - \bar{\mathbb{V}}_{T,k,\tau}(r) |
= o_{p}(1) $ for every $M>0$;

\item[(b)] $\sup_{\omega\le r \le1} \sup_{v_{i} \in\mathcal{V}_{i, M}} \|
\bar{\mathbb{W}}_{i,T,\tau_{i}}(r, v_{i}) - \bar{\mathbb{W}}_{i, T,\tau_{i}%
}(r) \| = o_{p}(1) $ for every $M>0$ and for $i = 1,2$,
\end{description}
\end{Proposition}

\noindent\textit{where }$\mathcal{V}_{M}=\mathcal{V}_{1,M}\times
\mathcal{V}_{2,M}$\textit{\ with }$\mathcal{V}_{i,M}=\{v_{i}\in R^{d_{i}%
}:\Vert v_{i}\Vert\leq M\}$\textit{\ for }$i=1,2$\textit{.}\vspace{0.3cm}

\begin{proof}
See Supplemental Material.
\end{proof}

%%%%%%%%%%%%%%%%%%%%%%%%%%%%%%%%%%%%%%%%%%%%%%%%%%%%%%%%%%%%%%%%%%%%%%%%%%%%%%

%%%%%%%%%%%%%%%%%%%%%%%%%%%%%%%%%%%%%%%%%%%%%%%%%%%%%%%%%%%%%%%%%%%%%%%%%%%%%%

%%%%%%%%%%%%%%% END - Lemma: approximation    %%%%%%%%%%%%%%%%%%%%%%%%%%%%%%%%%

%%%% Theorem: Argmin for Subsample  %%%%%%%%%%%%%%%%%%%%%%%%%%%%%%%%%%%%%%%%%%%

\begin{Proposition}
\label{Proposition:uniform-bahadur-subsample} Suppose Assumption A1-A5 hold.
Then, for $i=1,2$ and for each $\tau_{i}\in\mathcal{T}_{i}$,
\[
\sqrt{T}\{\hat{\beta}_{i,[Tr]}(\tau_{i})-\beta_{i}(\tau_{i})\}=-D_{i}%
^{-1}(\tau_{i})r^{-1}\bar{\mathbb{W}}_{i,T,\tau_{i}}(r)+o_{p}(1),
\]
uniformly in $r\in\lbrack\omega,1]$.
\end{Proposition}

\begin{proof}
The proof follows the line of Proposition \ref{Proposition:uniform-bahadur}
with Proposition \ref{Proposition:approx-self}(b). Hence, we omit the details.
\end{proof}

%%%% Theorem: Argmin for Subsample  %%%%%%%%%%%%%%%%%%%%%%%%%%%%%%%%%%%%%%%%%

%%%%%%%%%%%%%%%%%%%%%%%%%%%%%%%%%%%%%%%%%%%%%%%%%%%%%%%%%%%%%%%%%%%%%%%%%%%%%%%

%%%% Proposition: uniform linear approximation %%%%%%%%%%%%%%%%%%%%%%%%%%%%%%%%
\vspace{0.5cm}

\begin{Proposition}
\label{Proposition:u-l-approx-self} Suppose Assumption A1-A5 hold. Then, for
each $(k,\tau)\in\{1,\dots,p\}\times\mathcal{T}$,
\[
\sqrt{T}\left\{  \hat{\rho}_{\tau,[Tr]}(k)-\rho_{\tau}(k)\right\}
=\frac{r^{-1}\bar{\mathbb{V}}_{T,k,\tau}(r)+\nabla G^{(k)}(\tau)^{\top}%
\sqrt{T}\{\hat{\beta}_{[Tr]}(\tau)-\beta(\tau)\}}{\sqrt{\tau_{1}(1-\tau
_{1})\tau_{2}(1-\tau_{2})}}+o_{p}(1),
\]
uniformly in $r\in\lbrack\omega,1]$, where $\hat{\beta}_{[Tr]}=(\hat{\beta
}_{1,[Tr]}^{\top},\hat{\beta}_{2,[Tr]}^{\top})^{\top}.$
\end{Proposition}

\begin{proof}
A similar argument used in Proposition \ref{Proposition:u-l-approx} with
Proposition \ref{Proposition:approx-self}(a) yields the desired result and
thus we omit the detail.
\end{proof}

%%%% Proposition: uniform linear approximation %%%%%%%%%%%%%%%%%%%%%%%%%%%%%%%%

%%%%%%%%%%%%%%%%%%%%%%%%%%%%%%%%%%%%%%%%%%%%%%%%%%%%%%%%%%%%%%%%%%%%%%%%%%%%%%%

%%%%%%%%%%%%%%%%%%%%%%%%%%%%%%%%%%%%%%%%%%%%%%%%%%%%%%%%%%%%%%%%%%%%%%%%%%%%%%
\vspace{0.5cm}

\begin{proof}
[\textbf{Proof of Theorem \ref{theorem:self-norm}}]Proposition
\ref{Proposition:uniform-bahadur-subsample} and
\ref{Proposition:u-l-approx-self} imply that, for each $(k,\tau)\in
\{1,\dots,p\}\times\mathcal{T}$,
\[
\sqrt{T}\left\{  \hat{\rho}_{\tau,[Tr]}(k)-\rho_{\tau}(k)\right\}
=r^{-1}\lambda_{\tau,k}^{\top}\bar{\mathbb{B}}_{T,k,\tau}(r)+o_{p}(1),
\]
uniformly in $r\in\lbrack\omega,1]$. It follows that $\sqrt{T}(\hat{\rho
}_{\tau,[Tr]}^{(p)}-\rho_{\tau}^{(p)})=r^{-1}\Lambda_{\tau}^{(p)}%
\bar{\mathbb{B}}_{T,\tau}^{(p)}(r)+o_{p}(1)$ uniformly in $r\in\lbrack
\omega,1]$. This implies
\[
\frac{\lbrack Tr]}{\sqrt{T}}\left(  \hat{\rho}_{\tau,[Tr]}^{(p)}-\hat{\rho
}_{\tau,T}^{(p)}\right)  =\Lambda_{\tau}^{(p)}\left\{  \bar{\mathbb{B}%
}_{T,\tau}^{(p)}(r)-r\bar{\mathbb{B}}_{T,\tau}^{(p)}(1)\right\}  +o_{p}(1),
\]
uniformly in $r\in\lbrack\omega,1]$. From Proposition
\ref{Proposition:weak convergence-self}, $\{\Lambda_{\tau}^{(p)}%
(\bar{\mathbb{B}}_{T,\tau}^{(p)}(r)-r\bar{\mathbb{B}}_{T,\tau}^{(p)}%
(1)):r\in\lbrack\omega,1]\}$ weakly converges to $\{\Lambda_{\tau}^{(p)}%
(\bar{\mathbb{B}}_{\tau}^{(p)}(r)-r\bar{\mathbb{B}}_{\tau}^{(p)}%
(1)):r\in\lbrack\omega,1]\},$ which is equivalent in distribution to a
$p\times1$ vector of the Brownian bridge process $\{\Delta_{\tau}^{(p)}%
(\bar{\mathbf{B}}^{(p)}(r)-r\bar{\mathbf{B}}^{(p)}(1)):r\in\lbrack\omega,1]\}$
with $\Delta_{\tau}^{(p)}(\Delta_{\tau}^{(p)})^{\top}\equiv\Xi^{(p)}(\tau
,\tau),$and thus it follows from the continuous mapping theorem that
\[
\big (\sqrt{T}\hat{\rho}_{\tau,T}^{(p)},\hat{V}_{\tau,p}\big )\rightarrow
^{d}\big (\Delta_{\tau}^{(p)}\bar{\mathbf{B}}^{(p)}(1),\Delta_{\tau}^{(p)}%
\bar{\mathbf{V}}^{(p)}(\Delta_{\tau}^{(p)})^{\top}\big ).
\]
Thus, we obtain $\hat{S}_{\tau}^{(p)}\rightarrow^{d}\bar{\mathbf{B}}%
^{(p)}\left(  1\right)  ^{\top}(\bar{\mathbf{V}}^{(p)})^{-1}\bar{\mathbf{B}%
}^{(p)}(1)$. This completes the proof.
\end{proof}

\vspace{0.5cm}

\begin{proof}
[\textbf{Proof of Theorem \ref{theorem:self-norm-power}}]Under both fixed and
local alternative, the argument used in Theorem 4 gives
\[
\sqrt{T}\left(  \hat{\rho}_{\tau,[Tr]}^{(p)}-\rho_{\tau}^{(p)}\right)
=r^{-1}\Lambda_{\tau}^{(p)}\bar{\mathbb{B}}_{T,\tau}^{(p)}(r)+o_{p}(1),
\]
thereby yielding $\hat{V}_{\tau, p}\Rightarrow(\Lambda_{\tau}^{(p)}%
\Delta_{\tau}^{(p)})\bar{\mathbf{V}}^{(p)}(\Lambda_{\tau}^{(p)}\Delta_{\tau
}^{(p)})^{^{\top}}$.

(a) Under the fixed alternative, we have $\sqrt{T}\hat{\rho}_{\tau,T}%
^{(p)}=\Lambda_{\tau}^{(p)} \bar{\mathbb{B}}_{T,\tau}^{(p)}(1) +\sqrt{T}%
\rho_{\tau}^{(p)}+o_{p}(1), $ where the right-hand side diverges in
probability as $T\rightarrow\infty$. Since the critical value we use is finite
in probability from Theorem 4, we obtain the desired result.

(b) Under the local alternative, $\sqrt{T}\hat{\rho}_{\tau,T}^{(p)}%
=\Lambda_{\tau}^{(p)}\bar{\mathbb{B}}_{T,\tau}^{(p)}(1)+\xi_{\tau}^{(p)}%
+o_{p}(1). $ It follows that
\[
\hat{S}_{\tau}^{(p)}\rightarrow^{d}\left\{  \bar{\mathbf{B}}^{(p)}%
(1)+(\Lambda_{\tau}^{(p)}\Delta_{\tau}^{(p)})^{-1}\xi_{\tau}^{(p)}\right\}
^{^{\top}}\left(  \bar{\mathbf{V}}^{(p)}\right)  ^{-1}\left\{  \bar
{\mathbf{B}}^{(p)}(1)+(\Lambda_{\tau}^{(p)}\Delta_{\tau}^{(p)})^{-1}\xi_{\tau
}^{(p)}\right\}  .
\]
This completes the proof.
\end{proof}

%%%%%%%%%%%%%%%%%%%%%%%%%%%%%%%%%%%%%%%%%%%%%%%%%%%%%%%%%%%%%%%%%%%%%%%%%%%%%%%%

%%%%%%%%%%%%%%%%%%%%%%%%%%%%%%%%%%%%%%%%%%%%%%%%%%%%%%%%%%%%%%%%%%%%%%%%%%%%%%%%

\vspace{0.5cm}

%%%%%%%%%%%%%%%%%%%%%%%%%%%%%
%% setting for appendix: D
%%%%%%%%%%%%%%%%%%%%%%%%%%%%%
\renewcommand{\theLemma}{A\arabic{Lemma}}
\renewcommand{\theProposition}{A\arabic{Proposition}} \setcounter{Lemma}{0}

\vspace{0.5cm}

\noindent\textbf{Appendix D. Partial Cross-Quantilogram}

\vspace{0.5cm}

%%%% Partial Cross-Quantilogram %%%%%%%%%%%%%%%%%%%%%%%%%%%%%%%%%%%%%%%%%%%%%%%%
For $1\leq i,j\leq l$, let $\mathbf{1}_{ij}=1[y_{it}\leq q_{i,t}(\tau
_{i}),y_{jt}\leq q_{j,t}(\tau_{j})]$ and define
\[
\mathbb{V}_{T,ij}=\frac{1}{\sqrt{T}}\sum_{t=1}^{T}\left(  \mathbf{1}%
_{ij}-E\left[  \mathbf{1}_{ij}\right]  \right)  \ \ \mathrm{and}%
\ \ \mathbb{W}_{i,T}=\frac{1}{\sqrt{T}}\sum_{t=1}^{T}x_{it}\psi_{\tau_{i}%
}\left(  y_{it}-q_{i,t}(\tau_{i})\right)  .
\]

\vspace{0.5cm}

\begin{proof}
[\textbf{Proof of Theorem \ref{theorem:pcq}}]We first consider (a). The
correlation matrix $R_{\bar{\tau}}$ is symmetric and $\hat{R}_{\bar{\tau}%
}^{-1}-R_{\bar{\tau}}^{-1}=-\hat{R}_{\bar{\tau}}^{-1}(\hat{R}_{\bar{\tau}%
}-R_{\bar{\tau}})R_{\bar{\tau}}^{-1}$. It follows that $\mathrm{vec}(\hat
{P}_{\bar{\tau}}-P_{\bar{\tau}})=-P_{\bar{\tau}}\otimes\hat{P}_{\bar{\tau}%
}\mathrm{vec}(\hat{R}_{\bar{\tau}}-R_{\bar{\tau}})$, which implies
\[
\sqrt{T}(\hat{p}_{\bar{\tau},12}-p_{\bar{\tau},12})=-\sum_{i=1}^{l}\sum
_{j=1}^{l}p_{\bar{\tau},1i}\hat{p}_{\bar{\tau},2j}\sqrt{T}(\hat{r}_{\bar{\tau
},ij}-r_{\bar{\tau},ij}).
\]
Following the line of proof of Theorem 1, we can show $\hat{P}_{\bar{\tau}%
}=P_{\bar{\tau}}+o_{p}(1)$ and also have $\sqrt{T}(\hat{r}_{\bar{\tau}%
,ii}-r_{\bar{\tau},ii})=o_{p}(1)$ for $i=1,\dots,l$, from argument in Lemma
2.1 of Arcones (1998). Thus, we have
\[
\sqrt{T}(\hat{p}_{\bar{\tau},12}-p_{\bar{\tau},12})=-\sum_{\substack{1\leq
i,j\leq l\\i\not =j}}p_{\bar{\tau},1i}p_{\bar{\tau},2j}\sqrt{T}(\hat{r}%
_{\bar{\tau},ij}-r_{\bar{\tau},ij})+o_{p}(1).
\]
Proposition \ref{Proposition:u-l-approx} implies
\begin{align*}
\sqrt{T}(\hat{r}_{\bar{\tau},ij}-r_{\bar{\tau},ij})  &  =\mathbb{V}%
_{T,ij}+\nabla_{1}G_{ij}^{\top}\sqrt{T}\{\hat{\beta}_{i}(\tau_{i})-\beta
_{i}(\tau_{i})\}\\
&  \ \ \ +\nabla_{2}G_{ij}^{\top}\sqrt{T}\{\hat{\beta}_{j}(\tau_{j})-\beta
_{j}(\tau_{j})\}+o_{p}(1),
\end{align*}
for $1\leq i,j\leq l$ with $i\not =j$. Since $\mathbb{V}_{T,ij}=\mathbb{V}%
_{T,ji}$ and $\nabla_{2}G_{ij}=\nabla_{1}G_{ji}$ for $1\leq i,j\leq l$,
\[
\sqrt{T}(\hat{p}_{\bar{\tau},12}-p_{\bar{\tau},12})=-\sum_{\substack{1\leq
i,j\leq l\\i\not =j}}p_{\bar{\tau},1i}p_{\bar{\tau},2j}\mathbb{V}_{T,ij}%
-\sum_{i=1}^{l}\lambda_{\bar{\tau}i}^{\top}\sqrt{T}\{\hat{\beta}_{i}(\tau
_{i})-\beta_{i}(\tau_{i})\}+o_{p}(1),
\]
where $\lambda_{\bar{\tau}i}$ is defined in Theorem \ref{theorem:pcq}.
Proposition \ref{Proposition:uniform-bahadur} implies
\[
\sqrt{T}(\hat{p}_{\bar{\tau},12}-p_{\bar{\tau},12})=-\sum_{\substack{1\leq
i,j\leq l\\i\not =j}}p_{\bar{\tau},1i}p_{\bar{\tau},2j}\mathbb{V}_{T,ij}%
+\sum_{i=1}^{l}\lambda_{\bar{\tau}i}^{\top}D_{i}(\tau_{i})^{-1}\mathbb{W}%
_{i,T}+o_{p}(1).
\]
The asymptotic normality can be established by using the central limit theorem
for mixing random vectors. The proofs of (b) and (c) are similar to those of
Theorems 2 and 4, respectively, and thus we omit the details.
\end{proof}

\pagebreak

\noindent\textbf{Appendix E. Tables and Figures}

\begin{center}
\singlespacing Table 1. (size) Empirical rejection frequency of the Box-Ljung
test statistic $\hat{Q}_{\tau}^{(p)}$ based on the bootstrap procedure

(VAR\ with DGP1 and the nominal level 5\%)

%Table generated by Excel2LaTeX from sheet 'resultGamma01'%
\begin{tabular}
[c]{ccccccccccc}\hline
&  & \multicolumn{9}{c}{Quantiles $(\tau_{1}=\tau_{2})$}\\\hline
$T$ & $p$ & 0.05 & 0.10 & 0.20 & 0.30 & 0.50 & 0.70 & 0.80 & 0.90 &
0.95\\\hline\hline
500 & 1 & 0.051 & 0.025 & 0.037 & 0.045 & 0.040 & 0.043 & 0.043 & 0.033 &
0.047\\
& 2 & 0.017 & 0.032 & 0.043 & 0.072 & 0.068 & 0.060 & 0.057 & 0.036 & 0.012\\
& 3 & 0.011 & 0.022 & 0.051 & 0.073 & 0.066 & 0.055 & 0.050 & 0.032 & 0.010\\
& 4 & 0.007 & 0.022 & 0.047 & 0.062 & 0.059 & 0.057 & 0.046 & 0.026 & 0.008\\
& 5 & 0.009 & 0.025 & 0.035 & 0.052 & 0.051 & 0.052 & 0.054 & 0.027 &
0.006\\\hline
1000 & 1 & 0.033 & 0.030 & 0.037 & 0.048 & 0.047 & 0.039 & 0.037 & 0.052 &
0.042\\
& 2 & 0.018 & 0.037 & 0.045 & 0.051 & 0.043 & 0.046 & 0.052 & 0.041 & 0.015\\
& 3 & 0.011 & 0.031 & 0.049 & 0.056 & 0.044 & 0.054 & 0.045 & 0.028 & 0.006\\
& 4 & 0.013 & 0.027 & 0.049 & 0.053 & 0.041 & 0.055 & 0.041 & 0.022 & 0.008\\
& 5 & 0.007 & 0.022 & 0.044 & 0.040 & 0.044 & 0.040 & 0.036 & 0.021 &
0.006\\\hline
2000 & 1 & 0.038 & 0.034 & 0.040 & 0.034 & 0.034 & 0.048 & 0.050 & 0.034 &
0.054\\
& 2 & 0.028 & 0.025 & 0.043 & 0.035 & 0.045 & 0.051 & 0.050 & 0.035 & 0.024\\
& 3 & 0.023 & 0.033 & 0.031 & 0.045 & 0.050 & 0.045 & 0.042 & 0.029 & 0.018\\
& 4 & 0.017 & 0.023 & 0.042 & 0.052 & 0.038 & 0.036 & 0.038 & 0.025 & 0.016\\
& 5 & 0.009 & 0.025 & 0.038 & 0.038 & 0.035 & 0.035 & 0.034 & 0.019 &
0.014\\\hline
\end{tabular}

\end{center}

{\small \noindent Notes: The first and second columns report the sample size
}$T${\small \ and the number of lags }$p${\small \ for the Box-Ljung test
statistics }${\small \hat{Q}}_{\tau}^{(p)},$ {\small respectively. The rest of
columns show empirical rejection frequencies based on bootstrap critical
values at the 5\% significance level. The tuning parameter }${\small \gamma}%
${\small \ is set to be 0.01.}\singlespacing\pagebreak

\begin{center}
\singlespacing Table 2. (power) Empirical rejection frequency of the Box-Ljung
test statistic $\hat{Q}_{\tau}^{(p)}$ based on the bootstrap procedure

(VAR\ with DGP2 (GARCH-X process))

%Table generated by Excel2LaTeX from sheet 'resultGamma01'%
\begin{tabular}
[c]{ccccccccccc}\hline
&  & \multicolumn{9}{c}{Quantiles $(\tau_{1}=\tau_{2})$}\\\hline
$T$ & $p$ & 0.05 & 0.10 & 0.20 & 0.30 & 0.50 & 0.70 & 0.80 & 0.90 &
0.95\\\hline\hline
500 & 1 & 0.361 & 0.701 & 0.722 & 0.383 & 0.042 & 0.383 & 0.713 & 0.684 &
0.362\\
& 2 & 0.303 & 0.610 & 0.584 & 0.257 & 0.063 & 0.231 & 0.589 & 0.589 & 0.300\\
& 3 & 0.270 & 0.541 & 0.491 & 0.202 & 0.053 & 0.174 & 0.467 & 0.515 & 0.246\\
& 4 & 0.230 & 0.451 & 0.403 & 0.172 & 0.058 & 0.126 & 0.378 & 0.447 & 0.208\\
& 5 & 0.203 & 0.393 & 0.344 & 0.134 & 0.060 & 0.115 & 0.314 & 0.386 &
0.177\\\hline
1000 & 1 & 0.751 & 0.948 & 0.942 & 0.638 & 0.048 & 0.619 & 0.951 & 0.952 &
0.760\\
& 2 & 0.708 & 0.916 & 0.912 & 0.425 & 0.046 & 0.431 & 0.908 & 0.932 & 0.712\\
& 3 & 0.651 & 0.877 & 0.845 & 0.322 & 0.052 & 0.315 & 0.849 & 0.897 & 0.651\\
& 4 & 0.589 & 0.838 & 0.784 & 0.255 & 0.048 & 0.250 & 0.778 & 0.854 & 0.596\\
& 5 & 0.537 & 0.801 & 0.716 & 0.203 & 0.042 & 0.190 & 0.714 & 0.809 &
0.563\\\hline
2000 & 1 & 0.969 & 0.999 & 0.999 & 0.905 & 0.044 & 0.923 & 0.999 & 0.998 &
0.974\\
& 2 & 0.965 & 1.000 & 0.999 & 0.808 & 0.053 & 0.817 & 0.999 & 1.000 & 0.979\\
& 3 & 0.959 & 1.000 & 0.997 & 0.688 & 0.053 & 0.673 & 0.998 & 1.000 & 0.967\\
& 4 & 0.944 & 1.000 & 0.990 & 0.585 & 0.047 & 0.573 & 0.994 & 0.999 & 0.957\\
& 5 & 0.930 & 1.000 & 0.982 & 0.510 & 0.037 & 0.485 & 0.987 & 0.997 &
0.938\\\hline
\end{tabular}

\end{center}

{\small \noindent Notes: Same as Table 1.}\singlespacing\pagebreak

\begin{center}
\singlespacing Table 3. Empirical Rejection Frequencies of the sup-version of
the Box-Ljung test statistic $\sup_{\tau\in\mathcal{T}}\hat{Q}_{\tau}^{(p)}$
based on the bootstrap procedure

(VAR with DGP1/DGP2 and the nominal level 5\%)

%Table generated by Excel2LaTeX from sheet 'sup'%
\begin{tabular}
[c]{cccccc}\hline
$T$ & $p$ &  & DGP1 (size) &  & DGP2 (power)\\\hline\hline
500 & 1 &  & 0.004 &  & 0.624\\
& 2 &  & 0.007 &  & 0.460\\
& 3 &  & 0.008 &  & 0.356\\
& 4 &  & 0.008 &  & 0.265\\
& 5 &  & 0.009 &  & 0.221\\\hline
1000 & 1 &  & 0.004 &  & 0.976\\
& 2 &  & 0.011 &  & 0.946\\
& 3 &  & 0.006 &  & 0.895\\
& 4 &  & 0.003 &  & 0.825\\
& 5 &  & 0.007 &  & 0.765\\\hline
2000 & 1 &  & 0.012 &  & 1.000\\
& 2 &  & 0.015 &  & 1.000\\
& 3 &  & 0.020 &  & 1.000\\
& 4 &  & 0.020 &  & 1.000\\
& 5 &  & 0.017 &  & 0.999\\\hline
\end{tabular}

\end{center}

{\small \noindent Notes: The first and second columns report the sample size
}$T${\small \ and the number of lags }$p${\small \ for the sup-version of the
Box-Ljung test statistic }$\sup_{\tau\in\mathcal{T}}{\small \hat{Q}}_{\tau
}^{(p)},$ {\small respectively. The sup-version test statistic is the
Box-Ljung test statistic maximized over nine quantiles }${\small \tau}%
_{i}{\small =0.05,0.1,0.2,0.3,0.5,0.7,0.8,0.9}${\small and }${\small 0.95}%
.${\small \ The third and fourth columns show empirical rejection frequencies
based on bootstrap critical values at the 5\% significance level. The tuning
parameter }${\small \gamma}${\small \ is set to be 0.01.}%
\singlespacing\pagebreak

\begin{center}
\singlespacing Table 4. (size) Empirical Rejection Frequencies of the
Self-Normalized Statistics

(VAR with DGP1 and the nominal level: 5\%)

%Table generated by Excel2LaTeX from sheet 'Table'%
%Table generated by Excel2LaTeX from sheet 'Sheet2'%
\begin{tabular}
[c]{ccccccccccc}\hline
&  & \multicolumn{9}{c}{Quantiles $(\tau_{1}=\tau_{2})$}\\\hline
$T$ & $p$ & 0.05 & 0.10 & 0.20 & 0.30 & 0.50 & 0.70 & 0.80 & 0.90 &
0.95\\\hline\hline
500 & 1 & 0.043 & 0.000 & 0.000 & 0.007 & 0.003 & 0.013 & 0.007 & 0.000 &
0.047\\
& 2 & 0.090 & 0.010 & 0.007 & 0.003 & 0.003 & 0.003 & 0.000 & 0.003 & 0.127\\
& 3 & 0.130 & 0.007 & 0.000 & 0.007 & 0.003 & 0.000 & 0.003 & 0.000 & 0.143\\
& 4 & 0.150 & 0.007 & 0.000 & 0.000 & 0.000 & 0.000 & 0.000 & 0.000 & 0.167\\
& 5 & 0.187 & 0.003 & 0.000 & 0.000 & 0.000 & 0.000 & 0.000 & 0.000 &
0.177\\\hline
1000 & 1 & 0.010 & 0.013 & 0.010 & 0.013 & 0.020 & 0.003 & 0.007 & 0.003 &
0.007\\
& 2 & 0.023 & 0.007 & 0.000 & 0.007 & 0.000 & 0.003 & 0.003 & 0.007 & 0.037\\
& 3 & 0.040 & 0.003 & 0.010 & 0.000 & 0.007 & 0.003 & 0.007 & 0.000 & 0.047\\
& 4 & 0.043 & 0.000 & 0.007 & 0.000 & 0.007 & 0.003 & 0.003 & 0.000 & 0.047\\
& 5 & 0.047 & 0.000 & 0.007 & 0.000 & 0.000 & 0.000 & 0.003 & 0.000 &
0.053\\\hline
2000 & 1 & 0.013 & 0.030 & 0.017 & 0.017 & 0.033 & 0.013 & 0.020 & 0.017 &
0.027\\
& 2 & 0.007 & 0.000 & 0.007 & 0.007 & 0.027 & 0.010 & 0.027 & 0.017 & 0.020\\
& 3 & 0.017 & 0.000 & 0.003 & 0.003 & 0.013 & 0.010 & 0.003 & 0.003 & 0.013\\
& 4 & 0.013 & 0.000 & 0.003 & 0.000 & 0.010 & 0.007 & 0.003 & 0.000 & 0.013\\
& 5 & 0.010 & 0.003 & 0.003 & 0.000 & 0.007 & 0.003 & 0.000 & 0.000 &
0.017\\\hline
\end{tabular}

\end{center}

{\small \noindent Notes: The first and second columns report the sample size
}$T${\small \ and the number of lags }$p${\small \ for the test statistics
}${\small \hat{Q}}_{\tau}^{(p)},$ {\small respectively. The rest of columns
show empirical rejection frequencies given simulated critical values at 5\%
significance level. The trimming value }${\small \omega}${\small \ is set to
be 0.1.}\singlespacing

\begin{center}
Table 5. (power) Empirical Rejection Frequencies of the Self-Normalized Statistics

(VAR with DGP2: GARCH-X process)

%Table generated by Excel2LaTeX from sheet 'Table'%
%Table generated by Excel2LaTeX from sheet 'Sheet2'%
\begin{tabular}
[c]{ccccccccccr}\hline
&  & \multicolumn{9}{c}{Quantiles $(\tau_{1}=\tau_{2})$}\\\hline
$T$ & $p$ & 0.05 & 0.10 & 0.20 & 0.30 & 0.50 & 0.70 & 0.80 & 0.90 &
0.95\\\hline\hline
500 & 1 & 0.067 & 0.230 & 0.297 & 0.077 & 0.007 & 0.150 & 0.300 & 0.253 &
0.050\\
& 2 & 0.030 & 0.070 & 0.113 & 0.033 & 0.000 & 0.037 & 0.113 & 0.077 & 0.010\\
& 3 & 0.047 & 0.010 & 0.043 & 0.010 & 0.000 & 0.017 & 0.023 & 0.020 & 0.023\\
& 4 & 0.063 & 0.007 & 0.023 & 0.000 & 0.000 & 0.010 & 0.013 & 0.003 & 0.050\\
& 5 & 0.120 & 0.003 & 0.007 & 0.003 & 0.000 & 0.003 & 0.003 & 0.000 &
0.080\\\hline
1000 & 1 & 0.347 & 0.643 & 0.683 & 0.313 & 0.010 & 0.323 & 0.673 & 0.663 &
0.317\\
& 2 & 0.153 & 0.523 & 0.527 & 0.177 & 0.020 & 0.180 & 0.543 & 0.463 & 0.157\\
& 3 & 0.063 & 0.300 & 0.347 & 0.090 & 0.010 & 0.097 & 0.377 & 0.283 & 0.063\\
& 4 & 0.033 & 0.210 & 0.223 & 0.050 & 0.000 & 0.037 & 0.243 & 0.153 & 0.017\\
& 5 & 0.047 & 0.097 & 0.133 & 0.030 & 0.000 & 0.023 & 0.127 & 0.097 &
0.020\\\hline
2000 & 1 & 0.757 & 0.917 & 0.923 & 0.663 & 0.030 & 0.693 & 0.940 & 0.920 &
0.707\\
& 2 & 0.577 & 0.873 & 0.917 & 0.513 & 0.013 & 0.540 & 0.883 & 0.863 & 0.577\\
& 3 & 0.427 & 0.787 & 0.860 & 0.400 & 0.007 & 0.397 & 0.800 & 0.810 & 0.390\\
& 4 & 0.270 & 0.680 & 0.807 & 0.323 & 0.017 & 0.297 & 0.740 & 0.680 & 0.250\\
& 5 & 0.197 & 0.567 & 0.700 & 0.223 & 0.003 & 0.213 & 0.680 & 0.590 &
0.163\\\hline
\end{tabular}

\end{center}

{\small \noindent Notes: Same as Table 4.}\singlespacing\pagebreak

\begin{center}%
%TCIMACRO{\FRAME{itbpF}{3.198in}{2.8028in}{0in}{}{}{sup_low.eps}%
%maintain-aspect-ratio TRUE;  display "PICT";  valid_file "F";  width 3.198in;
%height 2.8028in;  depth 0in;  original-width 3.027in;
%original-height 2.6492in;  cropleft "0";  croptop "1";  cropright "1";
%cropbottom "0";  filename 'Sup_low.eps';file-properties "XNPEU";}} }%
%BeginExpansion
{\includegraphics[
height=2.8028in,
width=3.198in
]%
{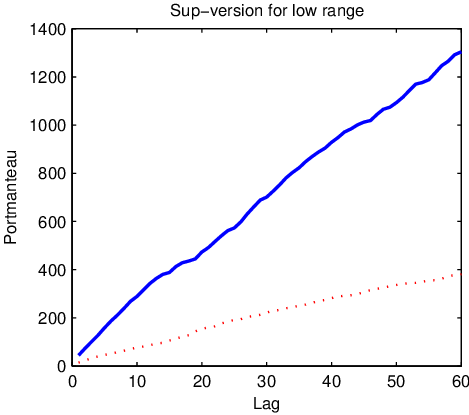}%
}
%EndExpansion
%TCIMACRO{\FRAME{itbpF}{3.198in}{2.8028in}{0in}{}{}{sup_high.eps}%
%maintain-aspect-ratio TRUE;  display "PICT";  valid_file "F";  width 3.198in;
%height 2.8028in;  depth 0in;  original-width 3.027in;
%original-height 2.6492in;  cropleft "0";  croptop "1";  cropright "1";
%cropbottom "0";  filename 'Sup_high.eps';file-properties "XNPEU";}} }%
%BeginExpansion
{\includegraphics[
height=2.8028in,
width=3.198in
]%
{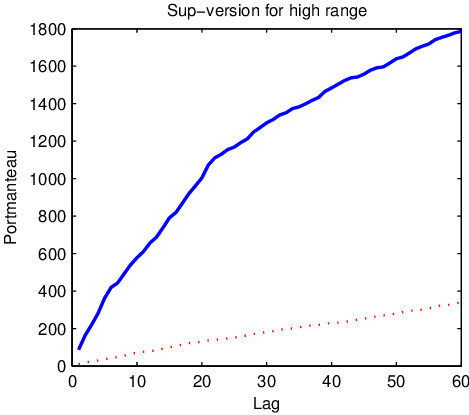}%
}
%EndExpansion

\end{center}

{\small \noindent Figure 1. Sup-version Box-Ljung test statistic }$\sup
_{\tau\in\mathcal{T}}\hat{Q}_{\tau}^{(p)}${\small \ for each lag }$p$
{\small to detect directional predictability from stock variance to stock
return}. {\small For the low range, we set\ }$\mathcal{T}{\small =[0.1,0.3]}%
${\small and }${\small \tau}_{i}{\small =0.1+0.02k}${\small for }%
${\small k=0,1,\ldots,10}.$ {\small We let }${\small \tau}_{1}{\small =\tau
}_{2}${\small for }$\hat{\rho}_{\tau}(k).${\small For the high range, we set
}$\mathcal{T}{\small =[0.7,0.9]}$ {\small and }${\small \tau}_{i}%
{\small =0.7+0.02k}$ {\small for }${\small k=0,1,\ldots,10}.$ {\small The
dashed lines are the 95\% bootstrap confidence intervals centred at the null
hypothesis.}\pagebreak

\begin{center}%
%TCIMACRO{\FRAME{itbpF}{6.4572in}{3.4406in}{0in}{}{}{cqr_10.eps}%
%maintain-aspect-ratio TRUE;  display "PICT";  valid_file "F";
%width 6.4572in;  height 3.4406in;  depth 0in;  original-width 6.553in;
%original-height 3.4779in;  cropleft "0";  croptop "1";  cropright "1";
%cropbottom "0";  filename 'CQR_10.eps';file-properties "XNPEU";}} }%
%BeginExpansion
{\includegraphics[
height=3.4406in,
width=6.4572in
]%
{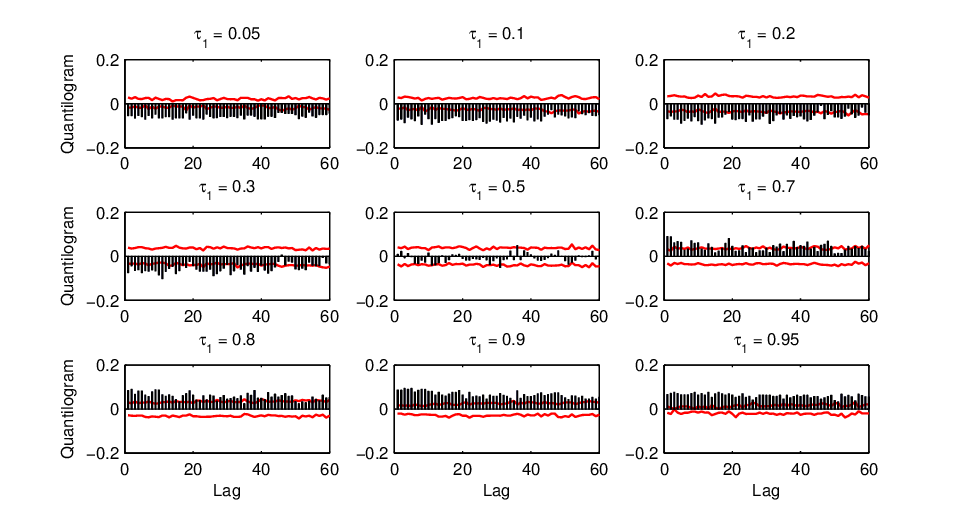}%
}
%EndExpansion

\end{center}

\noindent{\small Figure 2(a). The sample cross-quantilogram }$\hat{\rho}%
_{\tau}(k)${\small \ for }$\tau_{2}{\small =0.1}${\small \ to detect
directional predictability from stock variance to stock return. Bar graphs
describe sample cross-quantilograms and lines are the 95\% bootstrap
confidence intervals centred at zero.}

\begin{center}%
%TCIMACRO{\FRAME{itbpF}{6.4572in}{3.4406in}{0in}{}{}{qstat_10.eps}%
%maintain-aspect-ratio TRUE;  display "PICT";  valid_file "F";
%width 6.4572in;  height 3.4406in;  depth 0in;  original-width 6.553in;
%original-height 3.4779in;  cropleft "0";  croptop "1";  cropright "1";
%cropbottom "0";  filename 'Qstat_10.eps';file-properties "XNPEU";}} }%
%BeginExpansion
{\includegraphics[
height=3.4406in,
width=6.4572in
]%
{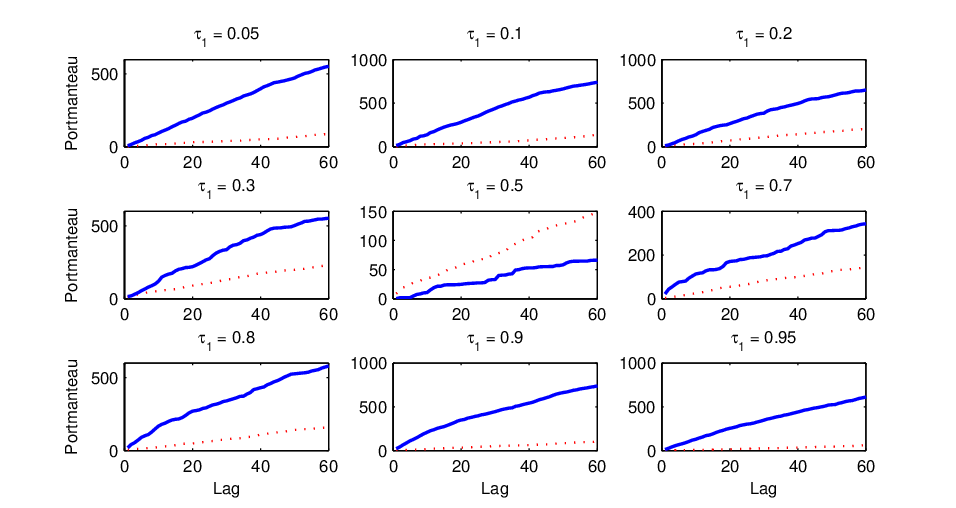}%
}
%EndExpansion

\end{center}

{\small \noindent Figure 2(b). Box-Ljung test statistic }$\hat{Q}_{\tau}%
^{(p)}${\small \ for each lag }$p${\small \ and quantile }$\tau$%
{\small \ using }$\hat{\rho}_{\tau}(k)${\small \ with }$\tau_{2}{\small =0.1}%
${\small . The dashed lines are the 95\% bootstrap confidence intervals
centred at zero.}\pagebreak

\begin{center}%
%TCIMACRO{\FRAME{itbpF}{6.4572in}{3.4406in}{0in}{}{}{cqr_90.eps}%
%maintain-aspect-ratio TRUE;  display "PICT";  valid_file "F";
%width 6.4572in;  height 3.4406in;  depth 0in;  original-width 6.553in;
%original-height 3.4779in;  cropleft "0";  croptop "1";  cropright "1";
%cropbottom "0";  filename 'CQR_90.eps';file-properties "XNPEU";}} }%
%BeginExpansion
{\includegraphics[
height=3.4406in,
width=6.4572in
]%
{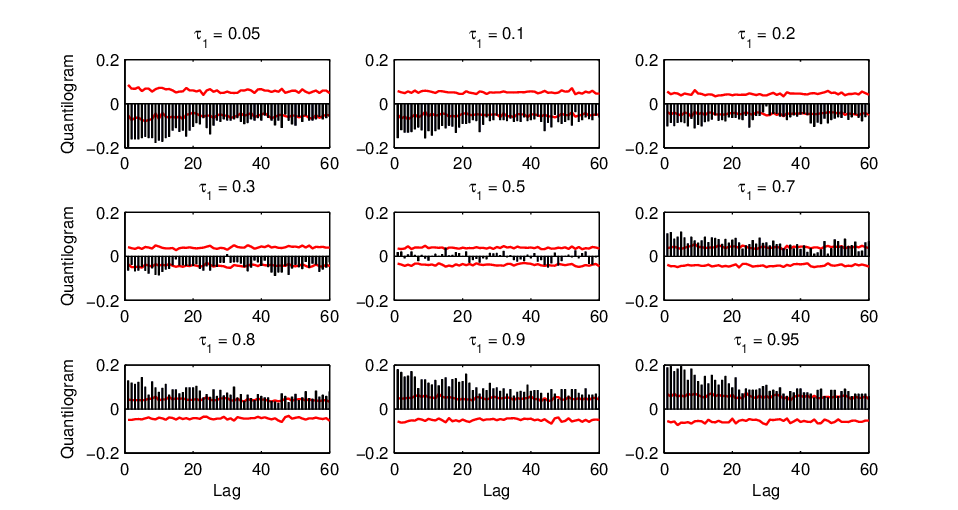}%
}
%EndExpansion

\end{center}

\noindent{\small Figure 3(a). The sample cross-quantilogram }$\hat{\rho}%
_{\tau}(k)${\small \ with }$\tau_{2}{\small =0.9}${\small \ to detect
directional predictability from stock variance to stock return. Same as Figure
1(a). \ }

\begin{center}%
%TCIMACRO{\FRAME{itbpF}{6.4572in}{3.4406in}{0in}{}{}{qstat_90.eps}%
%maintain-aspect-ratio TRUE;  display "PICT";  valid_file "F";
%width 6.4572in;  height 3.4406in;  depth 0in;  original-width 6.553in;
%original-height 3.4779in;  cropleft "0";  croptop "1";  cropright "1";
%cropbottom "0";  filename 'Qstat_90.eps';file-properties "XNPEU";}} }%
%BeginExpansion
{\includegraphics[
height=3.4406in,
width=6.4572in
]%
{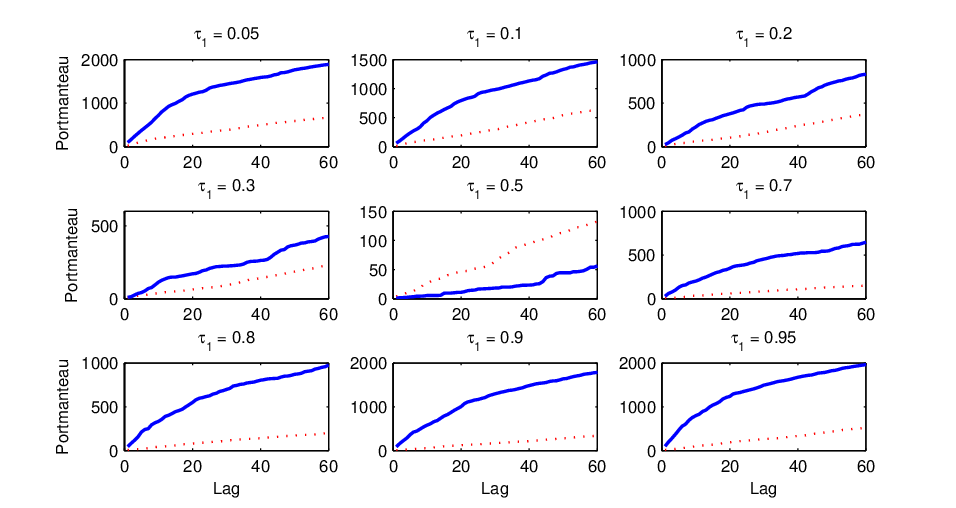}%
}
%EndExpansion

\end{center}

\noindent{\small Figure 3(b). Box-Ljung test statistic }$\hat{Q}_{\tau}^{(p)}%
${\small \ for each lag }$p${\small \ and quantile }$\tau${\small \ using
}$\hat{\rho}_{\tau}(k)${\small \ with }$\tau_{2}{\small =0.9}$.{\small \ Same
as Figure 1(b).}\pagebreak

\begin{center}
\singlespacing
%

%TCIMACRO{\FRAME{itbpF}{5.6903in}{6.4657in}{0in}{}{}{crossq.eps}%
%maintain-aspect-ratio TRUE;  display "PICT";  valid_file "F";
%width 5.6903in;  height 6.4657in;  depth 0in;  original-width 5.6612in;
%original-height 6.4383in;  cropleft "0";  croptop "1";  cropright "1";
%cropbottom "0";  filename 'crossQ.eps';file-properties "XNPEU";}} }%
%BeginExpansion
{\includegraphics[
height=6.4657in,
width=5.6903in
]%
{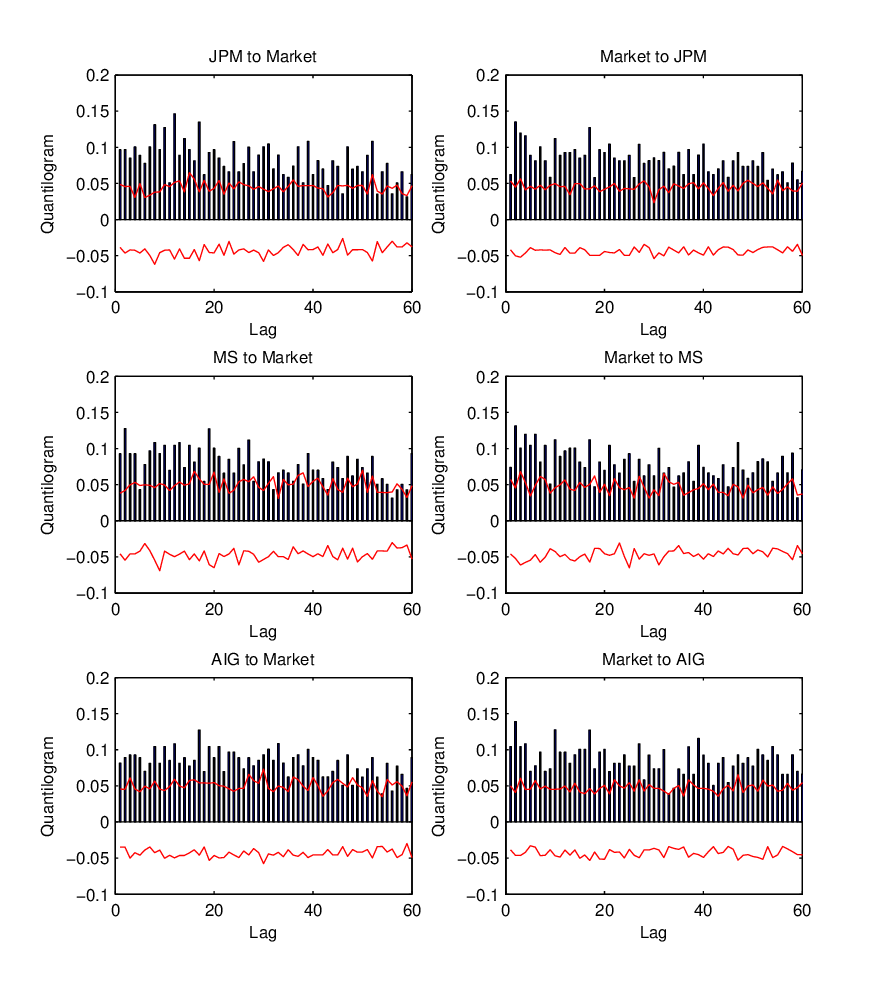}%
}
%EndExpansion

\end{center}

{\small \noindent Figure 4. The sample cross-quantilogram }$\hat{\rho}_{\tau
}(k)${\small . Bar graphs describe sample cross-quantilograms and lines are
the 95\% bootstrap confidence intervals centred at zero.}%
\singlespacing\pagebreak

\begin{center}
\singlespacing
%

%TCIMACRO{\FRAME{itbpF}{5.6903in}{6.4657in}{0in}{}{}{partialq.eps}%
%maintain-aspect-ratio TRUE;  display "PICT";  valid_file "F";
%width 5.6903in;  height 6.4657in;  depth 0in;  original-width 5.6612in;
%original-height 6.4383in;  cropleft "0";  croptop "1";  cropright "1";
%cropbottom "0";  filename 'PartialQ.eps';file-properties "XNPEU";}} }%
%BeginExpansion
{\includegraphics[
height=6.4657in,
width=5.6903in
]%
{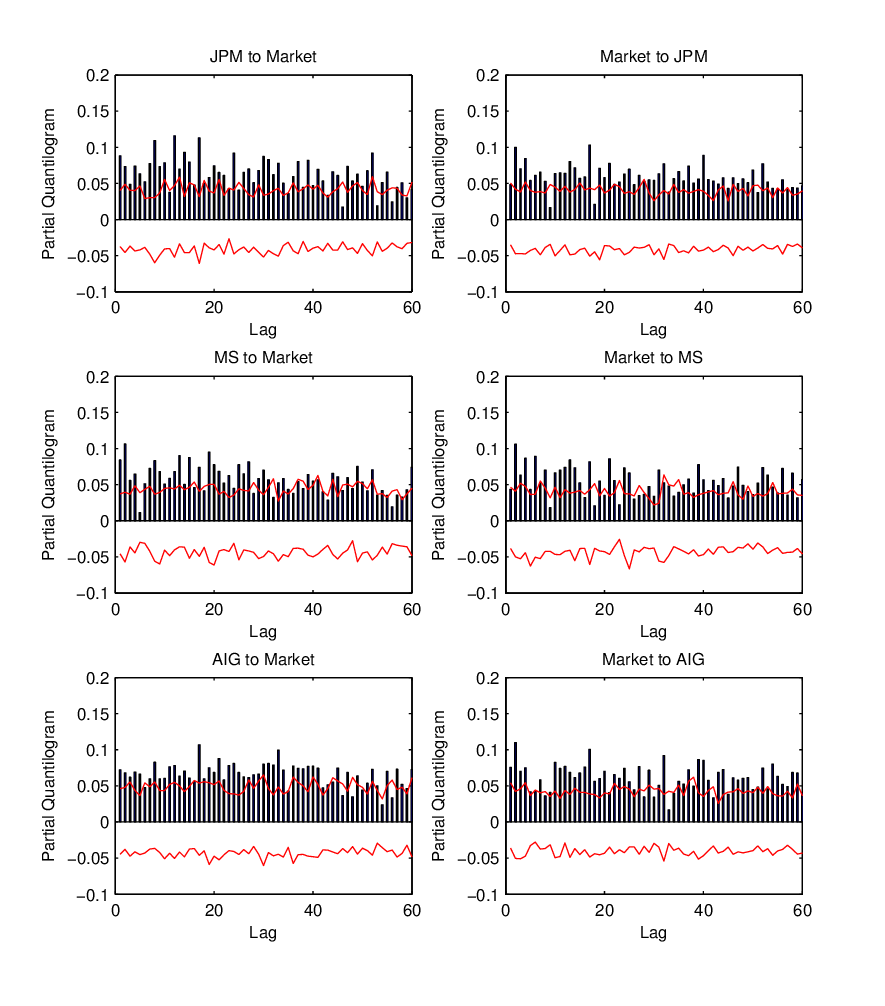}%
}
%EndExpansion

\end{center}

{\small \noindent Figure 5. The sample partial cross-quantilogram }$\hat{\rho
}_{\bar{\tau}|\mathbf{z}}(k)${\small . Bar graphs describe sample\ partial
cross-quantilograms and lines are the 95\% bootstrap confidence intervals
centred at zero.}\singlespacing\pagebreak

\doublespacing{\small \noindent }

\setstretch{1.00}%

%TCIMACRO{\TeXButton{Start Bib}{\begin{thebibliography}{20}}}%
%BeginExpansion
%
%EndExpansion

\end{document}